\documentclass[final,onefignum,onetabnum]{siamart220329}

\usepackage{braket}
\usepackage{array}
\usepackage[caption=false]{subfig}
\usepackage{pgfplots}
\usepackage{amsfonts}
\usepackage[utf8]{inputenc}
\usepackage[colorinlistoftodos]{todonotes}
\usepackage[T1, T2A]{fontenc}
\usepackage[english]{babel}
\usetikzlibrary{shapes.geometric}
\usepackage{rotating}
\usepackage{esvect}
\usepackage{float}
\usepackage{graphicx}
\usepackage{caption}
\usepackage{subcaption}
\usepackage{amssymb}
\usepackage{tikz}
\usepackage{pgf,tikz,tkz-base,tkz-euclide}
\usetikzlibrary{calc}
\usepackage{relsize}
\usepackage{mathrsfs}
\usetikzlibrary{arrows}
\usepackage{soul}
\usepackage[letterpaper,top=2cm,bottom=2cm,left=3cm,right=3cm,marginparwidth=1.75cm]{geometry}
\usepackage{xcolor}
\usepackage{hyperref}
\usepackage{wrapfig}

\usetikzlibrary{fadings}

\definecolor{airforceblue}{rgb}{0.36, 0.54, 0.66}
\definecolor{antiquebrass}{rgb}{0.8, 0.58, 0.46}
\definecolor{applegreen}{rgb}{0.55, 0.71, 0.0}
\definecolor{amber(sae/ece)}{rgb}{1.0, 0.49, 0.0}
\definecolor{arsenic}{rgb}{0.23, 0.27, 0.29}
\definecolor{armygreen}{rgb}{0.29, 0.33, 0.13}
\definecolor{ashgray}{rgb}{0.7, 0.75, 0.71}
\definecolor{auburn}{rgb}{0.43, 0.21, 0.1}
\definecolor{aurometalsaurus}{rgb}{0.43, 0.5, 0.5}
\definecolor{babyblue}{rgb}{0.54, 0.81, 0.94}
\definecolor{azure(web)(azuremist)}{rgb}{0.94, 1.0, 1.0}
\definecolor{babyblue}{rgb}{0.54, 0.81, 0.94}
\definecolor{bole}{rgb}{0.47, 0.27, 0.23}
\definecolor{brightgreen}{rgb}{0.4, 1.0, 0.0}
\definecolor{britishracinggreen}{rgb}{0.0, 0.26, 0.15}
\definecolor{bulgarianrose}{rgb}{0.28, 0.02, 0.03}
\definecolor{cadmiumyellow}{rgb}{1.0, 0.96, 0.0}
\definecolor{carolinablue}{rgb}{0.6, 0.73, 0.89}
\definecolor{chartreuse(traditional)}{rgb}{0.87, 1.0, 0.0}
\definecolor{chartreuse(web)}{rgb}{0.5, 1.0, 0.0}
\definecolor{cinereous}{rgb}{0.6, 0.51, 0.48}
\definecolor{coolblack}{rgb}{0.0, 0.18, 0.39}
\definecolor{deepskyblue}{rgb}{0.0, 0.75, 1.0}
\definecolor{darkpastelblue}{rgb}{0.47, 0.62, 0.8}
\definecolor{ceil}{rgb}{0.57, 0.63, 0.81}
\definecolor{cadetgray}{rgb}{0.57, 0.64, 0.69}
\definecolor{bluedefrance}{rgb}{0.19, 0.55, 0.91}

\newsiamthm{notation}{Notation}
\newsiamthm{example}{Example}
\newsiamthm{claim}{Claim}
\newsiamremark{remark}{Remark}
\newsiamremark{hypothesis}{Hypothesis}
\crefname{hypothesis}{Hypothesis}{Hypotheses}

\newcommand*{\Scale}[2][4]{\scalebox{#1}{$#2$}}%
\newcommand{\RR}{\mathbb{R}}

\usepackage{algorithmic}

\Crefname{ALC@unique}{Line}{Lines}

\usepackage{amsopn}

\usepackage{xspace}
\usepackage{bold-extra}
\usepackage[most]{tcolorbox}

\colorlet{texcscolor}{blue!50!black}
\colorlet{texemcolor}{red!70!black}
\colorlet{texpreamble}{red!70!black}
\colorlet{codebackground}{black!25!white!25}

\newcommand{\HP}[1]{{\color{blue}HP: #1}}

\pgfplotsset{compat=1.18} % Set the compatibility mode to 1.18

\newcommand{\hiddencomment}[1]{}
\lstdefinestyle{siamlatex}{%
  style=tcblatex,
  texcsstyle=*\color{texcscolor},
  texcsstyle=[2]\color{texemcolor},
  keywordstyle=[2]\color{texemcolor},
  moretexcs={cref,Cref,maketitle,mathcal,text,headers,email,url},
}

\tcbset{%
  colframe=black!75!white!75,
  coltitle=white,
  colback=codebackground, % bottom/left side
  colbacklower=white, % top/right side
  fonttitle=\bfseries,
  arc=0pt,outer arc=0pt,
  top=1pt,bottom=1pt,left=1mm,right=1mm,middle=1mm,boxsep=1mm,
  leftrule=0.3mm,rightrule=0.3mm,toprule=0.3mm,bottomrule=0.3mm,
  listing options={style=siamlatex}
}

\DeclareTotalTCBox{\code}{ v O{} }
{ %fontupper=\ttfamily\color{texemcolor},
  fontupper=\ttfamily\color{black},
  nobeforeafter,
  tcbox raise base,
  colback=codebackground,colframe=white,
  top=0pt,bottom=0pt,left=0mm,right=0mm,
  leftrule=0pt,rightrule=0pt,toprule=0mm,bottomrule=0mm,
  boxsep=0.5mm,
  #2}{#1}

\patchcmd\newpage{\vfil}{}{}{}
\begin{tcbverbatimwrite}{tmp_\jobname_header.tex}
\title{Area preserving Combescure transformations\thanks{\vspace{-0.3cm}%20.02.2024 %Submitted to the editors DATE.
\funding{This research has been supported by KAUST baseline funding (grant BAS/1/1679-01-01).}}}

\author{O. Pirahmad\thanks{CEMSE, King Abdullah University of Science and Technology (\email{pirahmad.olimjoni@kaust.edu.sa, helmut.pottmann@gmail.com, mikhail.skopenkov@gmail.com}).} \and H. Pottmann\footnotemark[2]
\and M. Skopenkov\footnotemark[2]}

\headers{Area preserving Combescure transformations}{O. Pirahmad, H. Pottmann, and M. Skopenkov}
\end{tcbverbatimwrite}
\input{tmp_\jobname_header.tex}

\ifpdf
\hypersetup{ pdftitle={Area preserving Combescure transformations} }
\fi

\begin{document}
\maketitle

%% ------------------------------------------------------------------
%% ABSTRACT
%% ------------------------------------------------------------------
\begin{tcbverbatimwrite}{tmp_\jobname_abstract.tex}
\begin{abstract}
Motivated by the design of flexible nets, we classify all nets of arbitrary size $m\times n$
that admit a continuous family of area-preserving Combescure transformations. There are just two different classes. The nets in the first class are special cases of cone nets that
have been recently studied by Kilian, M\"uller, and Tervooren. The second class consists of K{\oe}nigs nets having a Christoffel dual with the same areas of corresponding faces. We apply isotropic metric duality to get a new class of flexible nets in isotropic geometry. We also study the smooth analogs of the introduced classes. %The proof of the classification uses the geometry of pencils of quadrics.
%\HP{TO BE REWRITTEN LATER.} We study deformable nets in $\mathbb R^3$. A net is deformable if it is contained in a family of parallel nets with the same areas of corresponding faces. In other words, these nets are members of a family of nets generated by area-preserving Combescure transformations. We classify all deformable $m\times n$ nets with convex faces. We show that there are just two classes of deformable $m\times n$ nets. %, in contrast to the rather involved I.~Izmestiev's classification in Euclidean geometry. The duals are special discrete K{\oe}nigs nets. A study of further properties is currently under investigation. We have extended our results from $2\times2$ nets to larger sizes of nets. 
    %The proof of the classification uses elementary algebraic geometry involving pencils of quadrics. 
    %We also provide an algebraic proof for the classification of $2\times2$ deformable nets. The latter proof confirms our main result for deformable $2\times2$ nets without restriction to convex faces.
\end{abstract}


\end{tcbverbatimwrite}
\input{tmp_\jobname_abstract.tex}
%% ------------------------------------------------------------------
%% END HEADER
%% ------------------------------------------------------------------

%\section*{General questions:}
%\begin{enumerate}
    %\item Should we replace ``mesh'' by ``net'' throughout?
    %\item Can we find a more intuitive notation/names for classes~(i) and~(ii)?
    %\item Fix spelling of ``K\oe nigs'' throughout. {\color{green}Done!}
%\end{enumerate}

\section{Introduction}
\label{sec:intro}

The motivation for our paper originates from the search
for flexible quadrilateral meshes. Considering the faces as
rigid bodies and the edges where two faces meet as hinges, such meshes
allow for a continuous flexion. In other words, they are mechanisms. 
Very recently, there has been growing interest in this area, with applications
in rigid origami and transformable design \cite{aikyn2024flexible,chen2015,dang-feng,feng-dang,evans2015rigidly,lewis2016,Nawratil2022,t-nets-IASS,Tac09a}. In our paper, we suggest a new approach to the problem. 

We restrict to quad meshes of regular combinatorics and call them (discrete) \emph{nets} in the following. The flexibility is interesting to study starting from nets with $3 \times 3$ faces.
% It is well-known (Schief et al.~\cite{Schief2008}) that a net with $m \times n$ faces is flexible if and only if all its $3 \times 3$ sub-nets are. 
The flexible $3 \times 3$ Q-nets (meshes with planar faces) have been classified in a seminal 
paper by Izmestiev  \cite{izmestiev-2007}, which revealed a big
variety of types. Well-known examples
of flexible Q-nets of arbitrary size $m \times n$ are the Voss nets and T-nets, treated in detail by R.~Sauer \cite{sauer:1970}.
Very recently, Nawratil \cite{Nawratil2024} presented another special class of flexible Q-nets which he calls P-nets. They may be seen as a generalization
of rotational nets, can be easily designed, but cannot provide a large variety of shapes. 
Further important recent progress is due to He and Guest \cite{he2020rigid} who have been able to patch other types of $3 \times 3$ nets together
to larger flexible Q-nets. However, this is not yet a method for the design of flexible
structures, since the result is hardly predictable by the provided generation process. 

%\mscomm{MS: Do you think we could say something about optimization already here? To emphasize that we have a new approach to a complicated problem, not just replacing it with a simpler problem.}

In view of the difficulties which arise in Euclidean geometry,
we simplify the problem and turn our interest to {isotropic geometry}, which may be seen
as a \emph{structure-preserving degeneration} %simplified version 
of Euclidean one. %geometry. 
We hope that the examples found in isotropic geometry %the solution of the simplified problem 
can then be used as a good initialization to construct Euclidean ones %solve the original problem 
via numerical optimization. This approach has been (implicitly) used since as early as the work \cite{Muntz} by M\"untz from 1911,  who solved the Plateau problem for Euclidean minimal surfaces in a quite general setup by deformation of isotropic minimal surfaces (graphs of harmonic functions). Another successful example of a structure-preserving degeneration is tropical~geometry.

The geometry
in \emph{isotropic 3-space} $I^3$ is based on
a 6-parametric group of affine transformations in $\RR^3$ which preserve the isotropic semi-norm  $\|(x,y,z)\|_i:=\sqrt{x^2+y^2}$. This geometry has been 
systematically developed by K.~Strubecker \cite{strubecker1,strubecker2,strubecker3}. A detailed treatment is found in the monograph by H.~Sachs \cite{Sachs:1990}. 
The geometry in $I^3$ is not as degenerate as it may appear by just looking at the
metric. One can come up with so-called replacing invariants and obtains beautiful
counterparts to Euclidean results. For example,
one finds a definition of isotropic Gaussian curvature $K_i$  of a surface that has properties
very similar to the familiar Euclidean counterpart.  However, pursuing the Riemannian approach based on the isotropic metric, $K_i$
would vanish everywhere.   Isotropic curvature theory of surfaces
defines a
counterpart to the Euclidean shape operator via an isotropic Gauss map that maps
a surface via parallel tangent planes to the parabolic 
isotropic unit sphere $S_i^2: 2z=x^2+y^2$. If the surface
is %represented as 
a function graph  $z=f(x,y)$,  then $K_i$ is the determinant of the Hessian, i.e., % of $f$,  
$K_i=f_{xx}f_{yy}-f_{xy}^2$.

Since flexible nets may be seen as discrete versions of continuous isometric deformations of a surface, one needs a proper definition of isometric surfaces in $I^3$,
which has been missing until very recently. Requiring the preservation of the lengths of surface curves \emph{and} of
isotropic Gaussian curvature during an isometric deformation, one obtains
the desired non-degenerate analogue to the Euclidean case \cite{isometric-isotropic}.
Unlike Euclidean geometry, isotropic geometry possesses a metric duality,
realized via the polarity $\delta$ with respect to $S_i^2$. It has been shown \cite{isometric-isotropic} that $\delta$ transforms a pair of isometric surfaces
in $I^3$ to a pair of surfaces which are related by an \emph{area-preserving Combescure
transformation}. In the smooth setting, two surfaces $f(u,v)$ and $f^+(u,v)$
are related by a Combescure transformation (C-trafo), if corresponding tangent
vectors $f_u,f^+_u$ and $f_v,f^+_v$ of parameter lines are parallel. 

%\begin{figure}[htbp]
%    \centering
%    \includegraphics[scale=0.25]{fig/e1.png}%\quad\includegraphics[scale=0.27]{fig/e2.png}\quad\includegraphics[scale=0.27]{fig/e3.png}
%    \quad\includegraphics[scale=0.27]{fig/e4.png}%\quad\includegraphics[scale=0.27]{fig/e5.png}
%    \quad\includegraphics[scale=0.27]{fig/e6.png}%\quad\includegraphics[scale=0.27]{fig/e7.png}
%    \quad\includegraphics[scale=0.27]{fig/e8.png}
%    \caption{A sequence of deformations of a net from class (i) in $\mathbb R^3$. Corresponding edges are parallel and corresponding faces have equal areas. For a net from class~(i), the two faces of each $1\times2$ sub-net or each $2\times1$ sub-net are affine symmetric. See Theorem~\ref{th-mxn}}
%    \label{fig:type i}
%\end{figure}

\begin{figure}[!t]%[thbp]
    \centering
    \includegraphics[scale=0.04]%[scale=0.045]    
{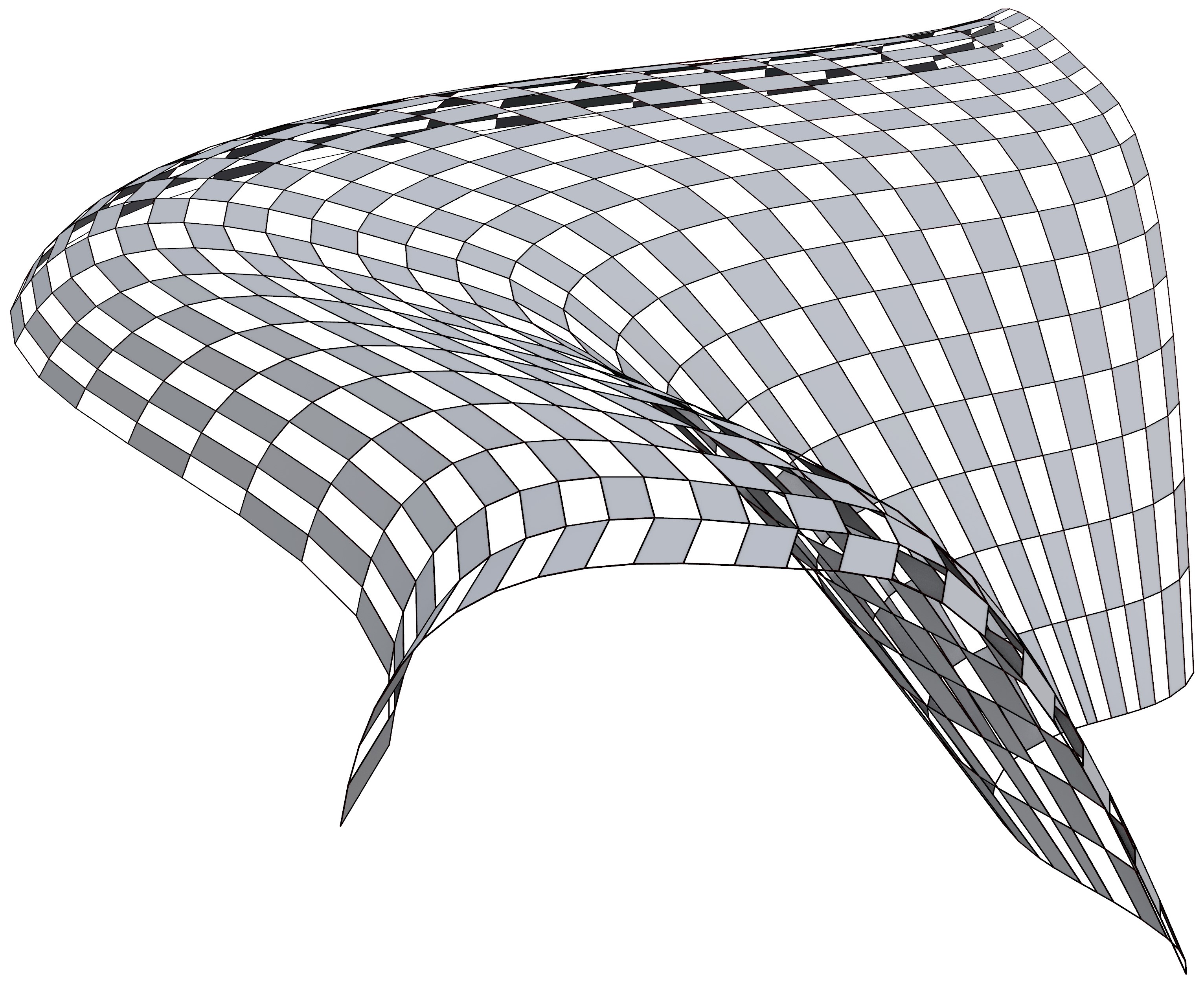}\qquad\qquad\qquad\includegraphics[scale=0.04]%[scale=0.045]
{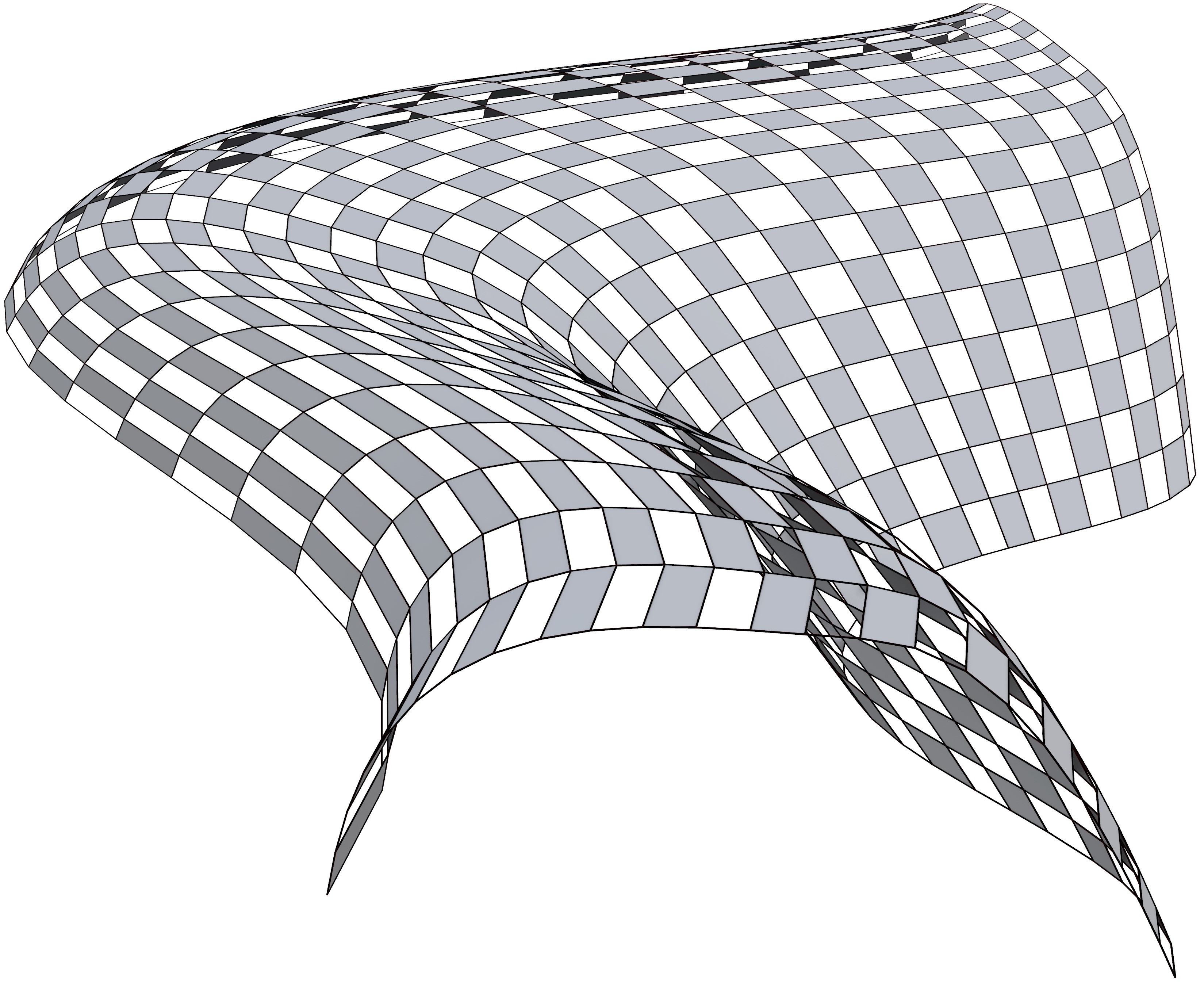}
    \quad\quad\quad\includegraphics[scale=0.04]%[scale=0.045]
    {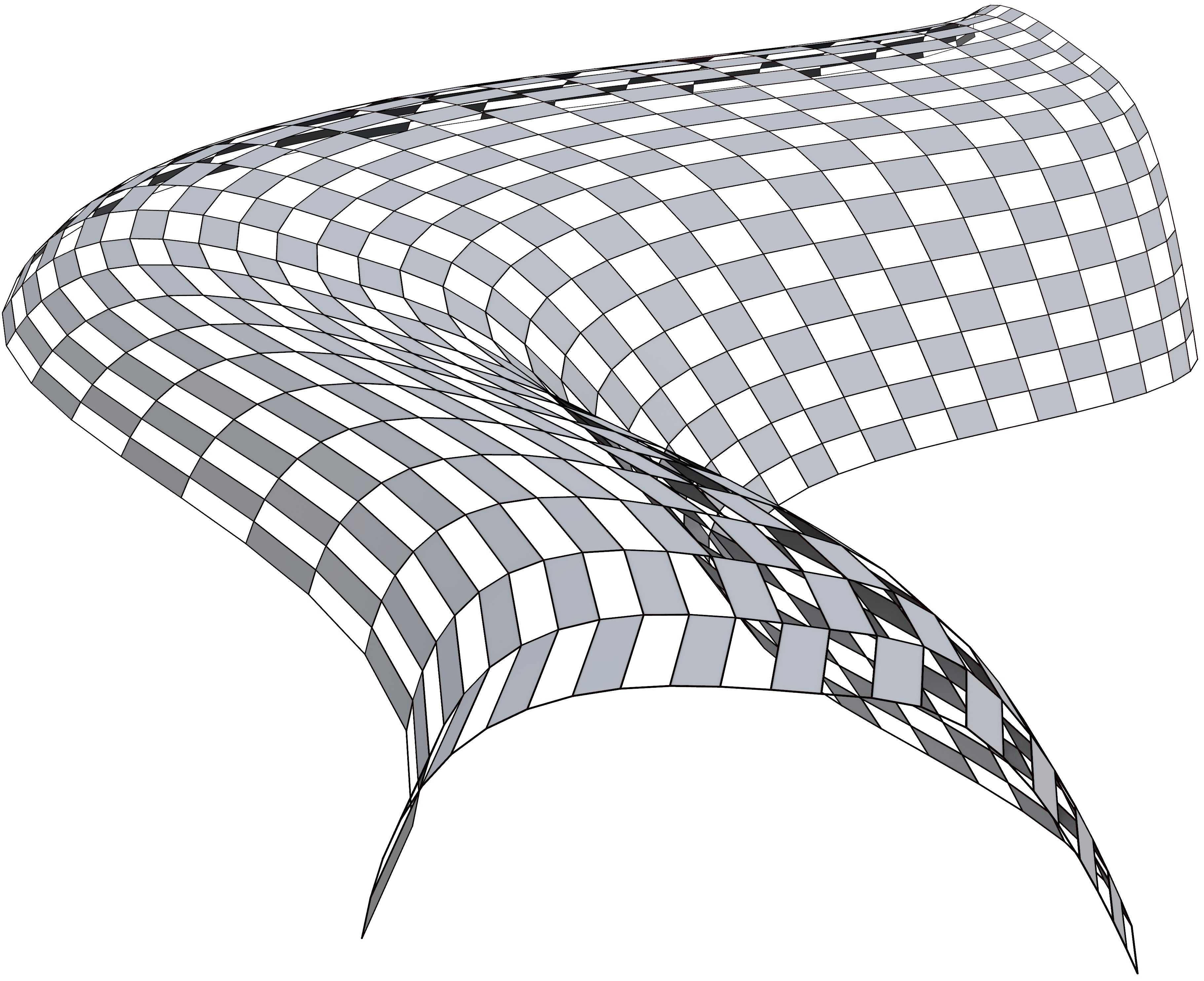}
    \qquad\qquad\qquad\includegraphics[scale=0.036]%[scale=0.04]
    {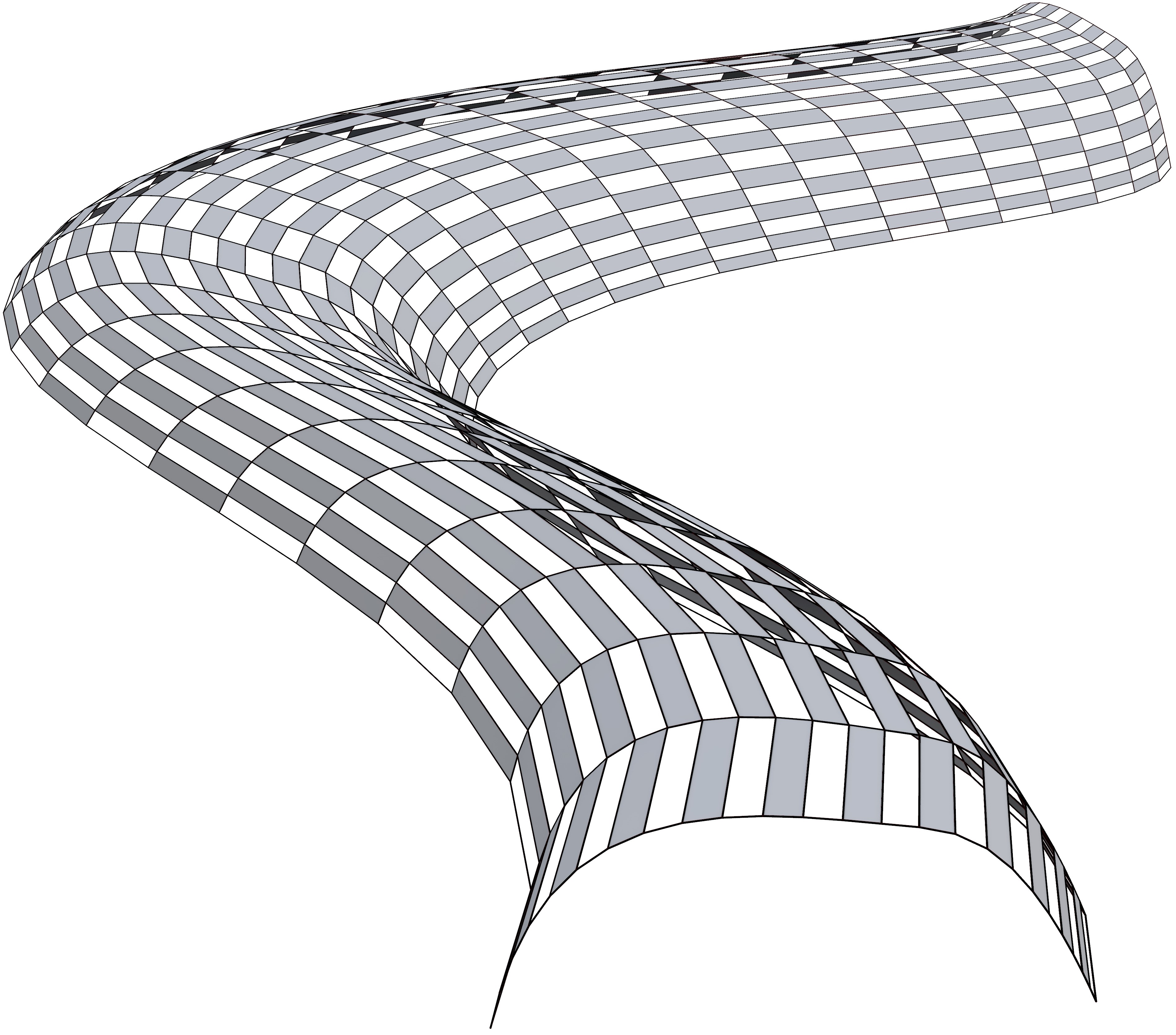}
    \caption{A sequence of deformations of a Q-net from class (i) in $\mathbb R^3$. Corresponding edges are parallel and corresponding faces have equal areas. For a net from class~(i), the two faces of each $1\times2$ sub-net or each $2\times1$ sub-net are affine symmetric. See Theorem~\ref{th-mxn}}
    \label{fig:type i}
\end{figure}

C-trafos play a fundamental role in discrete differential geometry \cite{bobenko-2008-ddg}. Two Q-nets are related by a C-trafo if corresponding edges are parallel.
Dual to the search for flexible Q-nets in $I^3$,
we have  to find all
\emph{Q-nets which allow for a continuous family of area preserving C-trafos}.
In the following, we denote such Q-nets as \emph{deformable}. See Figure~\ref{fig:type i}.

 This problem is one of affine geometry in $\RR^2$. Indeed, if two Q-nets are related by a C-trafo, corresponding faces are parallel, and we require that they have the same %oriented 
 area. This implies immediately the following basic fact. If two Q-nets in %$F_1, F_2 \subset 
$\RR^d$, where $d \ge 3$, are related by an area-preserving C-trafo, the same is true for their images under 
any parallel projection into a plane. %Thus, our problem is one of affine geometry in $\RR^2$. 
Conversely, %given 
any area preserving C-trafo in the plane
%, we can generate 
can be lifted to $\RR^d$ in infinitely many essentially different ways. %examples in $\RR^d$. 
We just add $d-2$ coordinates to the vertices of two transversal discrete parameter lines of one Q-net in the plane and
complete the lifted nets via parallelism. 

After having classified all deformable Q-nets in $I^3$, we will apply
metric duality and obtain all flexible Q-nets in $I^3$. Then, our goal is to transform these nets with a numerical
optimization algorithm to Euclidean flexible Q-nets. The latter
is, however, subject of a future publication.

\subsection{Pointers to the classical literature}
Area preserving maps between surfaces are a well studied
subject of classical differential geometry and cartography. In the
latter case, the focus is on maps from the sphere to the plane.

Already the seemingly simple cases of area preserving maps on the sphere $S^2$ or in the
Euclidean plane $\mathbb R^2$ possess a remarkable theory that reaches into non-Euclidean geometries. G.~Fubini \cite{fubini} found all area-preserving maps on $S^2$ via
the so-called kinematic mapping of oriented lines $L$ in elliptic 3-space $E^3$ to pairs $(L_l,L_r)$ of points (left and right image) on $S^2$.
For details on this map, which uses Clifford parallelism in $E^3$, we refer
to \cite{hrmueller-kin,pottwall:2001}. The normal lines
of a surface $\Phi \subset E^3$ get mapped to a left and right image
domain $\Phi_l$ and $\Phi_r$, respectively, which are related by an
area preserving map. All smooth area-preserving maps on $S^2$ can be generated in
this way. 

Only after Fubini's result, a similar construction of area preserving
maps in the plane has been discovered. Based on initial work by
G.~Scheffers \cite{scheffers}, K.~Strubecker \cite{strubecker4} developed the
theory of the so-called \emph{paratactic map} of contact elements $C$ in  isotropic 3-space $I^3$ to pairs of points $(C_l, C_r)$ in the plane.
A contact element is defined as a point and an incident plane. 
The paratactic map is one application of the geometry in $I^3$.  
It possesses limits of Clifford
translations which generate the paratactic map. Remarkably, if the
contact elements are formed by the points plus tangent planes of a surface $\Phi \subset I^3$, the left and right images in the plane
are related by an area preserving map. A special property related to our work is the following: Mapping the asymptotic net
of a negatively curved surface $\Phi$ results in two nets related
by a C-trafo. Hence, a negatively curved surface $\Phi$ generates an area preserving C-trafo $\Phi_l \mapsto \Phi_r$ between two nets in $\RR^2$. Note that we are searching for a continuous family of nets
related by area preserving C-trafos. 

Another relevant %classical 
result is %the famous
Minkowski's theorem on the existence and uniqueness of a convex polytope with given directions and areas of faces. The polytope surface can be viewed as a mesh (with combinatorics more complex than a square grid), and then the theorem implies that the mesh admits no non-congruent area-preserving C-trafos. %besides  translations or central symmetries. 
Via the metric duality $\delta$, this translates to \emph{isotropic %analog of the famous 
Cauchy's %rigidity 
theorem}: {a mesh in $I^3$ whose metric dual is the surface of a convex polytope is not flexible.} %Cf.~the degenerate Cauchy theorem~\cite[Lemma~3]{Mor20}.
%\mscomm{Can we describe meshes whose metric dual is the surface of a convex polytope in a nice way?}

\subsection{Contributions and overview}
Our contributions in this paper are as follows.

In Section~\ref{sec-deformable-dual-nets}, we provide a 
classification of all deformable nets with $2 \times 2$ 
faces ($ 3 \times 3$ vertices), since an $m \times n$ net is deformable if all
$2 \times 2$ sub-nets are, as we show later. This is in agreement
with the fact that in $I^3$ (as in Euclidean $\RR^3$) flexible Q-nets require
flexible sub-nets of $3  \times 3$ faces, where 
only the 4 vertices of the inner face matter. 
Remarkably, \emph{there are just two different classes}. See Theorem~\ref{th-classification}. The main part of the section presents
the proof of the classification. It amounts to discussing
a certain system of quadratic equations and identifying
the cases where it has a one-parametric real solution.
To keep our arguments elementary, we avoid tools such as complex projective space and resultants. %; using them would not give any advantage in our setup anyway. 

%\HP{move to section 2?} We also address infinitesimal deformability.  A net allows for an infinitesimal area-preserving C-trafo if and only if it is a K{\oe}nigs net. The K{\oe}nigs-dual net can be seen as a velocity diagram (see also \cite{isometric-isotropic}). 

In Section~\ref{sec-deformable-mxn}, we classify deformable $m \times n$ nets and show how to build them from boundary data. See Theorem~\ref{th-mxn}. All $2 \times 2$ sub-nets turn out to be of the same class. For class (i), the resulting
nets are special cases of cone nets that have been recently studied by Kilian, M\"uller, and Tervooren \cite{kilian2023smooth}. Such nets also appeared in free-form architectural glass structures with planar quadrilateral (in fact trapezoidal) faces \cite{glymphetal}. For class (ii), the resulting
nets are special cases of well-known K\oe nigs nets \cite{bobenko-2008-ddg}. The class consists of nets having a Christoffel dual with the same areas of corresponding faces. While class (i) can yield visually pleasing discretizations of
smooth deformable nets (see Figure~\ref{fig:type i}), class (ii) almost always exhibits
the opposite behavior and a crumpled appearance (see Figures~\ref{fig:type ii more} and~\ref{fig:type ii}). %\HP{add pointers to appropriate figures} \mscomm{MS: Done.}

This fact has a large impact  when turning to smooth analogs in 
Section~\ref{sec-smooth-def}. The first class (i) yields smooth double cone
nets where one family of cones are cylinders. The surfaces are
generalizations of translational surfaces, called
scale-translational surfaces in architecture \cite{glymphetal}. Ordinary translational surfaces appear as the
only smooth analogs of the second class (ii). They constitute
a special case of the smooth analogs of (i). Their
metric duals in $I^3$ are the
isotropic counterparts of Voss nets. 

In Section~\ref{sec:final}, we conclude the paper with a few remarks on the dual flexible Q-nets
whose detailed study and possibly conversion to Euclidean flexible
Q-nets through optimization shall form the content of a separate
publication. The nets of class (i) are instances of 
or closely related to
multi-conjugate nets in the sense of Bobenko et al. \cite{multinets-2018}. They are also
 of interest in architectural applications as discussed in  \cite{PP-2022,Cheng-CAD-2023}. Those
of class (ii) appear crumpled, which is a frequent effect in Euclidean flexible Q-nets such as
Miura origami and generalizations (see e.g. \cite{he2020rigid}).

%Throughout we denote by $S(P)$ the area of a polygon $P$. \mscomm{Or maybe replace $S(P)$ by $\mathrm{Area}(P)$. Maybe we collect the other important notations here.}

%Comparing our result and Minkowski's theorem in~\cite{gu2013variational} which states the existence and uniqueness(up to parallel translation) of convex polytope, we conclude that there is no deformable net in $\mathbb{R}^3$ that forms a convex polytope. Applications.

%\subsubsection*{K\"onigs nets}
%\mscomm{Plan: K\"onigs nets. Reciprocal nets. etc}

\begin{figure}[!t]%[htbp]
    \centering
    \includegraphics[scale=0.22]{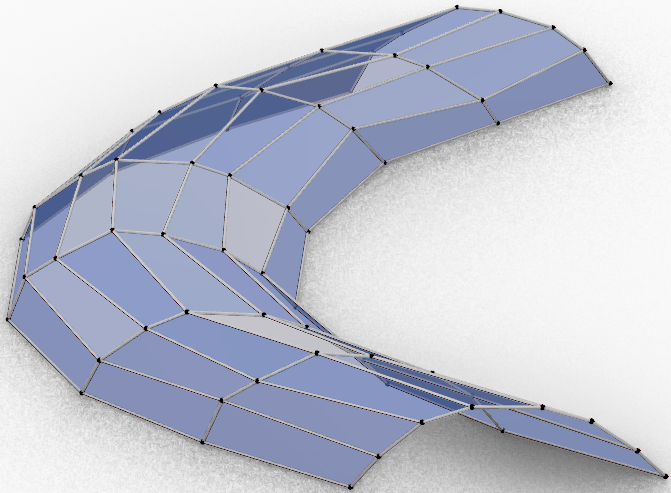}\quad\includegraphics[scale=0.22]{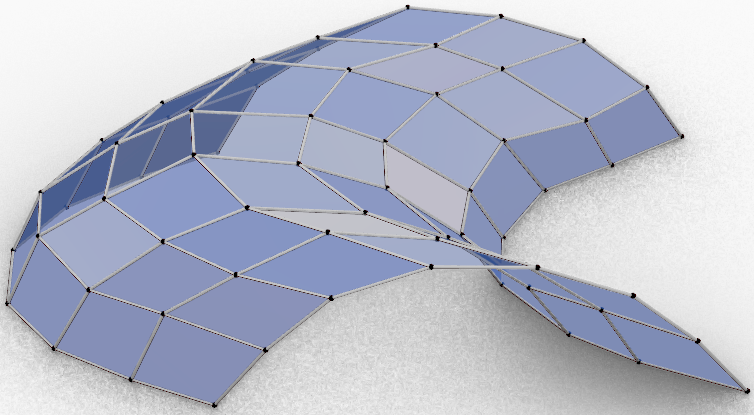}\quad\includegraphics[scale=0.22]{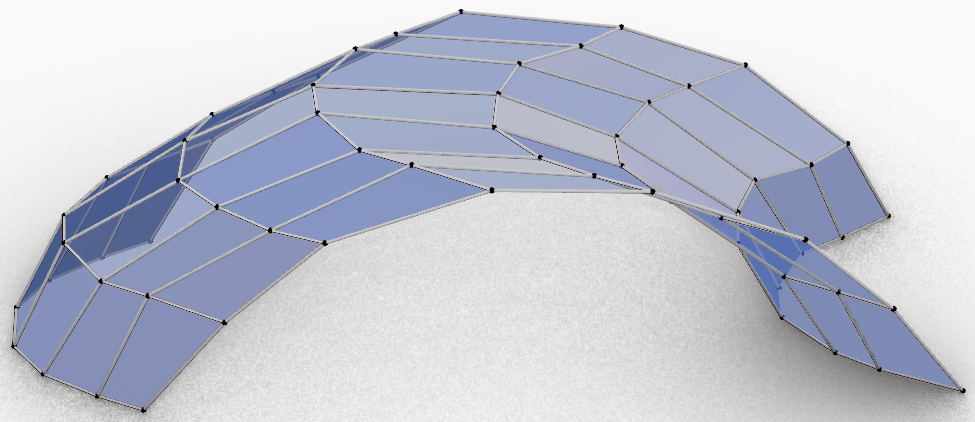}
    \caption{A sequence of deformations of a net from class (ii) in $\mathbb R^3$. %Corresponding edges are parallel and corresponding faces have equal areas. 
    Any two neighboring faces have equal so-called opposite ratios with respect to the common edge. See Section~\ref{ssec-statement-deformable-2x2} and Theorem~\ref{th-mxn}}
    \label{fig:type ii more}
\end{figure}

%\begin{figure}[htbp]
%    \centering
%    \includegraphics[scale=0.4]{fig/46.png}
%    \caption{TO BE REPLACED. A deformable $5\times 5$ net with $2\times 2$ sub-nets of type (ii) in $\mathbb R^3$. See Theorem~\ref{th-classification}}
%    \label{fig:type i}
%\end{figure}

%\begin{figure}[htbp]
%    \centering
%    \includegraphics[scale=0.19]{fig/h5.png}\quad\includegraphics[scale=0.19]{fig/h4.png}\quad\includegraphics[scale=0.19]{fig/h7.png}
%    \caption{A sequence of deformations of a $9\times9$ net from class (ii) in $\mathbb R^3$. %Corresponding edges are parallel and corresponding faces have equal areas. Any two neighbour faces have equal opposite ratios with respect to the common side. 
%    See Theorem~\ref{th-mxn}}
%    \label{fig:type ii}
%\end{figure}

\begin{figure}[htbp]
    \centering
    \includegraphics[scale=0.048]%[scale=0.055]
    {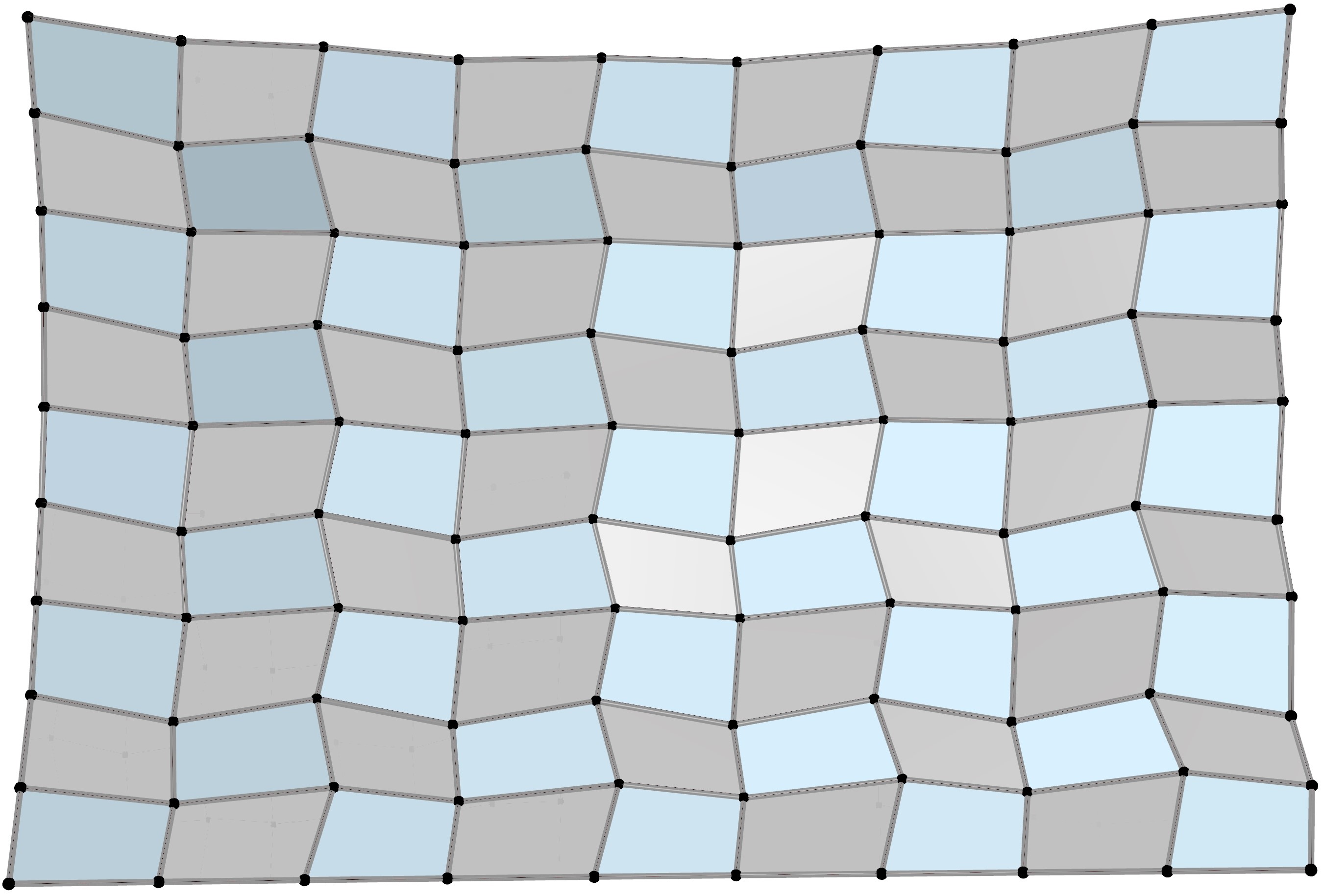}\quad\includegraphics[scale=0.048]%[scale=0.055]
    {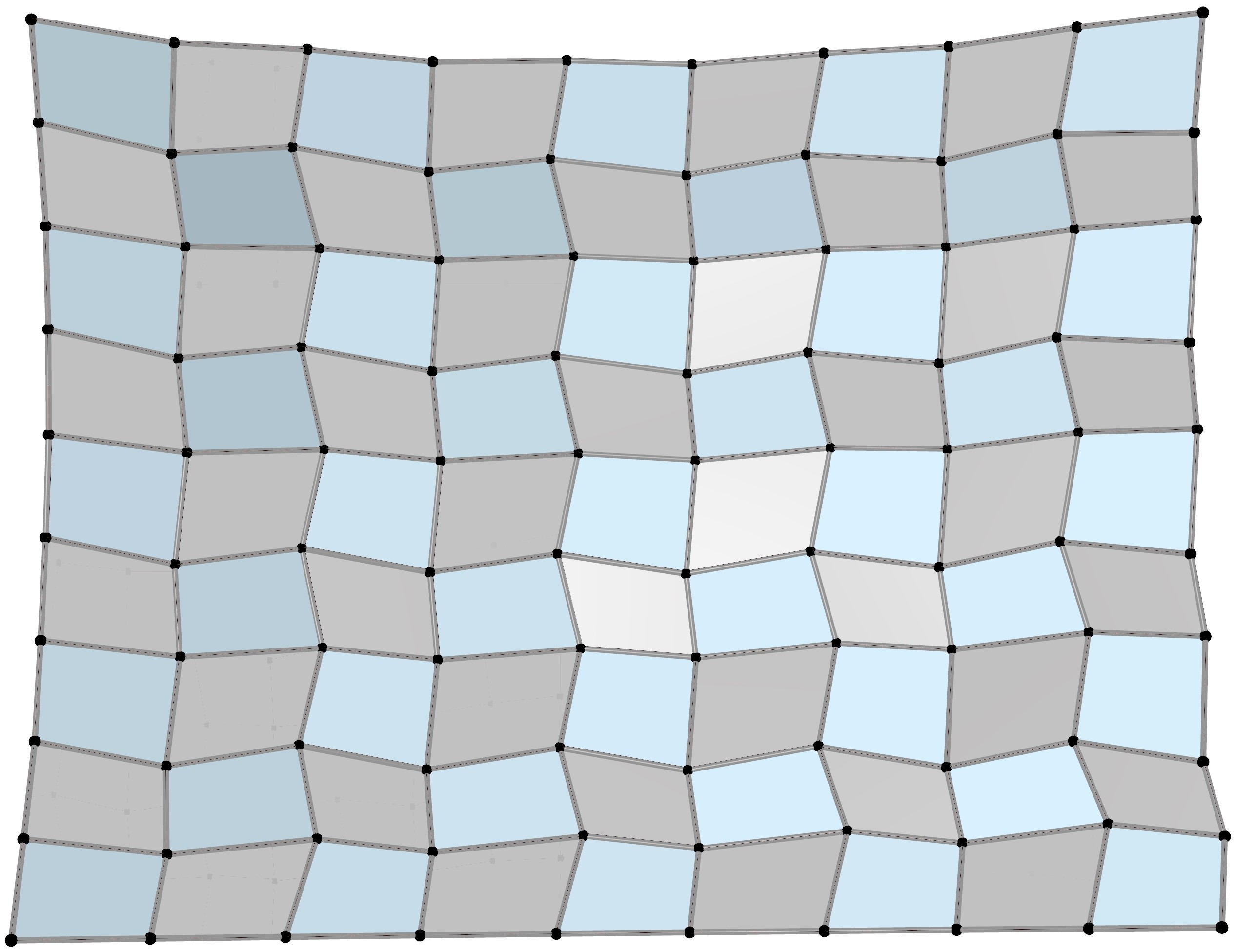}\quad\includegraphics[scale=0.048]%[scale=0.055]
    {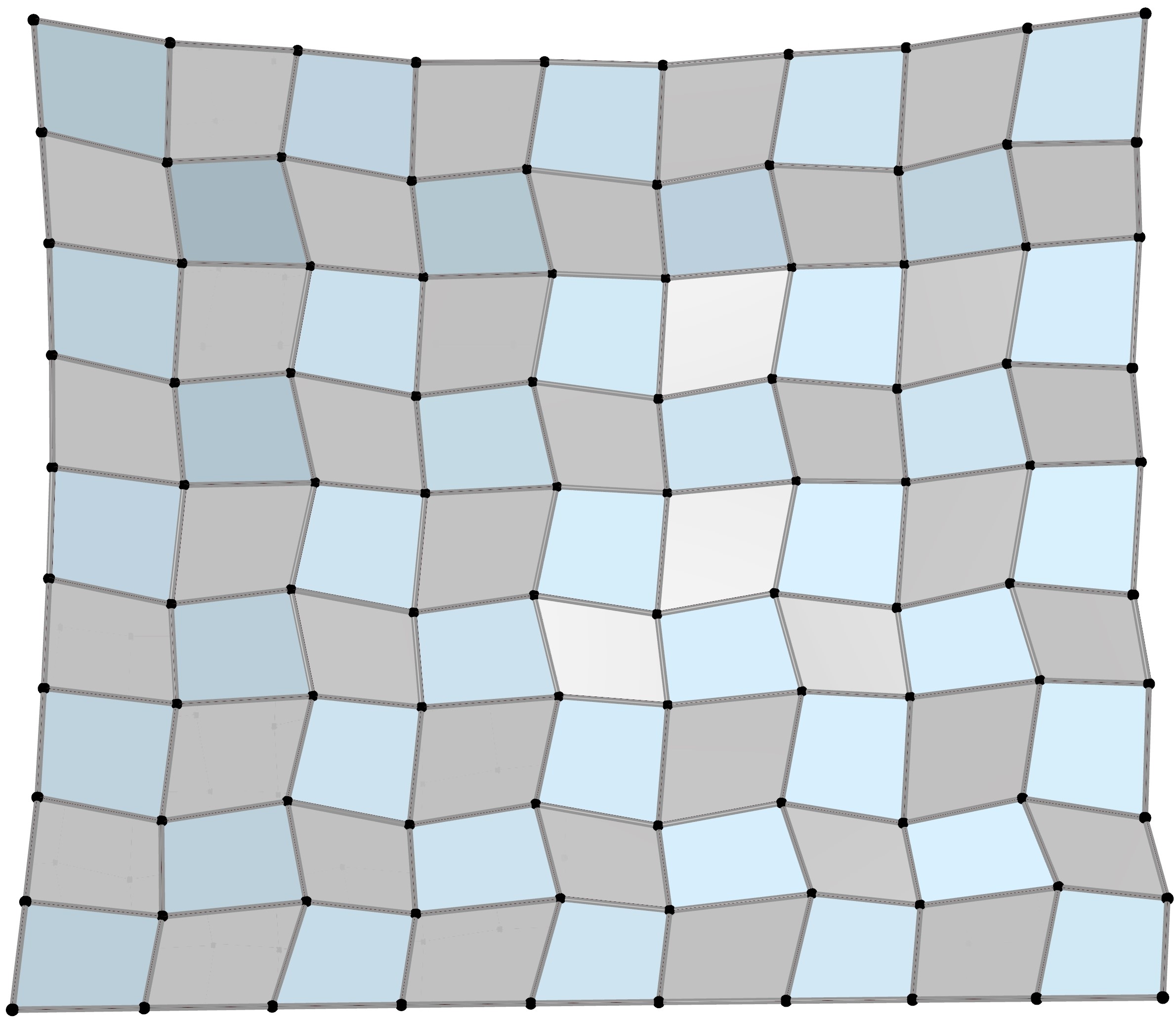}
    \caption{A sequence of deformations of a $9\times9$ net from class (ii) in $\mathbb R^3$. %Corresponding edges are parallel and corresponding faces have equal areas. Any two neighbour faces have equal opposite ratios with respect to the common side. 
    See Theorem~\ref{th-mxn}}
    \label{fig:type ii}
\end{figure}

%\begin{figure}[htbp]
%    \centering
%    \includegraphics[scale=0.28]{fig/30.png}\qquad\qquad\includegraphics[scale=0.28]{fig/31.png}
%    \caption{A deformable $5\times5$ net with its deformation in $\mathbb R^3$. Here all $2\times 2$ sub-nets are of type (ii).
%    See Theorem~\ref{th-classification}.
%    }
%    \label{fig:type ii}
%\end{figure}

%\mscomm{
%\bigskip
%------------------ START READING HERE ---------------------}

%\mscomm{------------------ END OF THE SECTION STILL IN PROGRESS ---------------------------}

\section{Deformable $2\times 2$ nets}\label{sec-deformable-dual-nets}
%\ \ \ \\
%{\color{green}OP:Short comments are mine.}

In this section, we introduce basic notions and characterize all deformable $2\times 2$ nets, i.e., nets that admit a family of non-congruent area-preserving Combescure transformations. %There turn out to be only two classes of such nets, described in Theorem~\ref{th-classification} and illustrated in Figure~\ref{fig-defs-i-ii} below. 
We state the Classification Theorem~\ref{th-classification} in Section~\ref{ssec-statement-deformable-2x2}, discuss the geometry of deformable nets in Section~\ref{ssec-properties-deformable-2x2}, and give the proof in Section~\ref{ssec-proof-deformable-2x2}. 

%\subsection{Classification of deformable $2\times 2$ nets}
\subsection{Statement of the classification}
\label{ssec-statement-deformable-2x2}
Let us introduce basic notions. 
By a $m\times n$ \emph{net} we mean a collection of $(m+1)(n+1)$ points $P_{ij}\in\mathbb{R}^d$ indexed by two integers $0\le i\le m$ and $0\le j\le n$ (i.e., a map $\{0,\dots,m\}\times \{0,\dots,n\}\to\mathbb{R}^d$), such that $P_{ij},P_{i+1,j},P_{i+1,j+1},P_{i,j+1}$ are consecutive vertices of a convex quadrilateral for all $0\le i< m$ and $0\le j< n$. See Figure~\ref{figure:lemma-eqqq}.
%\hiddencomment{To do later: check if the results remain true without convexity. In this def, the points are required to be distinct.} 
The latter quadrilaterals are called \emph{faces}, and their sides are called \emph{edges}.
%The latter (labeled) quadrilaterals are called \emph{faces}, and their (labeled) sides are called \emph{edges}. (The term `labeled' is accurately defined in~\cite[\S2]{Mor20}; it can be just omitted if all the points $P_{ij}$ are distinct.)
Faces and edges are \emph{labeled}, i.e., equipped with an assignment of indices to their vertices (see~\cite[\S2]{Mor20} for details).
An $m\times n$ net has $mn$ faces, thus the~name.
%; however, the points $P_{ij}$ need not to be distinct.

Edges or faces spanned by points with the same indices in two $m\times n$ nets are called \emph{corresponding}. 
%Edges or faces in two $m\times n$ nets are called \emph{corresponding} if their vertices have the same indices. 
Two $m\times n$ nets are \emph{parallel} or \emph{Combescure transformations} of each other if their corresponding edges are parallel. Two $m\times n$ nets are \emph{congruent}, if there is an isometry of $\mathbb{R}^d$ taking each point of the first net to the point of the second net with the same indices.

We come to the main notion of the paper.

\begin{definition} \label{def-deformable-net} (See Figure~\ref{figure:lemma-eqqq}) An $m\times n$ net is \emph{deformable} if contained in a continuous family of non-congruent parallel $m\times n$ nets %$P_{ij}(t)$ 
with the same areas of corresponding faces. 
Any member of this family is called a \emph{deformation}
%, or more accurately, an \emph{area-preserving Combescure transformation} 
of the initial net.
\end{definition}

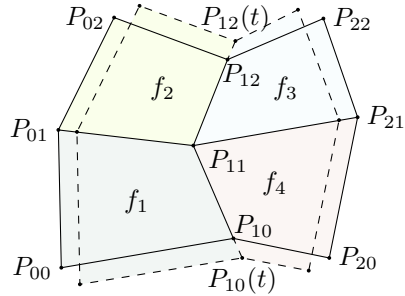
\begin{figure}[htbp]
  \centering
\begin{tikzpicture}[scale=0.13]%[scale=0.15]

\coordinate (A1) at (11.7, -16);
\coordinate (O) at (8.2, -24.8);
\coordinate (A2) at (25, -21.9);
\coordinate (B2) at (22.1, -36.3);
\coordinate (A3) at (12.3, -34.3);
\coordinate (B3) at (-5.3, -37.3);
\coordinate (A4) at (-5.6, -23.2);
\coordinate (B4) at (0.1, -11.7);
\coordinate (B1) at (21.5, -11.8);
\coordinate (A1t) at (12.5, -14.1);
\coordinate (B1t) at (18.9, -10.9);
\coordinate (A2t) at (23.1, -22.2);
\coordinate (B2t) at (20, -37.6);
\coordinate (A3t) at (13.2, -36.3);
\coordinate (B3t) at (-3.4, -39);
\coordinate (A4t) at (-3.7, -23.4);
\coordinate (B4t) at (2.7, -10.5);

\fill[babyblue!5] (O) -- (A1) -- (B1) -- (A2) -- cycle;
\fill[babyblue!5] (O) -- (A1t) -- (B1t) -- (A2t) -- cycle;
\fill[antiquebrass!8] (O) -- (A2) -- (B2) -- (A3) -- cycle;
\fill[antiquebrass!8] (O) -- (A2t) -- (B2t) -- (A3t) -- cycle;
\fill[britishracinggreen!5] (O) -- (A3) -- (B3) -- (A4) -- cycle;
\fill[britishracinggreen!5] (O) -- (A3t) -- (B3t) -- (A4t) -- cycle;
\fill[lime!8] (O) -- (A4) -- (B4) -- (A1) -- cycle;
\fill[lime!8] (O) -- (A4t) -- (B4t) -- (A1t) -- cycle;

\draw[black, thin] (A1) -- (O) -- (A2) -- (B2) -- (A3) -- (B3) -- (A4) -- (B4) -- (A1) -- (B1) -- (A2);
\draw[black, thin] (A4) -- (O) -- (A3);
\draw[black, thin, dashed] (A1t) -- (B1t) -- (A2t) -- (B2t) -- (A3t) -- (B3t) -- (A4t) -- (B4t) -- (A1t);
\draw[black, thin, dashed] (A1t) -- (A1);
\draw[black, thin, dashed] (A3t) -- (A3);

\filldraw[black] (O) circle (4pt);
\filldraw[black] (8.8, -24.3) circle (0pt) node[anchor=north west]{$P_{11}$};

\filldraw[black] (25, -21.9) circle (4pt) node[anchor=west]{$P_{21}$};
\filldraw[black] (12.3, -34.3) circle (4pt);
\filldraw[black] (11.4, -35.1) circle (0pt) node[anchor=south west]{$P_{10}$};
\filldraw[black] (-5.6, -23.2) circle (4pt) node[anchor=east]{$P_{01}$};
\filldraw[black] (11.7, -16) circle (4pt); 
\filldraw[black] (10.2, -15.5) circle (0pt) node[anchor=north west]{$P_{12}$};

\filldraw[black] (21.5, -11.8) circle (4pt) node[anchor=west]{$P_{22}$};
\filldraw[black] (22.1, -36.3) circle (4pt) node[anchor=west]{$P_{20}$};
\filldraw[black] (-5.3, -37.3) circle (4pt) node[anchor=east]{$P_{00}$};
\filldraw[black] (0.1, -11.7) circle (4pt) node[anchor=east]{$P_{02}$};

\filldraw[black] (23.1, -22.2) circle (4pt); %node[anchor=south east]{$A_2(t)$};
\filldraw[black] (-3.7, -23.4) circle (4pt); %node[anchor=north west]{$A_4(t)$};
\filldraw[black] (12.5, -14.1) circle (4pt) node[anchor=south]{$P_{12}(t)$};
\filldraw[black] (13.2, -36.3) circle (4pt) node[anchor=north]{$P_{10}(t)$};

\filldraw[black] (18.9, -10.9) circle (4pt); %node[anchor=south]{$B_1(t)$};
\filldraw[black] (20, -37.6) circle (4pt); %node[anchor=north]{$B_2(t)$};
\filldraw[black] (-3.4, -39) circle (4pt); %node[anchor=north]{$B_3(t)$};
\filldraw[black] (2.7, -10.5) circle (4pt); %node[anchor=south]{$B_4(t)$};

\filldraw[black] (17.63456,-18.97232) node[anchor=center]{$f_3$};
\filldraw[black] (18.74035,-28.64956) node[anchor=east]{$f_4$};
\filldraw[black] (2.40359,-30.5) node[anchor=center]{$f_1$};
\filldraw[black] (2.84118,-18.94961) node[anchor=west]{$f_2$};
\end{tikzpicture}

\caption{A $2\times 2$ net (solid lines) and its deformation (dashed lines). Corresponding edges are parallel and corresponding faces (colored with the same color) have equal areas. See Definition~\ref{def-deformable-net}.}
\label{figure:lemma-eqqq}
\end{figure}

\begin{figure}[htbp]
\centering

\begin{tikzpicture}[scale=0.08]%[scale=0.11]

\coordinate (P) at (0.06565, -0.004385);
\coordinate (P') at (20.001513, 0.837969);
\coordinate (D) at (40.064928, -0.002347);
\coordinate (C) at (4.581406, 21.030862);
\coordinate (B) at (20.015315, 12.739852);
\coordinate (C') at (35.568782, 21.032441);
\coordinate (D') at (-6.163002, 39.917815);
\coordinate (A) at (17.421152, 27.538802);
\coordinate (E) at (41.300775, 39.920233);

% Fill green regions
\fill[pattern=north west lines, pattern color=aurometalsaurus!70] (B) -- (C) -- (D') -- (A) -- cycle;
\fill[pattern=north east lines, pattern color=aurometalsaurus!70] (B) -- (C') -- (E) -- (A) -- cycle;

% Fill blue regions
\fill[pattern=north east lines, pattern color=auburn!70] (P') -- (P) -- (C) -- (B) -- cycle;
\fill[pattern=north west lines, pattern color=auburn!70] (D) -- (P') -- (B) -- (C') -- cycle;

\draw[applegreen, thick] (C) -- (B) -- (C');
\draw[applegreen, thick] (P') -- (B) -- (A);
\draw[applegreen, thick] (D') -- (A) -- (E) -- (C') -- (D) -- (P') -- (P) -- (C) -- (D');
\draw[applegreen, thick] (C) -- (B) -- (C');
%\draw[gray!70] (C) -- (B) -- (C');
%\draw[gray!70] (P') -- (B) -- (A);
%\draw[gray!70] (D') -- (A) -- (E) -- (C') -- (D) -- (P') -- (P) -- (C) -- (D');
%\draw[gray!70] (C) -- (B) -- (C');

\filldraw[airforceblue] (D) circle (10pt) node[anchor=east]{};
\filldraw[airforceblue] (C) circle (10pt) node[anchor=south east]{};
\filldraw[airforceblue] (B) circle (10pt) node[anchor=south]{};
\filldraw[airforceblue] (C') circle (10pt) node[anchor=south west]{};
\filldraw[airforceblue] (D') circle (10pt) node[anchor=west]{};
\filldraw[airforceblue] (A) circle (10pt) node[anchor=north]{};
\filldraw[airforceblue] (E) circle (10pt) node[anchor=south east]{};
\filldraw[airforceblue] (P') circle (10pt) node[anchor=west]{};
\filldraw[airforceblue] (P) circle (10pt) node[anchor=east]{};

%\filldraw[gray] (D) circle (10pt) node[anchor=east]{};
%\filldraw[gray] (C) circle (8pt) node[anchor=south east]{};
%\filldraw[gray] (B) circle (8pt) node[anchor=south]{};
%\filldraw[gray] (C') circle (8pt) node[anchor=south west]{};
%\filldraw[gray] (D') circle (8pt) node[anchor=west]{};
%\filldraw[gray] (A) circle (8pt) node[anchor=north]{};
%\filldraw[gray] (E) circle (8pt) node[anchor=south east]{};
%\filldraw[gray] (P') circle (8pt) node[anchor=west]{};
%\filldraw[gray] (P) circle (8pt) node[anchor=east]{};

\draw[gray, very thin, <->] (13.8,7.95) -- (27,7.95);
\draw[gray, very thin, <->] (13.1,24) -- (25.9,24);
\filldraw[white] (11.6,8) circle (62pt);
\filldraw[black] (11.6,8) circle (0.0pt) node[]{2};
\filldraw[white] (29,8) circle (62pt);
\filldraw[black] (29,8) circle (0.0pt) node[]{$2'$};
\filldraw[white] (11,24) circle (62pt);
\filldraw[black] (11,24) circle (0.0pt) node[]{1};
\filldraw[white] (28,24) circle (62pt);
\filldraw[black] (28,24) circle (0.0pt) node[]{$1'$};

\end{tikzpicture}
\qquad\qquad\qquad
\begin{tikzpicture}[scale=1.7]%[scale=2.1]

\coordinate (P) at (0.141595, 0.064912);
\coordinate (P') at (0.833477, 0.183469);
\coordinate (D) at (2.059665, 0.087001);
\coordinate (C) at (0.009566, 1.019064);
\coordinate (B) at (0.975775, 0.830942);
\coordinate (C') at (2.992722, 1.052864);
\coordinate (D') at (0.183345, 2.11201);
\coordinate (A) at (0.863328, 1.696711);
\coordinate (E) at (2.113132, 2.121726);

%\fill[blue!100] (C) -- (B) -- (P) -- (P') -- (C) -- (B) -- (D') -- (A) -- (C);
%\fill[green!100] (A) -- (C) -- (D') -- (B) -- (A) -- (C') -- (E) -- (B) -- (A);
%\fill[red!100] (C') -- (A) -- (E) -- (B) -- (C') -- (P') -- (D) -- (B) -- (C');
%\fill[yellow!100] (P') -- (C') -- (D) -- (B) -- (P') -- (C) -- (P) -- (B) -- (P');

\path[name path=line 1] (B) -- (D);
\path[name path=line 2] (P') -- (C');
\path[name path=line 3] (A) -- (C');
\path[name path=line 4] (B) -- (E);
\path[name path=line 5] (B) -- (D');
\path[name path=line 6] (A) -- (C);
\path[name path=line 7] (B) -- (P);
\path[name path=line 8] (P') -- (C);

\fill[name intersections={of=line 1 and line 2, by=X1}];
\fill[name intersections={of=line 3 and line 4, by=X2}];
\fill[name intersections={of=line 5 and line 6, by=X3}];
\fill[name intersections={of=line 7 and line 8, by=X4}];

% Fill red regions
\fill[pattern=north west lines, pattern color=antiquebrass!70] (B) -- (X1) -- (C') -- cycle;
\fill[pattern=north west lines, pattern color=antiquebrass!70] (P') -- (X1) -- (D) -- cycle;
\fill[pattern=north west lines, pattern color=antiquebrass!70] (B) -- (X2) -- (C') -- cycle;
\fill[pattern=north west lines, pattern color=antiquebrass!70] (A) -- (X2) -- (E) -- cycle;

% Fill blue regions
\fill[pattern=north east lines, pattern color=armygreen!70] (E) -- (X2) -- (C') -- cycle;
\fill[pattern=north east lines, pattern color=armygreen!70] (A) -- (X2) -- (B) -- cycle;
\fill[pattern=north east lines, pattern color=armygreen!70] (A) -- (X3) -- (B) -- cycle;
\fill[pattern=north east lines, pattern color=armygreen!70] (D') -- (X3) -- (C) -- cycle;

% Fill green regions
\fill[pattern=north west lines, pattern color=applegreen!70] (A) -- (X3) -- (D') -- cycle;
\fill[pattern=north west lines, pattern color=applegreen!70] (C) -- (X3) -- (B) -- cycle;
\fill[pattern=north west lines, pattern color=applegreen!70] (C) -- (X4) -- (B) -- cycle;
\fill[pattern=north west lines, pattern color=applegreen!70] (P) -- (X4) -- (P') -- cycle;

%yellow regions
\fill[pattern=north east lines, pattern color=aurometalsaurus!70] (P) -- (X4) -- (C) -- cycle;
\fill[pattern=north east lines, pattern color=aurometalsaurus!70] (P') -- (X4) -- (B) -- cycle;
\fill[pattern=north east lines, pattern color=aurometalsaurus!70] (B) -- (X1) -- (P') -- cycle;
\fill[pattern=north east lines, pattern color=aurometalsaurus!70] (C') -- (X1) -- (D) -- cycle;

\draw[applegreen, thick] (C) -- (B) -- (C');
\draw[applegreen, thick] (P') -- (B) -- (A);
\draw[applegreen, thick] (D') -- (A) -- (E) -- (C') -- (D) -- (P') -- (P) -- (C) -- (D');
\draw[applegreen, thick] (C) -- (B) -- (C');

\draw[gray!50, thin] (C) -- (A);
\draw[gray!50, thin] (B) -- (D');
\draw[gray!50, thin] (E) -- (B);
\draw[gray!50, thin] (A) -- (C');
\draw[gray!50, thin] (P') -- (C');
\draw[gray!50, thin] (B) -- (D);
\draw[gray!50, thin] (B) -- (P);
\draw[gray!50, thin] (C) -- (P');

\filldraw[airforceblue] (D) circle (0.5pt) node[anchor=east]{};
\filldraw[airforceblue] (C) circle (0.5pt) node[anchor=south east]{};
\filldraw[airforceblue] (B) circle (0.5pt) node[anchor=south]{};
\filldraw[airforceblue] (C') circle (0.5pt) node[anchor=south west]{};
\filldraw[airforceblue] (D') circle (0.5pt) node[anchor=west]{};
\filldraw[airforceblue] (A) circle (0.5pt) node[anchor=north]{};
\filldraw[airforceblue] (E) circle (0.5pt) node[anchor=south east]{};
\filldraw[airforceblue] (P') circle (0.5pt) node[anchor=west]{};
\filldraw[airforceblue] (P) circle (0.5pt) node[anchor=east]{};

\filldraw[white] (1.5,1.7) circle (2.6pt);
\filldraw[black] (1.5,1.7) circle (0.0pt) node[]{$2'$};
\filldraw[white] (1.7,1.19) circle (2.6pt);
\filldraw[black] (1.7,1.19) circle (0.0pt) node[]{$1'$};
\filldraw[white] (1.6,0.7) circle (2.6pt);
\filldraw[black] (1.6,0.7) circle (0.0pt) node[]{1};
\filldraw[white] (1.45,0.28) circle (2.6pt);
\filldraw[black] (1.45,0.28) circle (0.0pt) node[]{2};

\end{tikzpicture}
\caption{Two classes of deformable $2\times 2$ nets introduced in Theorem~\ref{th-classification}. Left: class (i). Quadrilaterals in each pair $1,1'$ and $2,2'$ are affine symmetric, i.e., there exist two affine maps: one maps $1$ to $1'$ and the other maps $2$ to $2'$, keeping the points of the common sides fixed. Right: class (ii). Two pairs of red triangles $1,2$ and $1',2'$ have equal area ratios. This property holds for other pairs of triangles with the same color.}
\label{fig-defs-i-ii}
\end{figure}
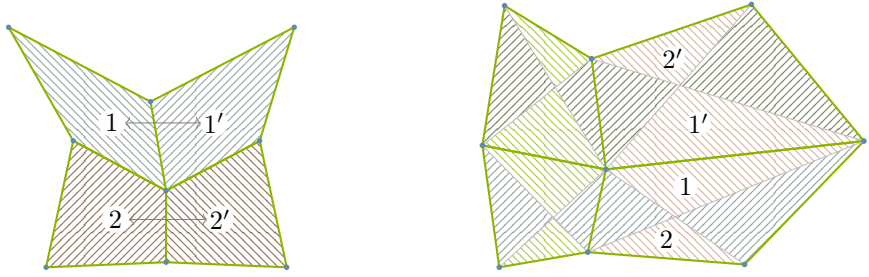

\begin{figure}[htbp]
  \centering
\begin{tikzpicture}[scale=0.8]%[scale=1.1]

\coordinate (A) at (-3.82, -0.74);
\coordinate (D) at (-2.8416, 1.03277);
\coordinate (S) at (-2.29447, 2.02757);
\coordinate (B) at (-1.06341, -0.70812);
\coordinate (R) at (0.50339, -0.68325);

\path[name path=line 1] (S) -- (B);
\path[name path=line 2] (D) -- (R);
\fill[name intersections={of=line 1 and line 2, by=C}];

\path[name path=line 3] (D) -- (B);
\path[name path=line 4] (A) -- (C);
\fill[name intersections={of=line 3 and line 4, by=Q}];

\fill[pattern=north east lines, pattern color=applegreen!70] (D) -- (Q) -- (C) -- cycle;
\fill[pattern=north east lines, pattern color=applegreen!70] (A) -- (Q) -- (B) -- cycle;

%\draw[airforceblue] (D) -- (A) -- (B) -- (C) -- (D) -- (B);
%\draw[airforceblue, thin] (A) -- (C);
%\draw[auburn, thin, dashed] (D) -- (S) -- (C) -- (R) -- (B);

\draw[black] (D) -- (A) -- (B) -- (C) -- (D);
\draw[gray!50, very thin] (D) -- (B);
\draw[gray!50, very thin] (A) -- (C);
\draw[gray, thin, dashed] (D) -- (S) -- (C) -- (R) -- (B);

\filldraw[black] (A) circle (0.5pt) node[anchor=north east]{$A$};
\filldraw[black] (B) circle (0.5pt) node[anchor=north]{$B$};
\filldraw[black] (S) circle (0.5pt) node[anchor=south]{$S$};
\filldraw[black] (C) circle (0.5pt) node[anchor=south west]{$C$};
\filldraw[black] (D) circle (0.5pt) node[anchor=east]{$D$};
\filldraw[black] (R) circle (0.5pt) node[anchor=west]{$R$};
\filldraw[black] (Q) circle (0.5pt); \filldraw[black] (-2.03,0.2) circle (0pt) node[anchor=east]{$Q$};

\end{tikzpicture}

\caption{The opposite ratio of the quadrilateral $ABCD$ with respect to the side $AB$ is the ratio of the areas of the colored triangles. See Proposition~\ref{p-opposite-ratio-equivalent-def} for equivalent definitions.}
\label{figure:def-oppos-rat}
\end{figure}
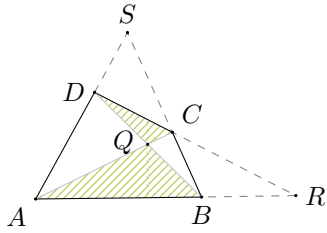

To state the classification theorem, we need a convention and two auxiliary notions.

The points of a $m\times n$ net need not be distinct. To deal with coincident points, 
we have labeled faces and edges.
%in what follows we assume that the faces and edges are labeled, i.e., equipped with an assignment of indices to their vertices (see~\cite[\S2]{Mor20} for details). 
Faces or edges with different labels are viewed as distinct even if they coincide as subsets of $\mathbb{R}^d$. In particular, a common edge of two faces must have common labels with both.

A pair of quadrilaterals in $\mathbb{R}^d$ with a common side is called \emph{affine symmetric with respect to the common side} if there is an affine map taking the first quadrilateral to the second one and keeping the points of the common side fixed. (Labels, if any, are ignored.) See Figure~\ref{fig-defs-i-ii}(left) for an illustration. %and Proposition~\ref{th-old} for simple equivalent definitions.

The ratio of the areas of triangles $AQB$ and $CQD$ in a quadrilateral $ABCD$ with the diagonals meeting at $Q$ is called the \emph{opposite ratio of the quadrilateral $ABCD$ with respect to the side} $AB$. 
See Figures~\ref{figure:def-oppos-rat} and~\ref{fig-defs-i-ii}(right). %for an illustration.
Notice that if two quadrilaterals are affine symmetric with respect to the common side, then their opposite ratios with respect to all corresponding sides are equal. 

\begin{theorem} \label{th-classification} 
A $2\times 2$ net is deformable if and only if at least one of the following conditions~holds:
\begin{enumerate}
    \item[\textup{(i)}] the four faces split into two pairs that are affine symmetric with respect to the common edges;
    \item[\textup{(ii)}] each pair of faces with a common edge has equal opposite ratios with respect to that edge.
    %\item[(i)] the four quadrilaterals split into two pairs affine symmetric with respect to the common sides;
    %\item[(ii)] each pair of quadrilaterals with a common side has equal opposite ratios with respect to this side.
\end{enumerate}
\end{theorem}

%%%%
%\begin{theorem} \label{th-sym} Enumerate the quadrilaterals of a $2\times 2$ net (in $\mathbb{R}^2$ or $\mathbb{R}^3$) counterclockwise. If both pairs of quadrilaterals $1,2$ and $3,4$ are affine symmetric with respect to the common sides, then the net is deformable.
%\end{theorem}

%\begin{remark}
%Here ($i$) holds if and only if conditions ($i$)--($iii$) of Theorem~\ref{th-old} below are satisfied.
%Example ($i$) obviously generalizes to a $2\times N$ net.
%\end{remark}

The reader interested in the proof of the theorem can proceed directly to Section~\ref{ssec-proof-deformable-2x2}, and now we discuss geometric properties of resulting classes~(i)--(ii) of deformable $2\times 2$ nets.

%We prove the theorem in Section~\ref{ssec-proof-deformable-2x2}; the reader can proceed directly to that section, and now discuss geometric properties of classes~(i)--(ii). Subsections~\ref{ssec-properties-deformable-2x2} and~\ref{ssec-proof-deformable-2x2} are independent.

\subsection{Geometric properties}
\label{ssec-properties-deformable-2x2}

There are equivalent descriptions of classes~(i) and~(ii), %of deformable $2\times 2$ nets, 
providing additional insight and simple ways to check conditions~(i) and~(ii) of Theorem~\ref{th-classification} for a given net.

\subsubsection{Class~(i)} Let us start with class~(i) and show that it indeed consists of deformable nets.
%condition~(i) in Theorem~\ref{th-classification} is sufficient for the net to be deformable.

\begin{example}\label{ex-i} 
A $2\times 2$ net %satisfying condition~(i) in Theorem~\ref{th-classification}
whose faces split into two pairs that are affine symmetric with respect to the common edges is deformable. 
\end{example}

\begin{proof} See Figure~\ref{figure:lemma-eqqq} and enumerate the faces as shown there. Assume that both pairs of faces $f_1,f_2$ and $f_3,f_4$ are affine symmetric with respect to the common sides.

%Indeed, 
Fix the common vertex $P_{11}$ of all the faces. %Let $QA_i$ be the common side of $f_i$ and $f_{i-1}$ for $i=1,2,3,4$, where $f_0:=f_4$ (see Figure~\ref{figure:lemma-eqqq}). 
Move the point $P_{10}$ of the net slightly along the line $P_{11}P_{10}$ to a new position $P_{10}(t)$. On the resulting segment $P_{11}P_{10}(t)$, construct a new quadrilateral $f_1(t)$ with the same area and the same side directions as $f_1$. Consider the affine symmetry taking $f_1$ to $f_2$, and apply it to $f_1(t)$. We get a quadrilateral $f_2(t)$ with the same area and the same side directions as $f_2$ because an affine map preserves parallelism of lines 
and the ratio of areas. Also, $f_1(t)$ and $f_2(t)$ share a common side because the points of the line $P_{11}P_{01}$ are fixed by the affine symmetry.

Analogously construct a quadrilateral $f_4(t)$ and its image $f_3(t)$ under the affine symmetry taking $f_4$ to $f_3$. Then both pairs $f_4,f_4(t)$ and $f_3,f_3(t)$ have parallel sides and equal areas, and $f_3(t)$ and $f_4(t)$ share a common side.
It remains to note that $f_2(t)$ and $f_3(t)$ also share a common side: indeed, both affine symmetries take $P_{10}$ to $P_{12}$ and preserve the ratio %of collinear segments 
$P_{11}P_{10}:P_{11}P_{10}(t)$, hence they take $P_{10}(t)$ to the same point $P_{12}(t)$. %on the line $P_{11}P_{12}$. 
Thus $f_1(t),f_2(t),f_3(t),f_4(t)$ together constitute a deformation of our net.
\end{proof}

Let us introduce a geometric characterization of affine symmetric quadrilaterals. See Figure~\ref{figure:properties}.

%There are several simple characterizations of affine symmetric pairs of quadrilaterals, giving simple ways to check condition~(i) of Theorem~\ref{th-classification}. For instance, two quadrilaterals $ABCD$ and $ABC'D'$ are affine symmetric with respect to the common side, if and only if their opposite ratios with respect to respective sides are equal. There is also the following geometric characterization. See Figure~\ref{figure:properties}.

\begin{proposition}\label{th-old}
Two convex quadrilaterals $ABCD$ and $ABC'D'$ in $\mathbb{R}^d$, where $C, D, C', D'$ are non-collinear, are affine symmetric with respect to their common side $AB$, if and only if $CC'\parallel DD'$ and the lines $AB$, $CD$, $C'D'$ are concurrent or parallel. Moreover, if the quadrilaterals are not coplanar, then the latter condition %is automatic. 
and the assumption that $C, D, C', D'$ are non-collinear
can be dropped.
%%%
%Consider a $2\times2$ net with the faces $F_i=QA_iB_iA_{i+1}$, where $i=1,2,3,4$ and $A_5:=A_1$. The pairs of faces $(F_1,F_4)$ and $(F_2,F_3)$ are affine symmetric with respect to their common sides, if and only if the following $3$ conditions hold:
%
%(1) $A_4B_4$, $QA_1$, $A_2B_1$ are concurrent or parallel, 
%
%(2) $A_4B_3$, $QA_3$ $A_2B_2$ are concurrent or parallel,
%
%(3) $B_4B_1\parallel A_4A_2\parallel B_3B_2$. 
%
%Moreover, if neither of the pairs $(F_1,F_4)$ and $(F_2,F_3)$ is coplanar, then (1) and (2) are automatic and can be dropped.
%%%
\end{proposition}

\begin{proof}[Proof] 
    \emph{The 'only if' part.} Let %Quadrilaterals
    $ABCD$ and $ABC'D'$ be affine symmetric. Then %, if and only if 
    there is an affine map that maps $C, D$ to $C', D'$, respectively, and preserves the points of the line $AB$.
    The map takes the intersection point $AB\cap CD$, if it exists, to $AB \cap C'D'$, which means that $AB, CD, C'D'$ are parallel or concurrent. In particular, $C\ne C'$ and $D\ne D'$, otherwise $C, D, C', D'$ are collinear. Since the affine map preserves ratios of parallel segments, by the Thales theorem we get $CC'\parallel DD'$.
    %, and we are done. 
    
    \emph{The 'if' part.} Consider the affine map that maps the triangle $ACB$ to $AC'B$. This map takes the line $CD$ to $C'D'$ because $AB, CD, C'D'$ are   
    concurrent or parallel. Since $CC'\parallel DD'$ and  $C, D, C', D'$ are non-collinear, by the Thales theorem it follows that $D$ is taken to $D'$, and we are done.  
\end{proof}

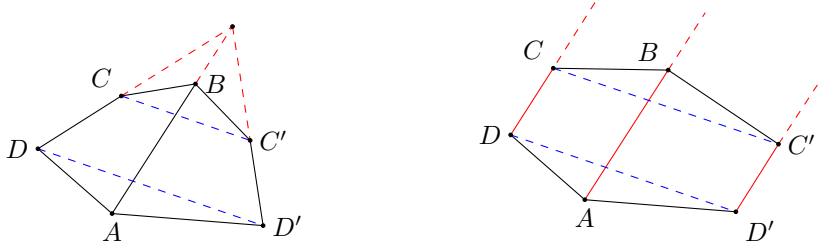
\begin{figure}[htbp]
  \centering
\begin{tikzpicture}[scale=0.17]
\draw[black, thin] (-2.50143, 1.26455) -- (4.02081, 5.3953) -- (9.81836, 6.3374) -- (14.09405, 1.91677) -- (15.10862, -4.75041) -- (3.29612, -3.80831) -- (-2.50143, 1.26455);
\draw[black, thin] (3.29612, -3.80831) -- (9.81836, 6.3374);

\draw[blue, dashed] (-2.50143, 1.26455) -- (15.10862, -4.75041);
\draw[blue, dashed] (4.02081, 5.3953) -- (14.09405, 1.91677);
\draw[red, dashed] (4.02081, 5.3953) -- (12.71713, 10.8305) -- (9.81836, 6.3374);
\draw[red, dashed] (12.71713, 10.8305) -- (14.09405, 1.91677);

\filldraw[black] (-2.50143, 1.26455) circle (4pt) node[anchor=east]{$D$};
\filldraw[black] (4.02081, 5.3953) circle (4pt) node[anchor=south east]{$C$};
\filldraw[black] (9.81836, 6.3374) circle (4pt) node[anchor=west]{$B$};
\filldraw[black] (14.09405, 1.91677) circle (4pt) node[anchor=west]{$C'$};
\filldraw[black] (15.10862, -4.75041) circle (4pt) node[anchor=west]{$D'$};
\filldraw[black] (3.29612, -3.80831) circle (4pt) node[anchor=north]{$A$};
\filldraw[black] (12.71713, 10.8305) circle (4pt) node[anchor=west]{};

\end{tikzpicture}\qquad\qquad\qquad 
\begin{tikzpicture}[scale=0.17]
\draw[black, thin] (0.83216, 6.48234) -- (9.81836, 6.3374) -- (18.44221, 0.53985); 
\draw[black, thin] (-2.50143, 1.26455) -- (3.29612, -3.80831) -- (15.10862, -4.75041);
\draw[red, thin] (-2.50143, 1.26455) -- (0.83216, 6.48234);
\draw[red, thin] (18.44221, 0.53985) -- (15.10862, -4.75041);
\draw[red, thin] (9.81836, 6.3374) -- (3.29612, -3.80831);
\draw[red, dashed] (4.09328, 11.62767) -- (0.83216, 6.48234);
\draw[red, dashed] (9.81836, 6.3374) -- (12.71713, 10.8305);
\draw[red, dashed] (18.44221, 0.53985) -- (21.48593, 5.17789);

\draw[blue, dashed] (-2.50143, 1.26455) -- (15.10862, -4.75041);
\draw[blue, dashed] (0.83216, 6.48234) -- (18.44221, 0.53985);

\filldraw[black] (-2.50143, 1.26455) circle (4pt) node[anchor=east]{$D$};
\filldraw[black] (0.83216, 6.48234) circle (4pt) node[anchor=south east]{$C$};
\filldraw[black] (9.81836, 6.3374) circle (4pt) node[anchor=south east]{$B$};
\filldraw[black] (18.44221, 0.53985) circle (4pt) node[anchor=west]{$C'$};
\filldraw[black] (15.10862, -4.75041) circle (4pt) node[anchor=north west]{$D'$};
\filldraw[black] (3.29612, -3.80831) circle (4pt) node[anchor=north]{$A$};

\end{tikzpicture}
\caption{Properties of two affine symmetric quadrilaterals. Blue lines are parallel and red lines are either concurrent or parallel. See Proposition~\ref{th-old}}
\label{figure:properties}
\end{figure}

%\mscomm{Write 2-3 sentences on how to check condition (i) using Proposition~\ref{th-old}. There was a nice figure before --- put it back (removing the old notation)!}

%Here the technical condition that $C, D, C', D'$ are non-collinear is automatic if the quadrilaterals are non-coplanar.

%Using the Proposition~\ref{th-old} 
As a consequence, we get the following geometric characterization of %deformable $2\times2$ nets of 
class~(i). %in Theorem~\ref{th-classification}. 
See Figure~\ref{figure:theorem(i)}.

\begin{proposition}
    \label{lemma-three-pic}
    Consider a $2\times2$ net with points $P_{ij}$, where $0\leq i,j\leq2$, such that each of the quadruples $P_{00},P_{10},P_{02},P_{12}$ and $P_{10},P_{20},P_{12},P_{22}$ are not collinear. The pairs of faces $$(P_{00}P_{10}P_{11}P_{01}, P_{01}P_{11}P_{12}P_{02})\qquad\text{and}\qquad (P_{10}P_{20}P_{21}P_{11}, P_{11}P_{21}P_{22}P_{12})$$ are affine symmetric with respect to their common edges if and only if $P_{00}P_{02}\parallel P_{10}P_{12}\parallel  P_{20}P_{22}$ and the lines in each of the triples $P_{00}P_{10},P_{01}P_{11},P_{02}P_{12}$ and $P_{10}P_{20},P_{11}P_{21},P_{12}P_{22}$ are parallel or concurrent. Moreover, if no two faces are coplanar, then the latter two conditions can be dropped.
    %%%%%%%%%%%%%
    %Consider a $2\times2$ net with points $P_{ij}$, where $0\leq i,j\leq2$, such that each of the quadruples $P_{00},P_{01},P_{20},P_{21}$ and $P_{01},P_{02},P_{21},P_{22}$ are not collinear. The pairs of faces $$(P_{00}P_{01}P_{11}P_{10}, P_{10}P_{11}P_{21}P_{20})\qquad\text{and}\qquad (P_{01}P_{02}P_{12}P_{11}, P_{11}P_{12}P_{22}P_{21})$$ are affine symmetric with respect to their common edges if and only if $P_{00}P_{20}\parallel P_{01}P_{21}\parallel  P_{02}P_{22}$ and the lines in each of the triples $P_{01}P_{00},P_{11}P_{10},P_{21}P_{20}$ and $P_{01}P_{02},P_{11}P_{12},P_{21}P_{22}$ are parallel or concurrent. Moreover, if no two faces are coplanar, the latter two conditions %is automatic. 
    %can be dropped.
\end{proposition} 

\begin{figure}[htbp]
  \centering
\begin{tikzpicture}[scale=0.165]
\draw[black, thin] (-2.50143, 1.26455) -- (4.02081, 5.3953) -- (9.81836, 6.3374) -- (14.09405, 1.91677) -- (15.10862, -4.75041) -- (4.45563, -7.64918) -- (-2.71884, -7.72165) -- (-4.89292, -4.46053) -- (-2.50143, 1.26455) -- (3.29612, -3.80831) -- (15.10862, -4.75041);
\draw[black, thin] (-2.71884, -7.72165) -- (3.29612, -3.80831) -- (9.81836, 6.3374);

\draw[red, dashed] (-4.89292, -4.46053) -- (-7.57429, -10.91031) -- (-2.71884, -7.72165);
\draw[red, dashed] (-7.57429, -10.91031) -- (4.45563, -7.64918);
\draw[orange, dashed] (4.02081, 5.3953) -- (12.71713, 10.8305) -- (9.81836, 6.3374);
\draw[orange, dashed] (12.71713, 10.8305) -- (14.09405, 1.91677);
\draw[blue, dashed] (-4.89292, -4.46053) -- (4.45563, -7.64918);
\draw[blue, dashed] (-2.50143, 1.26455) -- (15.10862, -4.75041);
\draw[blue, dashed] (4.02081, 5.3953) -- (14.09405, 1.91677);

\filldraw[black] (-2.50143, 1.26455) circle (4pt) node[anchor=east]{$P_{10}$};
\filldraw[black] (4.02081, 5.3953) circle (4pt) node[anchor=south east]{$P_{20}$};
\filldraw[black] (9.81836, 6.3374) circle (4pt) node[anchor=west]{$P_{21}$};
\filldraw[black] (14.09405, 1.91677) circle (4pt) node[anchor=west]{$P_{22}$};
\filldraw[black] (15.10862, -4.75041) circle (4pt) node[anchor=north]{$P_{12}$};
\filldraw[black] (4.45563, -7.64918) circle (4pt) node[anchor=north west]{$P_{02}$};
\filldraw[white] (-2.5, -10) circle (45pt) node[anchor=north]{};
\filldraw[black] (-2.71884, -7.72165) circle (4pt) node[anchor=north]{$P_{01}$};
\filldraw[black] (-4.89292, -4.46053) circle (4pt) node[anchor=east]{$P_{00}$};
\filldraw[black] (-7.57429, -10.91031) circle (4pt) node[anchor=east]{};
\filldraw[black] (3.29612, -3.80831) circle (4pt) node[anchor=north]{$P_{11}$};
\filldraw[black] (12.71713, 10.8305) circle (4pt) node[anchor=west]{};

\end{tikzpicture}\qquad 
\begin{tikzpicture}[scale=0.165]
\draw[black, thin] (0.83216, 6.48234) -- (9.81836, 6.3374) -- (18.44221, 0.53985); 
\draw[black, thin] (15.10862, -4.75041) -- (6.99205, -7.79412) -- (0.2524, -7.79412) -- (-2.35649, -4.533) -- (-2.50143, 1.26455) -- (3.29612, -3.80831) -- (15.10862, -4.75041);
\draw[black, thin] (0.2524, -7.79412) -- (3.29612, -3.80831);
\draw[orange, thin] (-2.50143, 1.26455) -- (0.83216, 6.48234);
\draw[orange, thin] (18.44221, 0.53985) -- (15.10862, -4.75041);
\draw[orange, thin] (9.81836, 6.3374) -- (3.29612, -3.80831);
\draw[orange, dashed] (4.09328, 11.62767) -- (0.83216, 6.48234);
\draw[orange, dashed] (9.81836, 6.3374) -- (12.71713, 10.8305);
\draw[orange, dashed] (18.44221, 0.53985) -- (21.48593, 5.17789);

\draw[red, dashed] (-2.35649, -4.533) -- (-2.21156, -11.12771) -- (0.2524, -7.79412);
\draw[red, dashed] (-2.21156, -11.12771) -- (6.99205, -7.79412);
\draw[blue, dashed] (-2.35649, -4.533) -- (6.99205, -7.79412);
\draw[blue, dashed] (-2.50143, 1.26455) -- (15.10862, -4.75041);
\draw[blue, dashed] (0.83216, 6.48234) -- (18.44221, 0.53985);

\filldraw[black] (-2.50143, 1.26455) circle (4pt) node[anchor=north]{};
\filldraw[black] (0.83216, 6.48234) circle (4pt) node[anchor=north]{};
\filldraw[black] (9.81836, 6.3374) circle (4pt) node[anchor=south]{};
\filldraw[black] (18.44221, 0.53985) circle (4pt) node[anchor=west]{};
\filldraw[black] (15.10862, -4.75041) circle (4pt) node[anchor=south]{};
\filldraw[black] (6.99205, -7.79412) circle (4pt) node[anchor=south]{};
\filldraw[black] (0.2524, -7.79412) circle (4pt) node[anchor=east]{};
\filldraw[black] (-2.35649, -4.533) circle (4pt) node[anchor=east]{};
\filldraw[black] (-2.21156, -11.12771) circle (4pt) node[anchor=east]{};
\filldraw[black] (3.29612, -3.80831) circle (4pt) node[anchor=west]{};
%\filldraw[black] (12.71713, 10.8305) circle (4pt) node[anchor=west]{};

\end{tikzpicture}\ \ \
\begin{tikzpicture}[scale=0.161]
\draw[black, thin] (0.83216, 6.48234) -- (9.81836, 6.3374) -- (18.44221, 0.53985); 
\draw[red, thin] (15.10862, -4.75041) -- (8.51391, -9.0261);
\draw[black, thin] (8.51391, -9.0261) -- (-2.71884, -7.72165); 
\draw[red, thin] (-9.02368, -3.01114) -- (-2.50143, 1.26455);
\draw[black, thin] (-2.50143, 1.26455) -- (3.29612, -3.80831) -- (15.10862, -4.75041);
\draw[black, thin] (-2.71884, -7.72165) -- (-9.02368, -3.01114);
\draw[red, thin] (-2.71884, -7.72165) -- (3.29612, -3.80831);
\draw[orange, thin] (-2.50143, 1.26455) -- (0.83216, 6.48234);
\draw[orange, thin] (18.44221, 0.53985) -- (15.10862, -4.75041);
\draw[orange, thin] (9.81836, 6.3374) -- (3.29612, -3.80831);
\draw[orange, dashed] (4.09328, 11.62767) -- (0.83216, 6.48234);
\draw[orange, dashed] (9.81836, 6.3374) -- (12.71713, 10.8305);
\draw[orange, dashed] (18.44221, 0.53985) -- (21.48593, 5.17789);

\draw[blue, dashed] (8.51391, -9.0261) -- (-9.02368, -3.01114);
\draw[blue, dashed] (-2.50143, 1.26455) -- (15.10862, -4.75041);
\draw[blue, dashed] (0.83216, 6.48234) -- (18.44221, 0.53985);

\draw[red, dashed] (-9.02368, -3.01114) -- (-11.70504, -4.75041);
\draw[red, dashed] (-7.57429, -10.91031) -- (-2.71884, -7.72165);
\draw[red, dashed] (8.51391, -9.0261) -- (2.8613, -12.72204);

\filldraw[black] (-2.50143, 1.26455) circle (4pt) node[anchor=north]{};
\filldraw[black] (0.83216, 6.48234) circle (4pt) node[anchor=north]{};
\filldraw[black] (9.81836, 6.3374) circle (4pt) node[anchor=south]{};
\filldraw[black] (18.44221, 0.53985) circle (4pt) node[anchor=west]{};
\filldraw[black] (15.10862, -4.75041) circle (4pt) node[anchor=south]{};
\filldraw[black] (8.51391, -9.0261) circle (4pt) node[anchor=south]{};
\filldraw[black] (-2.71884, -7.72165) circle (4pt) node[anchor=east]{};
\filldraw[black] (-9.02368, -3.01114) circle (4pt) node[anchor=east]{};
%\filldraw[black] (-7.57429, -10.91031) circle (4pt) node[anchor=east]{};
\filldraw[black] (3.29612, -3.80831) circle (4pt) node[anchor=west]{};
%\filldraw[black] (12.71713, 10.8305) circle (4pt) node[anchor=west]{};

\end{tikzpicture}
\caption{A geometric characterization of deformable $2\times 2$ nets of class (i) in Theorem~\ref{th-classification}. Blue lines are parallel. Red and orange triples of lines are either concurrent or parallel. See Proposition~\ref{lemma-three-pic}.}
\label{figure:theorem(i)}
\end{figure}
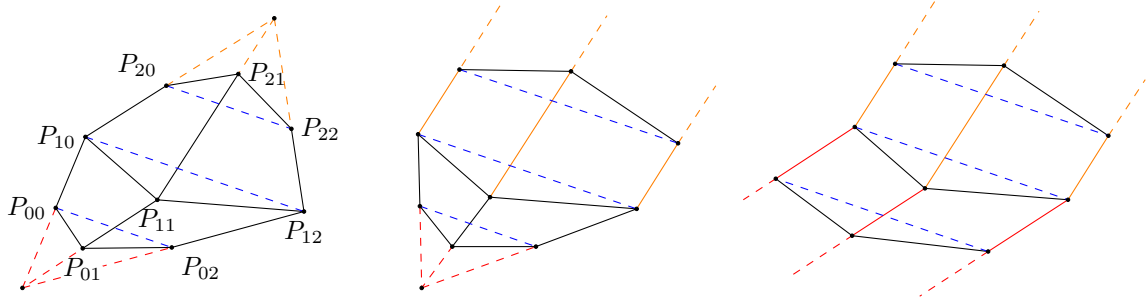

%Yet another characterization, making the analogy between classes~(i) and~(ii) transparent: two quadrilaterals %$ABCD$ and $ABC'D'$ are affine symmetric with respect to the common side if and only if their opposite ratios with respect to respective sides are equal.

\subsubsection{Class~(ii)}
\label{subs-class(ii)}

The geometric description of class~(ii) is less straightforward. 
To motivate it, we first 
briefly discuss a necessary condition for %deformability,  
a $2\times 2$ net to be deformable. 
This condition is actually the ``infinitesimal %(or first-order) 
deformability'' studied in \cite[\S6.2]{isometric-isotropic}; cf.~\cite[\S2.2]{karpenkov-2010} and~\cite[\S2]{Schief2008}.

Let $P_{ij}(t)$ be a deformation of a $2\times 2$ net $P_{ij}$, where $0\le i,j\le 2$. Identify points in $\mathbb{R}^d$ with vectors. Assume that the derivatives $P_{ij}'(0)$ exist and do not all coincide. Consider the deformations $P_{ij}(t)P_{i+1,j}(t)P_{i+1,j+1}(t)P_{i,j+1}(t)$ of a particular face. These quadrilaterals have parallel sides and equal areas for all $t$. In particular, the derivative of the area at $t=0$ vanishes. % (this is how one can understand an ``infinitesimal deformation''). 
By a known computation \cite[Theorem~13 and Eq.~(3)]{bobenko2010curvature}, the latter is equivalent to quadrilaterals $P_{ij}(0)P_{i+1,j}(0)P_{i+1,j+1}(0)P_{i,j+1}(0)$ and $P_{ij}'(0)P_{i+1,j}'(0)P_{i+1,j+1}'(0)P_{i,j+1}'(0)$
being \emph{dual}, i.e., having parallel corresponding sides and parallel non-corresponding diagonals (see Figure~\ref{figure:dual-of-2-quads}):
$$
P_{ij}(0)P_{i+1,j+1}(0)\parallel P_{i+1,j}'(0)P_{i,j+1}'(0),
\qquad
P_{i+1,j}(0)P_{i,j+1}(0)\parallel
P_{ij}'(0)P_{i+1,j+1}'(0).
$$

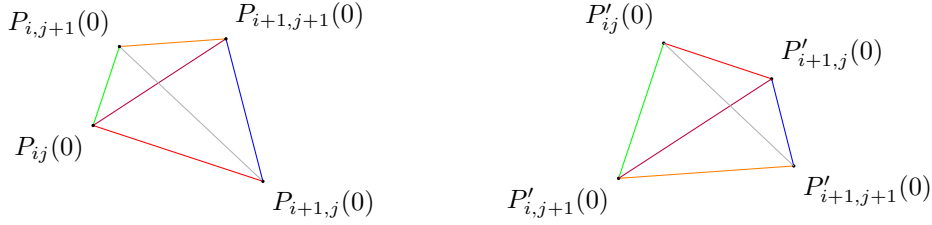
\begin{figure}[!t]%[htbp]
  \centering
\begin{tikzpicture}[scale=0.035]%[scale=0.05]

\coordinate (A) at (54.48536, -61.2148);
\coordinate (B) at (40.45495, -7.05598);
\coordinate (C) at (0, -10);
\coordinate (D) at (-10, -40);

\draw[blue, thin] (A) -- (B);
\draw[orange, thin] (B) -- (C);
\draw[green, thin] (C) -- (D);
\draw[red, thin] (D) -- (A);

\draw[purple, name path=line 1] (B) -- (D);
\draw[gray!60, name path=line 2] (A) -- (C);

\filldraw[black] (A) circle (12pt) node[anchor=north west]{$P_{i+1,j}(0)$};
\filldraw[black] (B) circle (12pt) node[anchor=south west]{$P_{i+1,j+1}(0)$};
\filldraw[black] (C) circle (12pt) node[anchor=south east]{$P_{i,j+1}(0)$};
\filldraw[black] (D) circle (12pt) node[anchor=north east]{$P_{ij}(0)$};

\end{tikzpicture}
\qquad\qquad
\begin{tikzpicture}[scale=0.06]%[scale=0.087]

\coordinate (C) at (0, -10);
\coordinate (D) at (-10, -40);
\coordinate (B') at (23.91368, -17.93708);
\coordinate (A') at (28.81781, -37.24708);

\draw[blue, thin] (A') -- (B');
\draw[orange, thin] (A') -- (D);
\draw[green, thin] (C) -- (D);
\draw[red, thin] (C) -- (B');

\filldraw[black] (C) circle (7pt) node[anchor=south east]{$P'_{ij}(0)$};
\filldraw[black] (D) circle (7pt) node[anchor=north east]{$P'_{i,j+1}(0)$};
\filldraw[black] (B') circle (7pt) node[anchor=south west]{$P'_{i+1,j}(0)$};
\filldraw[black] (A') circle (7pt) node[anchor=north west]{$P'_{i+1,j+1}(0)$};

\draw[purple, name path=line 1, thin] (B') -- (D);
\draw[gray!60, name path=line 2, thin] (A') -- (C);

\end{tikzpicture}
\caption{Dual quadrilaterals. Corresponding sides (with the same color) are parallel and non-corresponding diagonals (also with the same color) are parallel. See Subsection~\ref{subs-class(ii)}.}
\label{figure:dual-of-2-quads}
\end{figure}

This holds for all faces, hence points $P_{ij}'(0)$ form a $2\times 2$ net, and the nets $P_{ij}$ and $P_{ij}'(0)$ %(the so-called \emph{velocity diagram}) 
are \emph{Christoffel duals}, i.e., their corresponding faces are dual. See Figure~\ref{fig:christ-dual}. %A Christoffel dual exists only for so-called K\oe nigs nets, which we define in a moment. 
%In particular, $P_{ij}'(0)$ is indeed a $2\times 2$ net.
A net admitting a Christoffel dual is called a \emph{K\oe nigs} net. We conclude that an (``infinitesimally'') deformable $2\times 2$ net is a K\oe nigs~net. 
%(Passing to the \emph{second-order} infinitesimal deformability would not add much because the vanishing of the first two derivatives of the face areas at $t=0$ already implies constant face areas \cite[Eq.~(3)]{bobenko2010curvature}. --- WRONG!!!)

\begin{figure}[htbp]
\centering

\begin{tikzpicture}[scale=1.7]

\coordinate (P) at (0.141595, 0.064912);
\coordinate (P') at (0.833477, 0.183469);
\coordinate (D) at (2.059665, 0.087001);
\coordinate (C) at (0.009566, 1.019064);
\coordinate (B) at (0.975775, 0.830942);
\coordinate (C') at (2.992722, 1.052864);
\coordinate (D') at (0.183345, 2.11201);
\coordinate (A) at (0.863328, 1.696711);
\coordinate (E) at (2.113132, 2.121726);

\fill[lime!8] (B) -- (C) -- (D') -- (A) -- cycle;
\fill[britishracinggreen!5] (B) -- (A) -- (E) -- (C') -- cycle;
\fill[antiquebrass!8] (B) -- (C') -- (D) -- (P') -- cycle;
\fill[babyblue!5] (B) -- (P') -- (P) -- (C) -- cycle;

\draw[black, thin] (D') -- (A) -- (E);
\draw[black, thin] (D) -- (P') -- (P);
\draw[black, thin] (C) -- (B) -- (C');
\draw[red, thin] (P) -- (C);
\draw[blue, thin] (C) -- (D');
\draw[orange, thin] (E) -- (C');
\draw[green, very thin] (C') -- (D);
\draw[gray!60, thin] (P') -- (B);
\draw[purple, thin] (B) -- (A);

\filldraw[black] (P) circle (0.5pt) node[anchor=north east]{$P_{00}$};
\filldraw[black] (P') circle (0.5pt) node[anchor=north]{$P_{10}$};
\filldraw[black] (D) circle (0.5pt) node[anchor=north west]{$P_{20}$};
\filldraw[black] (C) circle (0.5pt) node[anchor=east]{$P_{01}$};
\filldraw[black] (B) circle (0.5pt) node[anchor=north west]{$P_{11}$};
\filldraw[black] (C') circle (0.5pt) node[anchor=west]{$P_{21}$};
\filldraw[black] (D') circle (0.5pt) node[anchor=south east]{$P_{02}$};
\filldraw[black] (A) circle (0.5pt) node[anchor=south]{$P_{12}$};
\filldraw[black] (E) circle (0.5pt) node[anchor=south west]{$P_{22}$};

%\fill[lime!8,path fading=west] (Q) -- (A1t) -- (B1t) -- (A2t) -- cycle;
%\fill[britishracinggreen!8,path fading=south] (Q) -- (A1t) -- (B1t) -- (A2t) -- cycle;

\end{tikzpicture}
\qquad\qquad
\begin{tikzpicture}[scale=1.18]%[scale=1,rotate=90]

\coordinate (P) at (-3.38159, 1.0081);
\coordinate (P') at (-2.02127, 0.1827);
\coordinate (D) at (0.58246, 1.06833);
\coordinate (C) at (-3.20801, 2.09565);
\coordinate (B) at (-2.24445, 1.91144);
\coordinate (C') at (-0.29254, 2.13816);
\coordinate (D') at (-3.33554, 3.02024);
\coordinate (A) at (-1.9398, 3.26467);
\coordinate (E) at (0.586, 3.04858);

\fill[babyblue!5] (B) -- (C) -- (D') -- (A) -- cycle;
\fill[antiquebrass!8] (B) -- (A) -- (E) -- (C') -- cycle;
\fill[britishracinggreen!5] (B) -- (C') -- (D) -- (P') -- cycle;
\fill[lime!8]  (B) -- (P') -- (P) -- (C) -- cycle;

\draw[black, thin] (D') -- (A) -- (E);
\draw[black, thin] (D) -- (P') -- (P);
\draw[black, thin] (C) -- (B) -- (C');
\draw[gray!60, thin] (B) -- (A);
\draw[purple, thin] (P') -- (B);
\draw[blue, thin] (P) -- (C);
\draw[red, thin] (C) -- (D');
\draw[green, very thin] (E) -- (C');
\draw[orange, thin] (C') -- (D);

\filldraw[black] (P) circle (0.65pt) node[anchor=north east]{$P_{02}'$};
\filldraw[black] (P') circle (0.65pt) node[anchor=north]{$P_{12}'$};
\filldraw[black] (D) circle (0.65pt) node[anchor=north west]{$P_{22}'$};
\filldraw[black] (C) circle (0.65pt) node[anchor=east]{$P_{01}'$};
\filldraw[black] (B) circle (0.65pt) node[anchor=north west]{$P_{11}'$};
\filldraw[black] (C') circle (0.65pt) node[anchor=west]{$P_{21}'$};
\filldraw[black] (D') circle (0.65pt) node[anchor=south east]{$P_{00}'$};
\filldraw[black] (A) circle (0.65pt) node[anchor=south]{$P_{10}'$};
\filldraw[black] (E) circle (0.65pt) node[anchor=south west]{$P_{20}'$};

\end{tikzpicture}
\caption{Two $2\times2$ nets that are Christoffel duals. Corresponding faces (of the same color) are~dual.}
\label{fig:christ-dual}
\end{figure}
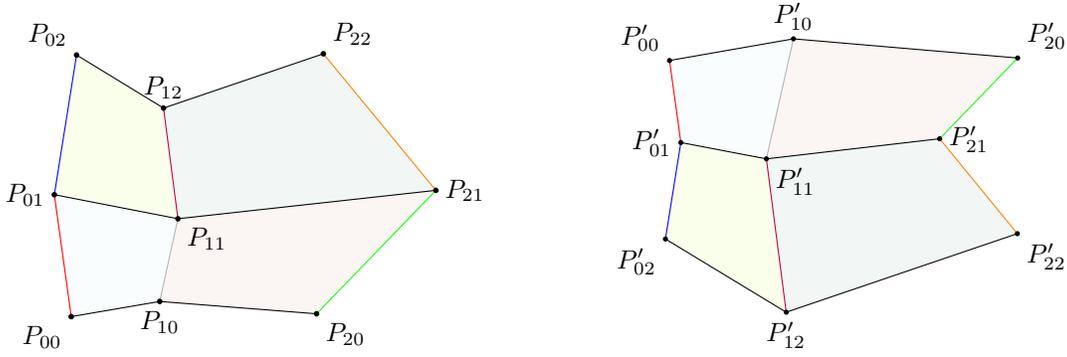

%It is time for precise statements. 
Several geometric characterizations of K\oe nigs nets are known. See Figure~\ref{figure:prop-inter-3-lines}. A $2\times2$ net with the vertices $P_{ij}$, where $0\leq i,j\leq2$, is a K\oe nigs net, if and only if
    \begin{equation}
        \label{eq-three-lines2}
        \frac{P_{10}Q_{00}}{Q_{00}P_{01}}\cdot\frac{P_{01}Q_{01}}{Q_{01}P_{12}}\cdot\frac{P_{12}Q_{11}}{Q_{11}P_{21}}\cdot\frac{P_{21}Q_{10}}{Q_{10}P_{10}}=1,
    \end{equation}
where $Q_{ij}:=P_{ij}P_{i+1,j+1}\cap P_{i+1,j}P_{i,j+1}$
%is the intersection point of diagonals of the face $P_{ij}P_{i+1,j}P_{i+1,j+1}P_{i,j+1}$ 
for $0\leq i,j\leq1$ \cite[Theorem~2.25]{bobenko-2008-ddg}. If no three points among $P_{01}, P_{12}, P_{21}, P_{10}$ are collinear, then~\eqref{eq-three-lines2} is equivalent to the three lines $Q_{00}Q_{10}, P_{01}P_{21}, Q_{01}Q_{11}$ %(or equivalently $Q_{00}Q_{01}, P_{10}P_{12}, Q_{10}Q_{11}$), 
being concurrent or parallel \cite[Theorem~9.12]{bobenko-2008-ddg}. If the points $P_{01}, P_{12}, P_{21}, P_{10}$ are not coplanar, then the latter condition 
is equivalent to the coplanarity of $Q_{00}, Q_{01}, Q_{10}, Q_{11}$~\cite[Theorem~2.26]{bobenko-2008-ddg}. 
%%%
%A $2\times2$ net with the vertices $P_{ij}$, where $0\leq i,j\leq2$, is a \emph{K\oe nigs net}, if the three lines $Q_{00}Q_{10}, P_{01}P_{21}, Q_{01}Q_{11}$ %(or equivalently $Q_{00}Q_{01}, P_{10}P_{12}, Q_{10}Q_{11}$), 
%%are either concurrent or parallel, where $Q_{ij}:=P_{ij}P_{i+1,j+1}\cap P_{i+1,j}P_{i,j+1}$
%%is the intersection point of diagonals of the face $P_{ij}P_{i+1,j}P_{i+1,j+1}P_{i,j+1}$ 
%for $0\leq i,j\leq1$. See Figure~\ref{figure:prop-inter-3-lines}. If the points $P_{01}, P_{12}, P_{21}, P_{10}$ are not coplanar, then the latter condition is equivalent to the coplanarity of $Q_{00}, Q_{01}, Q_{10}, Q_{11}$~\cite[Theorem~2.26]{bobenko-2008-ddg}. 
%%%

\begin{proposition} \label{p-Koenigs}
A $2\times2$ net satisfying one of conditions~(i)--(ii) of Theorem~\ref{th-classification} is a K\oe nigs~net.
%If a $2\times2$ net satisfies one of the conditions~(i) or (ii) of Theorem~\ref{th-classification}, then it is a K\oe nigs net.
%Let a $2\times2$ net with the vertices $P_{ij}$, where $0\leq i,j\leq2$, satisfy one of the conditions~(i) or (ii) of Theorem~\ref{th-classification}. Let $Q_{ij}$ be the intersection point of diagonals of the face $P_{ij}P_{i+1,j}P_{i+1,j+1}P_{i,j+1}$ for $0\leq i,j\leq1$. Then the lines in each of the triples $Q_{00}Q_{10}, P_{01}P_{21}, Q_{01}Q_{11}$ and $Q_{00}Q_{01}, P_{10}P_{12}, Q_{10}Q_{11}$ are either concurrent or parallel. Moreover, if the points $P_{01}, P_{12}, P_{21}, P_{10}$ are not coplanar, then the latter condition is equivalent to the coplanarity of the points $Q_{00}, Q_{01}, Q_{10}, Q_{11}$. 
\end{proposition}

We give direct elementary proof relying neither on the infinitesimal deformability nor the duality.

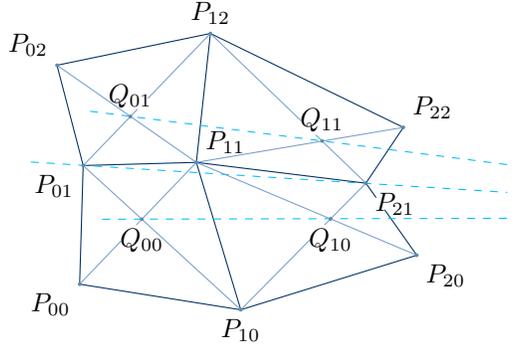
\begin{figure}[htbp]
  \centering
\begin{tikzpicture}[scale=0.7]%[scale=0.8]

\coordinate (B3) at (1.77675, -17.15178);
\coordinate (A3) at (4.83651, -17.6383);
\coordinate (B2) at (8.17937, -16.60145);
\coordinate (A4) at (1.84386, -14.89677);
\coordinate (Q) at (3.99291, -14.8369);
\coordinate (A2) at (7.21294, -15.23233);
\coordinate (B4) at (1.34722, -12.99074);
\coordinate (A1) at (4.25995, -12.38672);
\coordinate (B1) at (7.92434, -14.17194);

\coordinate (Q1) at (2.95795, -15.91689);
\coordinate (Q2) at (2.7398, -13.965);
\coordinate (Q3) at (6.38073, -14.43);
\coordinate (Q4) at (6.54181, -15.91);

%Endpoints of three lines
\coordinate (X1) at (1.97618, -13.86471);
\coordinate (Y1) at (10.13909, -14.91019);
\coordinate (X2) at (0.85058, -14.82965);
\coordinate (Y2) at (9.89748, -15.40683);
\coordinate (X3) at (2.20628, -15.91689);
\coordinate (Y3) at (10.0854, -15.899);

\draw[coolblack, thin] (A1) -- (Q) -- (A2) -- (B2) -- (A3) -- (B3) -- (A4) -- (B4) -- (A1) -- (B1) -- (A2);
\draw[coolblack, thin] (A4) -- (Q) -- (A3);
\draw[deepskyblue, thin, dashed] (X1) -- (Y1);
\draw[deepskyblue, thin, dashed] (X2) -- (Y2);
\draw[deepskyblue, thin, dashed] (X3) -- (Y3);

\draw[darkpastelblue,name path=line 1] (Q) -- (B4);
\draw[darkpastelblue,name path=line 2] (A1) -- (A4);
%\tkzInterLL(Q,B4)(A1,A4) \tkzGetPoint{P4}
%\node [label=90:$Q_{01}$] at (P4) {};

\draw[darkpastelblue,name path=line 3] (Q) -- (B1);
\draw[darkpastelblue,name path=line 4] (A1) -- (A2);
%\tkzInterLL(Q,B1)(A1,A2) \tkzGetPoint{P1}
%\node [label=90:$Q_{11}$] at (P1) {};

\draw[darkpastelblue,name path=line 5] (Q) -- (B2);
\draw[darkpastelblue,name path=line 6] (A2) -- (A3);
%\tkzInterLL(Q,B2)(A2,A3) \tkzGetPoint{P2}
%\node [label=-90:$Q_{10}$] at (P2) {};

\draw[darkpastelblue,name path=line 7] (Q) -- (B3);
\draw[darkpastelblue,name path=line 8] (A3) -- (A4);
%\tkzInterLL(Q,B3)(A3,A4) \tkzGetPoint{P3}
%\node [label=-90:$Q_{00}$] at (P3) {};

\filldraw[airforceblue] (Q) circle (0.7pt); 
\filldraw[black] (Q) circle (0pt) node[anchor=south west]{$P_{11}$};
\filldraw[airforceblue] (A2) circle (0.7pt);
\filldraw[white] (7.45, -15.65) circle (4pt);
\filldraw[black] (A2) circle (0pt) node[anchor=north west]{$P_{21}$};
\filldraw[airforceblue] (A3) circle (0.7pt);
\filldraw[black] (A3) circle (0pt) node[anchor=north]{$P_{10}$};
\filldraw[airforceblue] (A4) circle (0.7pt);
\filldraw[black] (A4) circle (0pt) node[anchor=north east]{$P_{01}$};
\filldraw[airforceblue] (A1) circle (0.7pt);
\filldraw[black] (A1) circle (0pt) node[anchor=south]{$P_{12}$};

\filldraw[airforceblue] (B1) circle (0.7pt);
\filldraw[black] (B1) circle (0pt) node[anchor=south west]{$P_{22}$};
\filldraw[airforceblue] (B2) circle (0.7pt);
\filldraw[black] (B2) circle (0pt) node[anchor=north west]{$P_{20}$};
\filldraw[airforceblue] (B3) circle (0.7pt);
\filldraw[black] (B3) circle (0pt) node[anchor=north east]{$P_{00}$};
\filldraw[airforceblue] (B4) circle (0.7pt);
\filldraw[black] (B4) circle (0pt) node[anchor=south east]{$P_{02}$};

\filldraw[airforceblue] (Q1) circle (0.7pt);
\filldraw[white] (2.79, -16.26) circle (6.5pt);
\filldraw[black] (Q1) circle (0pt) node[anchor=north]{$Q_{00}$};
\filldraw[airforceblue] (Q2) circle (0.7pt);
\filldraw[white] (3, -13.75) circle (4pt);
\filldraw[black] (Q2) circle (0pt) node[anchor=south]{$Q_{01}$};
\filldraw[airforceblue] (Q3) circle (0.7pt);
\filldraw[white] (6.2, -14.1) circle (6pt);
\filldraw[black] (Q3) circle (0pt) node[anchor=south]{$Q_{11}$};
\filldraw[airforceblue] (Q4) circle (0.7pt);
\filldraw[white] (6.38, -16.25) circle (6.3pt);
\filldraw[black] (Q4) circle (0pt) node[anchor=north]{$Q_{10}$};

%\filldraw[black] (P1) circle (0.5pt) node[anchor=north west]{$P_1=Q_{11}$};
%\filldraw[black] (P2) circle (0.5pt) node[anchor=north west]{$P_2=Q_{10}$};
%\filldraw[black] (P3) circle (0.5pt) node[anchor=north east]{$P_3=Q_{00}$};
%\filldraw[black] (P4) circle (0.5pt) node[anchor=south east]{$P_4=Q_{01}$};

%\filldraw[black] (P1) circle (0.5pt);
%\filldraw[black] (P2) circle (0.5pt);
%\filldraw[black] (P3) circle (0.5pt);
%\filldraw[black] (P4) circle (0.5pt);

%\draw[thin, dashed] (P4) -- (P1);
%\draw[thin, dashed] (P1) -- ++(1,0.013);
%\draw[thin, dashed] (P4) -- ++(-1,-0.013);

%\draw[thin, dashed] (A4) -- (A2);
%\draw[thin, dashed] (A2) -- ++(0.35,0.00462);
%\draw[thin, dashed] (A4) -- ++(-0.75,-0.01);

%\draw[thin, dashed] (P3) -- (P2);
%\draw[thin, dashed] (P2) -- ++(0.8,0.0104);
%\draw[thin, dashed] (P3) -- ++(-1,-0.013);

\end{tikzpicture}\\
\caption{Three dashed lines for a deformable $2\times 2$ net are either concurrent or parallel.}
\label{figure:prop-inter-3-lines}
\end{figure}

\begin{proof} %%Without loss of generality, it is enough to show that the lines $Q_{00}Q_{10}, P_{01}P_{21}, Q_{01}Q_{11}$ are either concurrent or parallel. 
%It suffices to prove that (see Figure~\ref{figure:prop-inter-3-lines}) %the following equality
%    \begin{equation}
%        \label{eq-three-lines2}
%        \frac{P_{10}Q_{00}}{Q_{00}P_{01}}\cdot\frac{P_{01}Q_{01}}{Q_{01}P_{12}}\cdot\frac{P_{12}Q_{11}}{Q_{11}P_{21}}\cdot\frac{P_{21}Q_{10}}{Q_{10}P_{10}}=1.
%    \end{equation}
%Then we will get a K\oe nigs net by the generalized Menelaus theorem~\cite[Theorem~9.12]{bobenko-2008-ddg}.
%%, and the coplanarity will follow from~\cite[Theorem~2.26]{bobenko-2008-ddg}.
%
Enumerate the faces containing the points $Q_{00},Q_{01},Q_{11},Q_{10}$ in the listed order so that the indices are cyclic modulo $4$.
%so that $i$-th and $(i-1)$-th face have a common edge for $i = 1, 2, 3, 4$, and the indices are cyclic modulo $4$. 
Denote by $r(i,j)$ the opposite ratio of $i$-th face with respect to its common edge with $j$-th face. Then ${P_{10}Q_{00}}/{Q_{00}P_{01}}=\sqrt{r(1,4)/r(1,2)}$ by a direct computation (see~\eqref{eq-ratio-from-opposite-ratio} below).
Substituting similar expressions for the other factors in~\eqref{eq-three-lines2}, we rewrite it in an equivalent form:
\begin{equation}
    \label{eq:rel-opp-rat=1}
        \frac{r(1,2)r(2,3)r(3,4)r(4,1)}{r(1,4)r(2,1)r(3,2)r(4,3)}=1.
\end{equation}

If the net is from class~(i), then $r(i,j-1)=r(i-1,j)$ either for all even $i,j$ or for all odd $i,j$, because affine symmetric quadrilaterals have equal opposite ratios with respect to corresponding edges.

If the net is from class~(ii), then $r(i,j)=r(j,i)$ for all $i,j$ different by $1$.

In both cases, the numerator and the denominator of~\eqref{eq:rel-opp-rat=1} cancel out completely, as required.
\end{proof}

To describe class~(ii) geometrically, we return to the duality. Notice that a Christoffel dual of a $2\times 2$ net is one of its Comberscure transformations. Now we show the following: if this particular Comberscure transformation is area-preserving, then there is a whole family of area-preserving Comberscure transformations, and this happens exactly for nets from class~(ii). 

\begin{proposition} \label{ex-ii-geo}
    If a $2\times 2$ net $P_{ij}$ ($0\le i,j\le 2$) has a Christoffel dual $P_{ij}^*$ ($0\le i,j\le 2$) with the same areas of corresponding faces, then
    \begin{equation} \label{eq-ii-geo}
    P_{ij}(t)=P_{ij}\cosh t +P_{ij}^*\sinh t,
    \qquad t\in [0,\varepsilon], 0\le i,j\le 2,
    \end{equation}
    is a family of area-preserving Comberscure transformations for sufficiently small $\varepsilon>0$. % (here points are viewed as vectors in~$\mathbb{R}^d$). 
\end{proposition}

\begin{proof}
Clearly, the %$2\times 2$ 
nets $P_{ij}$, $P_{ij}^*$, and $P_{ij}(t)$ are parallel. Let $f$, $f^*$, and $f(t)$ be their corresponding faces. Dual quadrilaterals $f$ and $f^*$ have opposite orientation, hence opposite oriented areas $\mathrm{Area}(f^*)=-\mathrm{Area}(f)$. Their mixed area $\mathrm{Area}(f,f^*)=0$ by \cite[Theorem~13]{bobenko2010curvature}. Then by \cite[Eq.~(3)]{bobenko2010curvature}, we get
\begin{align*}
    \mathrm{Area}(f(t))=\mathrm{Area}(f)\cosh^2 t +\mathrm{Area}(f^*)\sinh^2 t +2\mathrm{Area}(f,f^*)\cosh t\sinh t =\mathrm{Area}(f)=\mathrm{const}.
    %\qquad\text{as required.}
\end{align*}
\end{proof}

In this construction, opposite ratios arise as follows.

\begin{proposition} \label{p-dual-with-equal-areas}
    If two dual convex quadrilaterals $f$ and $f^*$ have equal areas, then the ratio of their corresponding sides $e$ and $e^*$ equals the square root of the opposite ratio of $f$ with respect to $e$.
\end{proposition}

\begin{proof}
%This is a simple computation in the spirit of
(Cf.~\cite[Proof of 
Lemma~2.20]{bobenko-2008-ddg}.) Let $A,B,C,D$ be consecutive vertices of $f$ so that $AB=e$, let $Q$
be the intersection point of the diagonals, $e_1$ and $e_2$ be some vectors along the diagonals,
$$
\overrightarrow{QA}=\alpha e_1, \quad
\overrightarrow{QB}=\beta e_2, \quad
\overrightarrow{QC}=\gamma e_1, \quad
\overrightarrow{QD}=\delta e_2.
$$
Since $f$ is convex, we may assume that $\alpha\beta\gamma\delta=1$ without loss of generality.

Construct a quadrilateral $A^*B^*C^*D^*$ with the intersection of the diagonals $Q^*$ by setting 
$$
\overrightarrow{Q^*A^*}=-\frac{e_2}{\alpha}, \quad
\overrightarrow{Q^*B^*}=-\frac{e_1}{\beta}, \quad
\overrightarrow{Q^*C^*}=-\frac{e_2}{\gamma}, \quad
\overrightarrow{Q^*D^*}=-\frac{e_1}{\delta}
\qquad\text{\cite[Eq.~(2.28)]{bobenko-2008-ddg}}.
$$
The resulting quadrilateral is dual to $f$ and has the same area because $\alpha\beta\gamma\delta=1$.
Since $f^*$ is unique up to scaling and translation \cite[Lemma~2.20]{bobenko-2008-ddg}, we may assume that 
$f^*=A^*B^*C^*D^*$ so that $e^*=A^*B^*$. 

By the similarity of triangles $ABQ$ and $A^*B^*Q^*$, the ratio of $e$ and $e^*$ equals $AB/A^*B^*=|\alpha\beta|$. The opposite ratio of $f$ with respect to $e$ is $|\alpha\beta|/|\gamma\delta|=|\alpha\beta|^2$, as required.
\end{proof}

As a consequence, we get the following %geometric 
characterization of class~(ii). See Figures~\ref{fig-defs-i-ii}(right) and~\ref{fig:christ-dual}.
%~\mscomm{We need to illustrate this with a figure.}.

\begin{proposition}
    \label{p-class-ii}
    For a $2\times 2$ net, the following two conditions are equivalent: 
\begin{itemize}
    \item each pair of faces with a common edge has equal opposite ratios with respect to that edge;
    \item the net has a Christoffel dual with the same areas of corresponding faces.
\end{itemize}
\end{proposition}

\begin{proof}
    Each face has a dual quadrilateral of the same area \cite[Lemma~2.20]{bobenko-2008-ddg}. Performing a central symmetry, if necessary, we may assume that for any edge of the form $P_{ij}P_{i+1,j}$, the corresponding oriented side of the dual quadrilateral has the same direction, and for any edge of the form $P_{ij}P_{i,j+1}$, the corresponding oriented side has the opposite direction. See Figure~\ref{figure:dual-of-2-quads}. The resulting dual quadrilaterals are unique up to translation. By Proposition~\ref{p-dual-with-equal-areas}, their oriented sides fit to compose a whole Christoffel dual net if and only if the required opposite ratios are equal. % in the initial faces.
\end{proof}

Propositions~\ref{ex-ii-geo} and~\ref{p-class-ii} show that class~(ii) indeed consists of deformable nets.

%Notice that a Christoffel dual of a $2\times 2$ net is its Comberscure transformation. Proposition~\ref{p-class-ii} and Example~\ref{ex-ii-geo} imply that if this particular Comberscure transformation is area-preserving, then there is a whole family of area-preserving Comberscure transformations.

\subsubsection{Opposite ratios}
%Concerning class~(ii), the definition in terms of opposite ratios seems to be the most visual. Thus we discuss the latter notion in greater detail. 
Let us discuss further properties of the opposite ratio. See Figure~\ref{figure:def-oppos-rat}.
%Let us give alternative expressions for the opposite ratio.

%There are alternative formulae for the opposite ratio, which are sometimes helpful. 

\begin{proposition}
    \label{p-opposite-ratio-equivalent-def}
    If the rays $BC$ and $AD$ extending the sides of a convex quadrilateral $ABCD$
    meet at a point $S$ and the diagonals meet at $Q$, then the opposite ratio of $ABCD$ with respect to $AB$~is %$(1-S(ABCD)/S(ABS))^{-1}$.
    \begin{equation}\label{eq-opposite-ratio-equivalent-def}
    \frac{\mathrm{Area}(AQB)}{\mathrm{Area}(CQD)}=\frac{AQ\cdot BQ}{CQ\cdot DQ}
    =\frac{AS\cdot BS}{CS\cdot DS}=\frac{\mathrm{Area}(ASB)}{\mathrm{Area}(CSD)}
    =\left(1-\frac{\mathrm{Area}(ABCD)}{\mathrm{Area}(ABS)}\right)^{-1}.
    \end{equation}
\end{proposition}

\begin{proof} 
    %Let $F,E$ be points in $AC,BD$, respectively such that $CE\parallel AS$ and $DF\parallel SB$. %See Figure~\ref{figure:properties2}. Using similarity of triangles $\triangle ASC\sim\triangle ADF$, $\triangle BSD\sim\triangle BCE$, $\triangle PDF\sim\triangle PBC$ and $\triangle PCE\sim\triangle PAD$, we get $AS/SC=AD/DF$, $SB/SD=BC/CE$, $BC/DF=BP/PD$ and $AD/CE=AP/PC$, respectively. Then, using these equalities, we get 
    %$$\left(1-\frac{S(ABCD)}{S(ABS)}\right)^{-1}=\frac{S(ABS)}{S(CDS)}=\frac{AS}{SC}\cdot\frac{SB}{SD}=\frac{AD}{DF}\cdot\frac{BC}{CE}=\frac{AD}{CE}\cdot\frac{BC}{DF}=\frac{AP}{PC}\cdot\frac{BP}{PD}=\frac{S(APB)}{S(CDP)}.$$   
    The second equality follows from Menelaus' theorem for the triangles $ADQ$ and $BCQ$: %with points $B,C,S$ and $A,D,S$, respectively, 
    $$\frac{AS}{DS}=\frac{BQ}{BD}\cdot\frac{CA}{CQ}\quad\text{and}\quad\frac{BS}{CS}=\frac{AQ}{AC}\cdot\frac{DB}{DQ}.$$
    The other three equalities are straightforward. 
    %
    %Then, using these equalities, we obtain the desired expressions
    %$$\left(1-\frac{S(ABCD)}{S(ABS)}\right)^{-1}=\frac{S(ABS)}{S(CDS)}=\frac{AS}{SC}\cdot\frac{SB}{SD}=\left(\frac{BP}{DB}\cdot\frac{AC}{PC}\right)\cdot\left(\frac{AP}{CA}\cdot\frac{BD}{PD}\right)=\frac{AP}{PC}\cdot\frac{BP}{PD}=\frac{S(APB)}{S(CDP)}.$$
\end{proof}

Thus, in particular, %this implies that 
rays $BC$ and $AD$ intersect if and only if this opposite ratio is greater than~$1$.

Further, just like~\eqref{eq-opposite-ratio-equivalent-def} expresses the opposite ratio in terms of certain length ratios, those length ratios themselves can be expressed in terms of opposite ratios:
\begin{equation}\label{eq-ratio-from-opposite-ratio}
\frac{AQ}{CQ}=\sqrt{\frac{\mathrm{Area}(AQB)}{\mathrm{Area}(CQD)}\cdot\frac{\mathrm{Area}(DQA)}{\mathrm{Area}(BQC)}}\quad\text{and}\quad
\frac{BQ}{DQ}=\sqrt{\frac{\mathrm{Area}(AQB)}{\mathrm{Area}(CQD)}\cdot\frac{\mathrm{Area}(BQC)}{\mathrm{Area}(DQA)}}.
%\frac{AP}{CP}=\sqrt{\frac{S(APB)}{S(CPD)}\cdot\frac{S(DPA)}{S(BPC)}}\quad\text{and}\quad
%\frac{BP}{DP}=\sqrt{\frac{S(APB)}{S(CPD)}\cdot\frac{S(BPC)}{S(DPA)}}.
\end{equation}
This shows that two opposite ratios of quadrilateral with respect to two adjacent sides determine the quadrilateral uniquely up to affine transformations: given vertices $A, B, C$, the first equation uniquely determines $Q$, and the second one uniquely determines~$D$. We arrive at the following corollary.  

\begin{corollary} \label{cor-existence-uniqueness-opposite-ratios}
There exists a unique convex quadrilateral $ABCD$ with given (non-collinear) vertices $A,B,C$ and given (positive) opposite ratios with respect to the sides $AB$ and $BC$.
\end{corollary}

%\begin{proof} We uniquely determine $P$ and then $D$ using~\eqref{eq-ratio-from-opposite-ratio}.
%\end{proof}

This suggests the following construction of deformable $2\times 2$ nets. A pair of opposite faces $P_{00}P_{01}P_{11}P_{10}$ and $P_{11}P_{12}P_{22}P_{21}$ can be prescribed arbitrarily as long as %they have the unique common vertex $P_{11}$ and 
the triples $P_{01},P_{11},P_{12}$ and $P_{10},P_{11},P_{21}$ are non-collinear. 
Then the remaining pair of faces is uniquely determined by condition~(ii) of Theorem~\ref{th-classification}, leading to a unique $2\times 2$ net from class (ii). %(unless some of the vertices of the resulting mesh coincide, which is not allowed, or one of the triples $P_{01},P_{11},P_{12}$ and $P_{10},P_{11},P_{21}$ is collinear). 
Analogously, the pair of opposite faces leads to at most two $2\times 2$ nets from class~(i), depending on which pairs of faces are affine symmetric.

\subsection{Proof of the classification}
\label{ssec-proof-deformable-2x2}

First, we reduce the classification of deformable $2\times 2$ nets to solving a system of quadratic equations. We need the following notation.
See Figure~\ref{figure:notation} (left).

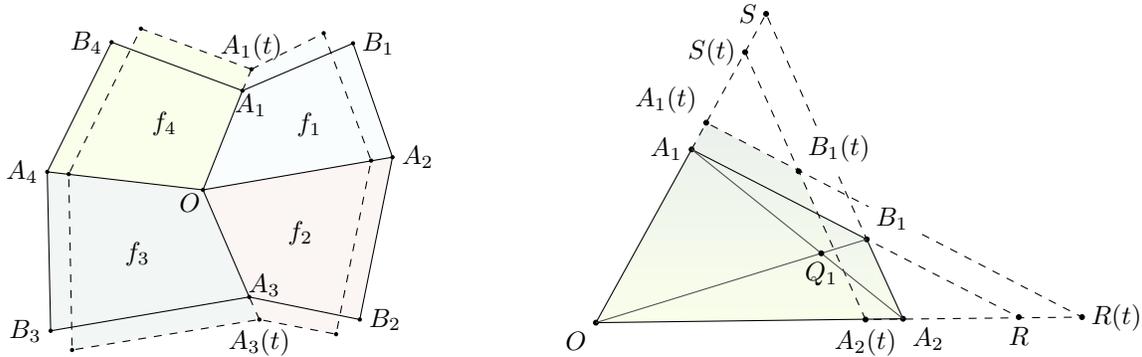
\begin{figure}[htbp]
  \centering
\begin{tikzpicture}[scale=0.15]

\coordinate (A1) at (11.7, -16);
\coordinate (O) at (8.2, -24.8);
\coordinate (A2) at (25, -21.9);
\coordinate (B2) at (22.1, -36.3);
\coordinate (A3) at (12.3, -34.3);
\coordinate (B3) at (-5.3, -37.3);
\coordinate (A4) at (-5.6, -23.2);
\coordinate (B4) at (0.1, -11.7);
\coordinate (B1) at (21.5, -11.8);
\coordinate (A1t) at (12.5, -14.1);
\coordinate (B1t) at (18.9, -10.9);
\coordinate (A2t) at (23.1, -22.2);
\coordinate (B2t) at (20, -37.6);
\coordinate (A3t) at (13.2, -36.3);
\coordinate (B3t) at (-3.4, -39);
\coordinate (A4t) at (-3.7, -23.4);
\coordinate (B4t) at (2.7, -10.5);

\fill[babyblue!5] (O) -- (A1) -- (B1) -- (A2) -- cycle;
\fill[babyblue!5] (O) -- (A1t) -- (B1t) -- (A2t) -- cycle;
\fill[antiquebrass!8] (O) -- (A2) -- (B2) -- (A3) -- cycle;
\fill[antiquebrass!8] (O) -- (A2t) -- (B2t) -- (A3t) -- cycle;
\fill[britishracinggreen!5] (O) -- (A3) -- (B3) -- (A4) -- cycle;
\fill[britishracinggreen!5] (O) -- (A3t) -- (B3t) -- (A4t) -- cycle;
\fill[lime!8] (O) -- (A4) -- (B4) -- (A1) -- cycle;
\fill[lime!8] (O) -- (A4t) -- (B4t) -- (A1t) -- cycle;

\draw[black, thin] (A1) -- (O) -- (A2) -- (B2) -- (A3) -- (B3) -- (A4) -- (B4) -- (A1) -- (B1) -- (A2);
\draw[black, thin] (A4) -- (O) -- (A3);
\draw[black, thin, dashed] (A1t) -- (B1t) -- (A2t) -- (B2t) -- (A3t) -- (B3t) -- (A4t) -- (B4t) -- (A1t);
\draw[black, thin, dashed] (A1t) -- (A1);
\draw[black, thin, dashed] (A3t) -- (A3);

\filldraw[black] (O) circle (4pt);
\filldraw[black] (8.8, -24.3) circle (0pt) node[anchor=north east]{$O$};

\filldraw[black] (25, -21.9) circle (4pt) node[anchor=west]{$A_2$};
\filldraw[black] (12.3, -34.3) circle (4pt);
\filldraw[black] (11.4, -35.1) circle (0pt) node[anchor=south west]{$A_3$};
\filldraw[black] (-5.6, -23.2) circle (4pt) node[anchor=east]{$A_4$};
\filldraw[black] (11.7, -16) circle (4pt); 
\filldraw[black] (10.2, -15.5) circle (0pt) node[anchor=north west]{$A_1$};

\filldraw[black] (21.5, -11.8) circle (4pt) node[anchor=west]{$B_1$};
\filldraw[black] (22.1, -36.3) circle (4pt) node[anchor=west]{$B_2$};
\filldraw[black] (-5.3, -37.3) circle (4pt) node[anchor=east]{$B_3$};
\filldraw[black] (0.1, -11.7) circle (4pt) node[anchor=east]{$B_4$};

\filldraw[black] (23.1, -22.2) circle (4pt); %node[anchor=south east]{$A_2(t)$};
\filldraw[black] (-3.7, -23.4) circle (4pt); %node[anchor=north west]{$A_4(t)$};
\filldraw[black] (12.5, -14.1) circle (4pt) node[anchor=south]{$A_1(t)$};
\filldraw[black] (13.2, -36.3) circle (4pt) node[anchor=north]{$A_3(t)$};

\filldraw[black] (18.9, -10.9) circle (4pt); %node[anchor=south]{$B_1(t)$};
\filldraw[black] (20, -37.6) circle (4pt); %node[anchor=north]{$B_2(t)$};
\filldraw[black] (-3.4, -39) circle (4pt); %node[anchor=north]{$B_3(t)$};
\filldraw[black] (2.7, -10.5) circle (4pt); %node[anchor=south]{$B_4(t)$};

%\filldraw[black] (18.63456,-16.97232) node[anchor=center]{$f_1$};
\filldraw[black] (17.63456,-18.97232) node[anchor=center]{$f_1$};
\filldraw[black] (18.74035,-28.64956) node[anchor=east]{$f_2$};
\filldraw[black] (2.40359,-30.5) node[anchor=center]{$f_3$};
\filldraw[black] (2.84118,-18.94961) node[anchor=west]{$f_4$};

\end{tikzpicture}
%%%%%%%%%%%%%%%%%%%%%%%%%%%%%%%%
\qquad\qquad
\begin{tikzpicture}[scale=1.3]

\coordinate (Q) at (-3.82, -0.74);
\coordinate (A1) at (-2.8416, 1.03277);
\coordinate (A1t) at (-2.69238, 1.30634);
\coordinate (S) at (-2.08, 2.42);
\coordinate (St) at (-2.29447, 2.02757);
\coordinate (A2t) at (-1.06341, -0.70812);
\coordinate (A2) at (-0.68, -0.7);
\coordinate (R) at (0.50339, -0.68325);
\coordinate (Rt) at (1.15, -0.68325);
\coordinate (B1t) at (-1.74733, 0.80894);
\coordinate (B1) at (-1.05098, 0.11259);

\fill[bluedefrance!8] (Q) -- (A1) -- (B1) -- (A2) -- cycle;
\fill[lime!8,path fading=west] (Q) -- (A1) -- (B1) -- (A2) -- cycle;
\fill[britishracinggreen!8,path fading=south] (Q) -- (A1) -- (B1) -- (A2) -- cycle;

\fill[lime!8] (Q) -- (A1t) -- (B1t) -- (A2t) -- cycle;
\fill[lime!8,path fading=west] (Q) -- (A1t) -- (B1t) -- (A2t) -- cycle;
\fill[britishracinggreen!8,path fading=south] (Q) -- (A1t) -- (B1t) -- (A2t) -- cycle;

\draw[black, thin] (Q) -- (A1) -- (B1) -- (A2) -- (Q);
\draw[black, thin, dashed] (A1) -- (S) -- (B1) -- (R);
\draw[black, thin, dashed] (St) -- (A2t) -- (Rt) -- (A1t);

\filldraw[black] (A2) circle (0.7pt) node[anchor=north west]{$A_2$};
\filldraw[black] (S) circle (0.7pt) node[left]{$S$};
\filldraw[black] (Q) circle (0.7pt) node[anchor=north east]{$O$};
\filldraw[black] (Rt) circle (0.7pt) node[anchor=west]{$R(t)$};
\filldraw[black] (A1t) circle (0.7pt) node[anchor=south east]{$A_1(t)$};
\filldraw[black] (B1) circle (0.7pt) node[anchor=south west, fill=white, rounded corners=5pt]{$B_1$};
\filldraw[black] (R) circle (0.7pt) node[anchor=north]{$R$};
\filldraw[black] (A1) circle (0.7pt) node[anchor=east]{$A_1$};
\filldraw[black] (St) circle (0.7pt) node[anchor=east]{$S(t)$};
\filldraw[black] (A2t) circle (0.7pt) node[anchor=north]{$A_2(t)$};
\filldraw[black] (B1t) circle (0.7pt) node[anchor=south west, fill=white, rounded corners=5pt]{$B_1(t)$};

%\node[draw, fill=white, circle, inner sep=2pt] (mynode) at (Q) {};
%\node[above, fill=white] at (Q) {Label};

\draw[black!70, thin] (A1) -- (A2);
\draw[black!70, thin] (Q) -- (B1);

\path[name path=line 1] (A1) -- (A2);
\path[name path=line 2] (Q) -- (B1);
\fill[name intersections={of=line 1 and line 2, by=Q1}];

\filldraw[black] (Q1) circle (0.7pt) node[anchor=north]{$Q_1$};
\end{tikzpicture}
\caption{Notation used throughout Subsection~\ref{ssec-proof-deformable-2x2} (left) and the proofs of Lemmas~\ref{l-system}--\ref{remark-lemma} (right)}
\label{figure:notation}
\end{figure}

%\begin{notation} \label{notation} See Figure~\ref{figure:notation}.
The faces of a $2\times 2$ net are denoted by $f_1,f_2,f_3,f_4$ so that $f_i$ and $f_{i-1}$ have a common edge for $i=1,2,3,4$, and the indices are cyclic modulo $4$. %Hereafter we write $f_0:=f_4$ and $f_5:=f_1$. 
The common vertex of all the faces is denoted by $O$, and the common edge of $f_i$ and $f_{i-1}$ is denoted by $OA_i$. The vertex of $f_i$ other than $O,A_i,A_{i+1}$ is denoted by $B_i$. 

The \emph{simple ratio of a quadrilateral $ABCD$ with respect to the oriented side} $AB$ is (see Figure~\ref{figure:def-oppos-rat})
$$
\begin{cases}
    AB/RA ,&\text{if the rays $BA$ and $CD$ intersect at a point $R$};\\
    0,     &\text{if }AB\parallel CD;\\
    -AB/RA ,&\text{if the rays $AB$ and $DC$ intersect at a point $R$.}
\end{cases}
$$
The simple ratios of the face $f_i$ with respect to $OA_i$ and $OA_{i+1}$ 
are denoted by $l_i$ and $m_i$ respectively.

For two collinear vectors $\overrightarrow{AB}$ and $\overrightarrow{CD}$, denote by $\overrightarrow{AB}/\overrightarrow{CD}$ the number $k$ such that $\overrightarrow{AB}=k\cdot\overrightarrow{CD}$.

\begin{lemma} \label{l-system}
%For a $2\times 2$ net, denote by $l_i$ and $m_i$ the simple ratios of the face $f_i$ with respect to the edges $f_i\cap f_{i-1}$ and $f_i\cap f_{i+1}$ respectively, oriented from $Q$, under Notation~\ref{notation}.
%%Consider a $2\times 2$ net with the faces $QA_iB_iA_{i+1}$, where $i=1,2,3,4$ and  $A_5:=A_1$. For $i=1,2,3,4$ denote by $l_i$ and $m_i$ the simple ratios of the face $QA_iB_iA_{i+1}$ with respect to the sides $QA_i$ and $QA_{i+1}$ respectively.
%%%
%denote by $(\lambda_i,\mu_i)$ the coordinate vector of $\overrightarrow{QB}_i$ relative to $\{\overrightarrow{QA}_i,\overrightarrow{QA}_{i+1}\}$ for $i=1,2,3,4$,  For $i=1,2,3,4$ denote 
%$$P_i(x_i,x_{i+1}):=\mu_i(\mu_i-1)x_i^2+2\lambda_i\mu_ix_ix_{i+1}+\lambda_i(\lambda_i-1)x_{i+1}^2-(\lambda_i+\mu_i)(\lambda_i+\mu_i-1),$$ where $x_{5}:=x_{1}$. %,\mu_5:  =\mu_1,\lambda_5:=\lambda_1$. 
%Under Notation~\ref{notation},
A $2\times 2$ net is deformable if and only if the system of equations
\begin{equation}\label{neweq-main-system}
P_i(x_i,x_{i+1}):=l_ix_i^2+2x_ix_{i+1}+m_ix_{i+1}^2-(l_i+m_i+2)=0 \qquad\text{for }i=1,2,3,4
\end{equation}
%\begin{equation}\label{eq-main-system}
%P_i(x_i,x_{i+1})=0 \qquad\text{for }i=1,2,3,4
%\end{equation}
has a non-constant continuous family of real solutions 
$\big(x_1(t),\dots,x_4(t)\big)$%$\big(x_1(t),x_2(t),x_3(t),x_4(t)\big)$ 
with $x_1(0)=\dots=x_4(0)=1$. 
\end{lemma}

\begin{remark}
    Each such family determines a family of parallel nets with $OA_i(t)/OA_i=x_i(t)$.
\end{remark}

\begin{proof}[Proof] 
 Let us prove the following expression for the ratio of oriented areas:
\begin{equation}\label{eq-area-ratio}
    \frac{\mathrm{Area}(OA_1B_1A_2)}{\mathrm{Area}(OA_1A_2)}=\frac{l_1+m_1+2}{1-l_1m_1}.
\end{equation}
Indeed, first assume that the lines $OA_1,OA_2$ intersect with $A_2B_1,A_1B_1$ at some points $S,R$, respectively. See Figure~\ref{figure:notation}(right). Then ${\overrightarrow{OA_2}}/{\overrightarrow{OR}}=-m_1$ and ${\overrightarrow{OA_1}}/{\overrightarrow{OS}}=-l_1$. By the Menelaus theorem  for points $A_2,B_1,S$ on the extensions of the sides of the triangle $OA_1R$, we get
$$\frac{\overrightarrow{A_1R}}{\overrightarrow{B_1R}}=1+\frac{\overrightarrow{A_1B_1}}{\overrightarrow{B_1R}}=1+\frac{\overrightarrow{SA_1}}{\overrightarrow{OS}}\cdot\frac{\overrightarrow{OA_2}}{\overrightarrow{RA_2}}=1-(1+l_1)\cdot\frac{m_1}{1+m_1}=\frac{1-l_1m_1}{1+m_1}.$$
Then 
\begin{multline*}
\frac{\mathrm{Area}(OA_1B_1A_2)}{\mathrm{Area}(OA_1A_2)}=\frac{\mathrm{Area}(OA_1R)}{\mathrm{Area}(OA_1A_2)}-\frac{\mathrm{Area}(A_2B_1R)}{\mathrm{Area}(OA_1A_2)}
=\frac{\overrightarrow{OR}}{\overrightarrow{OA_2}}-\frac{\overrightarrow{A_2R}}{\overrightarrow{OA_2}}\cdot\frac{\overrightarrow{B_1R}}{\overrightarrow{A_1R}}
=\\=-\frac{1}{m_1}-\frac{1+m_1}{-m_1}\cdot\frac{1+m_1}{1-l_1m_1}
=\frac{l_1+m_1+2}{1-l_1m_1}.
\end{multline*}
This proves~\eqref{eq-area-ratio} unless $OA_1B_1A_2$ has a pair of parallel sides. In the latter case, \eqref{eq-area-ratio} is established, e.g., by a limiting argument.

Now assume that the net is deformable. Without loss of generality, assume that $O$ is fixed during the deformation. In a deformation, take the face  $OA_1(t)B_1(t)A_2(t)$ with the same area and the same
side directions as $OA_1B_1A_2$. Denote $x_1(t)={OA_1(t)}/{OA_1}$ and $x_2(t)={OA_2(t)}/{OA_2}$. Then the simple ratios of the face $OA_1(t)B_1(t)A_2(t)$ with respect to the sides $OA_1(t)$ and $OA_2(t)$ are 
$l_1(t)=l_1x_1(t)/x_2(t)$ and $m_1(t)=m_1x_2(t)/x_1(t)$ respectively. Applying~\eqref{eq-area-ratio} two times, we get 
\begin{multline*}
\frac{l_1+m_1+2}{1-l_1m_1}=\frac{\mathrm{Area}(OA_1B_1A_2)}{\mathrm{Area}(OA_1A_2)}=
\frac{\mathrm{Area}(OA_1(t)B_1(t)A_2(t))}{\mathrm{Area}(OA_1(t)A_2(t))}\cdot\frac{\mathrm{Area}(OA_1(t)A_2(t))}{\mathrm{Area}(OA_1A_2)}=\\=\frac{l_1(t)+m_1(t)+2}{1-l_1(t)m_1(t)}x_1(t)x_2(t)
=\frac{l_1x_1(t)^2+2x_1(t)x_2(t)+m_1x_2(t)^2}{1-l_1m_1}.
\end{multline*}
We have arrived at $P_1(x_1(t),x_2(t))=0$. Analogously we get~\eqref{neweq-main-system} for $i=2,3,4$. 

Conversely, given a family of solutions of~\eqref{neweq-main-system}, we construct a family of parallel nets with the same areas of faces: the point $O$ is fixed; the points $A_i(t)$ are determined by ${OA_i(t)}/{OA_i}=x_i(t)$.
\end{proof}

%To shorten most of the expressions, we substitute $\frac{\mu_i-1}{\lambda_i}=l_i$ and $\frac{\lambda_i-1}{\mu_i}=m_i$ in~\eqref{eq-main-system}, and instead of it, we consider the following system of equations
%\begin{equation}\label{neweq-main-system}
%P_i(x_i,x_{i+1}):=l_ix_i^2+2x_ix_{i+1}+m_ix_{i+1}^2-(l_i+m_i+2)=0 \qquad\text{for }i=1,2,3,4
%\end{equation}

Conditions (i)--(ii) in Theorem~\ref{th-classification} are restated in terms of the simple ratios $l_i$ and $m_i$ as follows.

\begin{lemma}\label{remark-lemma} 
For each $i=1,2,3,4$, we have $l_i+1,m_i+1,1-l_im_i>0$.
%Under Notation~\ref{notation}, we have $l_i+1,m_i+1,1-l_im_i>0$ for each $i=1,2,3,4$. \hiddencomment{We use the convexity!}
%the numbers $l_i+1,m_i+1,1-l_im_i$ do not vanish and have the same sign. In particular, $l_i+m_i+2\ne 0$. 
%
%Moreover, 
Two faces $f_i$ and $f_{i+1}$: %$QA_iB_iA_{i+1}$ and $QA_{i+1}B_{i+1}A_{i+2}$ 
\begin{enumerate}
    \item[\textup{(i)}] are affine symmetric with respect to their common edge if and only if $l_i=m_{i+1}$ and $l_{i+1}=m_i$;
    \item[\textup{(ii)}] have equal opposite-ratios with respect to their common edge if and only if $\frac{1-m_il_i}{(1+l_i)^2}=\frac{1-m_{i+1}l_{i+1}}{(1+m_{i+1})^2}$.
\end{enumerate}
\end{lemma}

The latter two fractions are the expressions for the opposite ratios in terms of simple ratios.

\begin{proof}[Proof]
Assume without loss of generality that $i=1$. 

First, we derive a useful expression for $l_1+1,m_1+1$, and $1-l_1m_1$. Let $Q_1$ be the intersection point of the diagonals of $OA_1B_1A_{2}$ and let $S=OA_1\cap A_2B_1$ (if the latter lines are not parallel). See Figure~\ref{figure:notation}(right). %\mscomm{Can we simplify the proof using the Menelaus theorem instead of the auxiliary similar triangles?} {\color{red}OP:Done! Needed part of the solution lies inside the dashed lines}
Using Menelaus' theorem for points $A_2,B_1,S$ and the triangle $A_1Q_1O$, we get %an equation which is equivalent to
$$l_1+1=\frac{\overrightarrow{A_1O}}{\overrightarrow{OS}}+1=\frac{\overrightarrow{A_1S}}{\overrightarrow{OS}}=\frac{\overrightarrow{A_1A_2}}{\overrightarrow{Q_1A_2}}\cdot\frac{\overrightarrow{B_1Q_1}}{\overrightarrow{B_1O}}.$$
%The left side of this equation is equal to $\overrightarrow{A_1Q}/\overrightarrow{QR}+1=l_1+1$. 
Notice that the resulting expression remains true, even if 
$OA_1\parallel A_2B_1$. Similarly, we have
$$m_1+1=\frac{\overrightarrow{A_2A_1}}{\overrightarrow{Q_1A_1}}\cdot\frac{\overrightarrow{B_1Q_1}}{\overrightarrow{B_1O}}.$$
Substituting $l_1+1,m_1+1$ into $1-l_1m_1=(m_1+1)[1-(l_1+1)+(l_1+1)/(m_1+1)]$, we get
$$1-l_1m_1=\frac{\overrightarrow{A_2A_1}}{\overrightarrow{Q_1A_1}}\cdot\frac{\overrightarrow{B_1Q_1}}{\overrightarrow{B_1O}}\cdot\left(1-\frac{\overrightarrow{A_1A_2}}{\overrightarrow{Q_1A_2}}\cdot\frac{\overrightarrow{B_1Q_1}}{\overrightarrow{B_1O}}-\frac{\overrightarrow{Q_1A_1}}{\overrightarrow{Q_1A_2}}\right)=\frac{\overrightarrow{A_2A_1}}{\overrightarrow{Q_1A_1}}\cdot\frac{\overrightarrow{B_1Q_1}}{\overrightarrow{B_1O}}\cdot\frac{\overrightarrow{Q_1O}}{\overrightarrow{B_1O}}\cdot \frac{\overrightarrow{A_1A_2}}{\overrightarrow{Q_1A_2}}.
$$
In particular, $m_1+1, l_1+1, 1-l_1m_1>0$ by the convexity.
%have the same sign because $\frac{\overrightarrow{A_2A_1}}{\overrightarrow{P_1A_1}}\cdot\frac{\overrightarrow{P_1Q}}{\overrightarrow{B_1Q}}>0$.
\hiddencomment{To do later: This should be updated once we fix the setup (convex vs non-convex faces).  Those values are all positive for a convex quadrilateral, and the signs need not be the same for a quadrilateral with one of the angles $A_1, A_2, Q$ non-convex.}

We have analogous expressions for $m_2+1$, $l_2+1$, $1-l_2m_2$ in terms of 
the point $Q_2=OB_2\cap A_2A_3$. 

\hiddencomment{To do later: This should be updated once we fix the setup (convex vs non-convex faces).  Those values are all positive for a convex quadrilateral, and the signs need not be the same for a quadrilateral with one of the angles $A_1, A_2, Q$ non-convex.}

(i) %\mscomm{Rewrite in form ``The conditions $m_1=l_2$ and $l_1=m_2$ are equivalent to \dots, which is equivalent to \dots, which is equivalent to $QA_1B_1A_2$ and $QA_2B_2A_3$ being affine symmetric'' to make the proof much shorter and clearer. }
%
%Denote the intersection of diagonals of $QA_1B_1A_2$ and $QA_2B_2A_3$ by $P_1$ and $P_2$, respectively and let $T$ be an affine map that maps $\triangle QB_1A_2$ to $\triangle QB_2A_2$. 
From those expressions, we get
$$\frac{m_1+1}{l_1+1}=\frac{\overrightarrow{A_2Q_1}}{\overrightarrow{Q_1A_1}},\quad
\frac{l_2+1}{m_2+1}=\frac{\overrightarrow{A_2Q_2}}{\overrightarrow{Q_2A_3}},\quad
\frac{1}{m_1+1}+\frac{1}{l_1+1}
=\frac{\overrightarrow{B_1O}}{\overrightarrow{B_1Q_1}}, \quad
\frac{1}{m_2+1}+\frac{1}{l_2+1}
=\frac{\overrightarrow{B_2O}}{\overrightarrow{B_2Q_2}}.
$$
Then $(l_1,m_1)=(m_2,l_2)$ is equivalent to 
\begin{equation}\label{eq-ratios}
\frac{\overrightarrow{A_2Q_1}}{\overrightarrow{Q_1A_1}} =\frac{\overrightarrow{A_2Q_2}}{\overrightarrow{Q_2A_3}}
\quad\text{and}\quad 
\frac{\overrightarrow{B_1O}}{\overrightarrow{B_1Q_1}} 
=\frac{\overrightarrow{B_2O}}{\overrightarrow{B_2Q_2}}.
\end{equation}
Consider the affine map taking the triangle $OB_1A_2$ to $OB_2A_2$. 
Condition~\eqref{eq-ratios} means that the affine map takes $A_1$ to $A_3$. The latter is equivalent to $OA_1B_1A_2$ and $OA_2B_2A_3$ being affine symmetric.

(ii) 
Substituting expressions for $l_1+1,m_2+1,1-l_1m_1,1-l_2m_2$, we get 
$$
\frac{1-m_1l_1}{(1+l_1)^2}
=\frac{\overrightarrow{Q_1A_2}}{\overrightarrow{Q_1A_1}}\cdot\frac{\overrightarrow{Q_1O}}{\overrightarrow{Q_1B_1}}
\qquad\text{and}\qquad
\frac{1-m_2l_2}{(1+m_2)^2}
=\frac{\overrightarrow{Q_2A_2}}{\overrightarrow{Q_2A_3}}\cdot\frac{\overrightarrow{Q_2O}}{\overrightarrow{Q_2B_2}}.$$
The right sides are equal to the opposite ratios of the quadrilaterals $OA_1B_1A_2$ and $OA_2B_2A_3$.
%by Proposition~\ref{p-opposite-ratio-equivalent-def}.
\end{proof}
%Before proving the Theorem~\ref{th-classification}, we need the following Lemma.

%Next we show that condition~(ii) in Theorem~\ref{th-classification} is sufficient for the net to be deformable.

This restatement of conditions~(i)--(ii) of Theorem~\ref{th-classification} gives another proof that they are sufficient.

\begin{example}\label{ex-ii}
A $2\times 2$ net satisfying one of conditions~(i)--(ii) in Theorem~\ref{th-classification} is deformable. 
Moreover, for any edge, there is a deformation increasing its length and a deformation decreasing it.
%such that the corresponding edge is longer and a deformation such that the corresponding edge is shorter.
%Moreover, for any edge, there is a deformation such that the corresponding edge is longer and a deformation such that the corresponding edge is shorter.
%%Consider a 
%A $2\times 2$ net %with the faces $f_1,f_2,f_3,f_4$ enumerated counterclockwise. Assume 
%such that for $i=1,\dots,4$ the quadrilaterals $f_i$ and $f_{i+1}$ have equal opposite ratios with respect to their common side %. Let us show that the net 
%is deformable. 
\end{example}

\begin{proof} %Use Notation~\ref{notation} and 
By Lemma~\ref{l-system}, it suffices to construct a family of solutions of system~\eqref{neweq-main-system}.
%Consider classes~(i) and~(ii) separately.

(i) Let both pairs $f_1,f_2$ and $f_3,f_4$ be affine symmetric with respect to the common sides. Then by Lemma~\ref{remark-lemma} we get $l_1=m_2$, $l_2=m_1$, $l_3=m_4$, $l_4=m_3$. For $m_1,m_3\ne0$, we get the following family: % of solutions of system~\eqref{neweq-main-system}:
\begin{align*}
x_1(t) &= 1+t,& 
x_2(t) &= \frac{-(1+t)+\sqrt{(1+t)^2(1-l_1m_1) + m_1(l_1+m_1+2)}}{m_1},\\ 
x_3(t) &= 1+t,& 
x_4(t) &= \frac{-(1+t)+\sqrt{(1+t)^2(1-l_3m_3) + m_3(l_3+m_3+2)}}{m_3}.
\end{align*}
For $m_1=0$ or $m_3=0$ we set $x_2(t)=-l_1(1+t)/2$ or $x_4(t)=-l_3(1+t)/2$ respectively.

(ii) Let the equal opposite ratios of $f_i$ and $f_{i+1}$ be $1/k_{i+1}^2$, where $k_{i+1}>0$, for $i=1,2,3,4$ with $k_5=k_1$.
Then %using notation of Lemma~\ref{l-system}, 
by Lemma~\ref{remark-lemma} we get (see an automated checking in~\cite[Section 1]{pirahmad-maple}) 
%\mscomm{--- Add references to automated checking everywhere analogously to Ref. [40] in Khusrav's paper \url{https://arxiv.org/pdf/2307.05968.pdf} --- see the bottom of page 8 for an example how to refer to a computation in a external file.}
$$
l_i=\frac{k_{i+1}^2-1}{1+k_ik_{i+1}} \quad\text{and}\quad 
m_i=\frac{k_{i}^2-1}{1+k_ik_{i+1}}.
$$
We get the following family of solutions of system~\eqref{neweq-main-system} (see~\cite[Section 2]{pirahmad-maple}):
\begin{align*}
x_1(t) &= \frac{(1+t)^2(1-k_1) + 1+ k_1}{2(1+t)},& x_2(t) &= \frac{(1+t)^2(1+k_2) + 1- k_2}{2(1+t)},\\ x_3(t) &= \frac{(1+t)^2(1-k_3) + 1+ k_3}{2(1+t)},& x_4(t) &= \frac{(1+t)^2(1+k_4) + 1- k_4}{2(1+t)}.
\end{align*}
%By Lemma~\ref{l-system}, the net is deformable.

The `moreover' part holds for the edge $OA_1$ because $x_1'(0)\ne 0$ in both cases~(i) and~(ii). It remains to notice that if %a deformation 
an area-preserving Combescure transformation increases the length of an edge, then it decreases the length of any adjacent edge in the same face.
\end{proof}

To show that conditions~(i)--(ii) in Theorem~\ref{th-classification} are necessary, %we need the following construction. 
consider space $\mathbb{R}^3$ with the coordinates $x_1,x_2,x_3$. 
%The equations $P_1(x_1,x_{2})=0$ and $P_2(x_2,x_{3})=0$ 
Equations~\eqref{neweq-main-system} for $i=1$ and $2$
define two cylinders,
each being 
%either hyperbolic or elliptic because their discriminants $1-m_1l_1$ and $1-m_2l_2$ are non-zero by Lemma~\ref{remark-lemma}. The cylinders are 
centrally symmetric with respect to the origin $O$.
%$(0,0,0)$ and pass through the points $(1,1,1)$ and $(-1,-1,-1)$ (why we need this). 
Denote by $C_1$ their intersection. 
%It is an (affine) algebraic curve of degree at most $4$ which can split into several components.
%of degree %at most 4 (or exactly 
%$4$, if all the components are counted with multiplicity.
Analogously define the sets $C_i$ in space with the coordinates $x_i,x_{i+1},x_{i+2}$ for $i=2,3,4$, where the indices are cyclic modulo $4$.
Those sets are (affine) algebraic curves of degree at most $4$ which can split into several components. We first consider the case when each curve $C_1,\dots,C_4$ contains a \emph{conic}, i.e., an irreducible curve of degree $2$.

\begin{lemma} \label{l-conic} 
The curve $C_i$ contains a conic if and only if $f_i$ and $f_{i+1}$
%the two quadrilaterals with the common side $QA_{i+1}$ 
have equal opposite ratios with respect to their common edge.
\end{lemma}

\begin{remark}\label{l-m2=l1=0} 
In particular, by Lemma~\ref{remark-lemma}(ii), if $m_2=l_1=0$ then $C_1$ contains a conic. 
\end{remark}

\begin{proof}[Proof]
Assume without loss of generality that $i=1$. 

Let us prove the `only if' part. Consider a conic contained in $C_1$. It is a plane section of the cylinder $P_1(x_1,x_2)=0$, which is hyperbolic because the discriminant $1-m_1l_1>0$ by Lemma~\ref{remark-lemma}.  Hence the conic is a hyperbola.
\hiddencomment{In the nonconvex setup:
It is a plane section of the cylinder $P_1(x_1,x_2)=0$, which is either hyperbolic or elliptic because the discriminant $1-m_1l_1$ does not vanish by Lemma~\ref{remark-lemma}.  Hence the conic is either an ellipse or a hyperbola.} Its center is the intersection of the axes $x_1=x_2=0$ and $x_2=x_3=0$ of the two cylinders $P_1(x_1,x_2)=0$ and $P_2(x_2,x_3)=0$ with the plane of the conic. Hence the center is the origin $O$.
%%%
%Since $C$ is centrally symmetric with respect to the origin, it follows that  it contains either a centrally symmetric conic, or two conics centrally symmetric to each other. Let us prove that the latter case is actually impossible. Indeed, otherwise let $(c_1,c_2,c_3)\ne O$ be the center of one of the conics. 
%1) Let for some real $\alpha$ the left side of equation
%$P_1(x_1,x_2)-\alpha P_2(x_2,x_3)=0$ is expressible as 
%$$(ax_1+bx_2+cx_3+d)(ax_1+bx_2+cx_3-d).$$ Then, comparing their coefficients, we obtain
%$$\begin{cases}
%    ac=0,\\
%    ab=1,\\
%    bc=-\alpha,\\
%    a^2=l_1,\\ 
%    c^2=-\alpha m_2,\\
%    b^2=m_1-\alpha l_2,\\ 
%    d^2=l_1 + m_1 + 2 - \alpha(l_2 + m_2 + 2)
%\end{cases}$$
%
%From the first three equations, we get $c=0$ and $\alpha=0$. Then,
%$1=(ab)^2=l_1m_1,$
%implies $\lambda_1+\mu_1=1$ which is impossible.
%%%
%Subcase (ia): $C$ contains a conic centrally symmetric with respect to $O$. Then 
Consider the quadric
\begin{equation}\label{eq-case2}
(l_2+m_2+2)P_1(x_1,x_2)-(l_1+m_1+2)P_2(x_2,x_3)=0.
\end{equation}
It contains the origin $O$ and our conic. Hence it splits into one or two planes,
both passing through the origin by central symmetry. 
Thus the determinant of the quadratic form in the left side of~\eqref{eq-case2} is (see~\cite[Section 3]{pirahmad-maple}) 
\begin{equation}\label{eq-equal-opposite}
(l_1+m_1+2)(l_2+m_2+2)\left((m_2+1)^2(1-m_1l_1)-(l_1+1)^2(1-l_2m_2)\right)=0.
\end{equation}
The first two factors are nonzero by Lemma~\ref{remark-lemma}.
By Lemma~\ref{remark-lemma}(ii) the desired opposite ratios are equal. 
%%%% OLD %%%%
%Thus the left side of~\eqref{eq-case2} splits as $(a_1x_1+a_2x_2+a_3x_3)(b_1x_1+b_2x_2+b_3x_3)$ for some $a_i$ and $b_i$. Comparing the coefficients, we obtain
%\begin{align*}
%a_1b_1&=l_1(l_2 + m_2+2),       & a_2b_3 + a_3b_2&=-2(l_1 + m_1+2),\\
%a_2b_2&= m_1(m_2+2)-l_2(l_1+2), & a_1b_3 + a_3b_1&=0,\\
%a_3b_3&=-m_2(l_1+m_1+2),        & a_1b_2 + a_2b_1&= 2(l_2 + m_2+2).
%\end{align*}
%%\begin{equation}
%%    \begin{cases}\label{syst-eq}
%%    a_1b_3 + a_3b_1=0,\\
%%    a_3b_3=-m_2(l_1+m_1+2),\\ a_1b_1=l_1(l_2 + m_2+2),\\
%%    a_2b_2= m_1(m_2+2)-l_2(l_1+2),\\ 
%%    a_2b_3+ a_3b_2=-2(l_1 + m_1+2),\\
%%    a_1b_2 + a_2b_1= 2(l_2 + m_2+2).
%%\end{cases}
%%\end{equation}
%By the identity 
%$$a_1b_1(a_2b_3+a_3b_2)^2+a_2b_2(a_1b_3+a_3b_1)^2+a_3b_3(a_1b_2+a_2b_1)^2=(a_1b_3+a_3b_1)(a_2b_3+a_3b_2)(a_1b_2+a_2b_1)+4a_1b_1a_2b_2a_3b_3$$
%(which holds for arbitrary real numbers $a_1,a_2,a_3,b_1,b_2,b_3$) \mscomm{It is much more convenient to have a link to a separate Maple file where all long computations are performed; in particular, where \eqref{eq-equal-opposite} is checked as well}, we obtain
%\begin{equation}\label{eq-equal-opposite}
%(l_1+1)^2(1-l_2m_2)=(m_2+1)^2(1-m_1l_1).
%\end{equation}
%Then by Lemma~\ref{remark-lemma}(ii) the required opposite-ratios are equal. 
%%the quadrilaterals $QA_1B_1A_2$ and $QA_2B_2A_3$ have equal opposite-ratios with respect to their common side. %Analogously, the other pairs of quadrilaterals with a common side have equal opposite-ratios with respect to the side. We arrive at condition~(ii).

Let us prove the `if' part. Assume that the opposite ratios are equal. Then by Lemma~\ref{remark-lemma}(ii) we get~\eqref{eq-equal-opposite}.
If $l_1=m_2=0$, then $C_1$ contains the conic $(x_1(t),x_2(t),x_3(t))=\left(\frac{m_1+2-m_1t^2}{2t},t,\frac{l_2+2-l_2t^2}{2t}\right)$. Otherwise assume that $l_1\ne0$ (the case $l_1=0$, $m_2\ne 0$ is similar). Then $C_1$ is contained in the quadric
\begin{multline*}
(1-l_2m_2)l_1P_1(x_1,x_2)-(1-l_1m_1)m_2P_2(x_2,x_3)
=\\=(1-l_2m_2)(l_1x_1+x_2)^2-(1-l_1m_1)(x_2+m_2x_3)^2
-(l_1+1)^2(1-l_2m_2)+(m_2+1)^2(1-l_1m_1)=0.
\end{multline*}
Here the free term vanishes by~\eqref{eq-equal-opposite}, and $1-l_1m_1,1-l_2m_2>0$ %have the same sign 
by Lemma~\ref{remark-lemma}. Thus we get a product of two linear polynomials in $x_1,x_2,x_3$, and the quadric is the union of two planes. Then $C_1$ is the intersection of the planes with the cylinder $P_2(x_2,x_3)=0$, i.e., contains a conic.
\end{proof}

We have the following direct corollary.

\begin{corollary} \label{l-all-conics} 
If $C_1,\dots,C_4$ all contain conics, then %the net satisfies 
condition~(ii) of Theorem~\ref{th-classification} holds.
\end{corollary}

Now we turn to the case when one of the curves $C_i$, say, $C_1$, does not contain a conic. In this case we find the orthogonal projection of $C_1$ to the $x_1x_3$-plane, or, thinking algebraically, we eliminate 
$x_2$ from the system $P_1(x_1,x_2)=P_2(x_2,x_3)=0$. 
We need the following standard observation.

\begin{lemma} \label{ll-projection}
    The orthogonal projection of an algebraic curve $C\subset\mathbb{R}^3$ to the $x_1x_3$-plane preserves the degree of $C$, if the projection is an injective map outside a finite subset, the inverse map is given by rational functions, and the projectivization of $C$ does not contain the improper point of the $x_2$-axis.
\end{lemma}

\begin{proof}
    This is a step where we need slightly more advanced tools. We extend the projection to the complex projective space. Since the inverse map is given by rational functions, it also extends, and thus the projection remains injective. Take a %projective 
    plane passing through the improper point of the $x_2$-axis but not tangent to $C$ and not containing the excluded and singular points of $C$. The plane has the same number of intersection points with $C$ and its projection. By the Bezout theorem, we are done.
\end{proof}

We say that an algebraic equation is \emph{reduced} if it has a minimal degree among all the equations with the same solution set in $\mathbb{R}^2$.

\begin{lemma} \label{l-projection} 
If $C_1$ does not contain a conic, then its orthogonal projection to the $x_1x_3$-plane is an irreducible curve %of degree $3$ or $4$ 
(possibly with finitely many points excluded), given by one of the following %reduced 
equations:
\begin{enumerate}
    \item[\textup{(a)}] 
for $m_1\ne 0$ or $l_2\ne 0$, it is given by the reduced degree $4$ equation
\end{enumerate}
\begin{equation}\label{eq-deg4-3}
    \small{\left|
\begin{array}{cccc}
m_1                 & 0                   & l_2                 & 0 \\
2x_1                & m_1                 & 2x_3                & l_2\\
l_1x_1^2-(l_1+m_1+2)& 2x_1                & m_2x_3^2-(l_2+m_2+2)& 2x_3      \\
0                   & l_1x_1^2-(l_1+m_1+2)& 0                   & m_2x_3^2-(l_2+m_2+2) 
\end{array}
    \right|=0}
\end{equation}
%%%
%(a) 
%for $m_1\ne 0$ it is given by the degree $4$ equation
%\begin{multline}
%    \label{eq-deg4-3}
%4(l_1x_1^2-l_1-m_1-2)(l_2x_1-m_1x_3)^2
%-\\-4x_1(l_2x_1-m_1x_3)(l_1l_2x_1^2-m_1m_2x_3^2+m_1m_2-l_1l_2+2m_1-2l_2)
%+\\+
%m_1(l_1l_2x_1^2-m_1m_2x_3^2+m_1m_2-l_1l_2+2m_1-2l_2)^2=0;
%\end{multline}
%
%(b) for $m_1= 0$, $l_2\ne 0$ it is given by the degree $4$ equation
%\begin{equation}
%    \label{eq-m_1=0<>l_2}
%4(m_2x_3^2-l_2-m_2-2)x_1^2
%-4x_3x_1(l_1x_1^2-l_1-2)
%+l_2(l_1x_1^2-l_1-2)^2=0;
%\end{equation}
%%%
\begin{enumerate}
\item[\textup{(b)}] 
for $m_1=l_2= 0$, it is given by the reduced degree $3$ equation 
\end{enumerate}
\begin{equation}
    \label{express-x2-m_1=l_2=0}
    l_1x_1^2x_3-m_2x_3^2x_1-(l_1+2)x_3+(m_2+2)x_1=0.
\end{equation} 
\end{lemma}

\begin{remark}
    In either case, the left side of~\eqref{eq-deg4-3}--\eqref{express-x2-m_1=l_2=0} is the resultant of $P_1(x_1,x_2)$ and~$P_2(x_2,x_3)$.
    %, up to a factor of $m_1$ or $m_1^2$ in case of equation~\eqref{eq-deg4-3}.
    % $m_1^2$ appears when $l_2=0$
\end{remark}

\begin{proof}[Proof]
Notice that $C_1$ contains no straight lines; otherwise, this line lies in both cylinders and hence is parallel to both axes $Ox_1$ and $Ox_3$, which is impossible. 

    (a): Assume that $m_1\ne 0$ without loss of generality. %In this case, we 
    Eliminate the terms quadratic in $x_2$ from the system $P_1(x_1,x_2)=P_2(x_2,x_3)=0$ by taking 
\begin{equation}
    \label{express-x2}
    \small{l_2P_1(x_1,x_2)-m_1P_2(x_2,x_3)=2x_2(l_2x_1-m_1x_3)+
    \underbrace{l_1l_2x_1^2-m_1m_2x_3^2+m_1m_2-l_1l_2+2(m_1-l_2)}_{Q(x_1,x_3)}=0.}
\end{equation}
Expressing $x_2$ (or better $2x_2(l_2x_1-m_1x_3)$) through the resulting expression $Q(x_1,x_3)$ and substituting into the equation $4(l_2x_1-m_1x_3)^2\cdot P_1(x_1,x_{2})=0$, we get
\begin{equation}
    \label{eq-deg4}
    4(l_1x_1^2-l_1-m_1-2)(l_2x_1-m_1x_3)^2-4x_1(l_2x_1-m_1x_3)Q(x_1,x_3)+
m_1Q(x_1,x_3)^2=0.
%\\=
%    4l_1x_1^2(l_2x_1-m_1x_3)^2-4x_1(l_2x_1-m_1x_3)(l_1l_2x_1^2-m_1m_2x_3^2)+
%m_1(l_1l_2x_1^2-m_1m_2x_3^2)^2+\dots=0,
\end{equation}
%where the dots stay for terms of degree less than $4$. 
The resulting equation is equivalent to~\eqref{eq-deg4-3} (see~\cite[Section 4]{pirahmad-maple}) %\mscomm{MS: Olimjon, will you add an automated checking to the github file?} 
and describes a set \emph{containing} the projection of $C_1$. %to the plane $x_1Ox_3$.

Let us show that the set actually \emph{coincides} with the projection after removing finitely many points. Indeed, take an arbitrary $(x_1,x_3)$ satisfying~\eqref{eq-deg4}. Consider the following two possibilities.

If $l_2x_1-m_1x_3\ne 0$, then set $x_2=Q(x_1,x_3)/2(l_2x_1-m_1x_3)$ so that~\eqref{express-x2} is satisfied. Then \eqref{eq-deg4} implies $P_1(x_1,x_2)=0$. Then by~\eqref{express-x2} we get $P_2(x_2,x_3)=0$ because $m_1\ne 0$. Thus $(x_1,x_2,x_3)\in C_1$ and $(x_1,x_3)$ belongs to the projection of $C_1$.

If $l_2x_1-m_1x_3= 0$, then \eqref{eq-deg4} implies $Q(x_1,x_3)=0$. 
There are only finitely many such points $(x_1,x_3)$ unless the linear polynomial $l_2x_1-m_1x_3$ divides $Q(x_1,x_3)$. Let us show that the latter is impossible.  
Indeed, otherwise $l_2x_1-m_1x_3$ divides the right side of~\eqref{express-x2}
as well. Hence the latter splits into two linear factors, and~\eqref{express-x2} defines the union of two planes (possibly coincident). %Since $m_1\ne 0$, it follows that 
Then $C_1$ is the intersection of those two planes with the cylinder $P_1(x_1,x_2)=0$. Hence $C_1$ contains a conic or a line. This contradiction shows
that curve~\eqref{eq-deg4} and the line $l_2x_1-m_1x_3= 0$ have finitely many common points, which we just drop.
% hence~\eqref{express-x2} is satisfied. Let $x_2$ be a (possibly complex) solution of the quadratic equation $P_1(x_1,x_2)=0$. 
%Then again by~\eqref{express-x2} we get $P_2(x_2,x_3)=0$ and $(x_1,x_3)$ belongs to the projection of $C_1$. \mscomm{Modify the argument to remove complex numbers!}

%, because otherwise there is a curve $C_0$ which is not a projection of $C$ and satisfies the equation~(\ref{eq-deg4}). Then, $l_2x_1-m_1x_3\equiv0$ on the curve $C_0$ and the left side of equation~(\ref{express-x2}) is identical 0. Comparing the coefficients, we get 
%$$\begin{cases}
%    m_1m_2-l_1l_2+2(m_1-l_2)=0\\
%    l_1l_2-\frac{m_2l_2^2}{m_1}=0
%\end{cases}$$
%
%If $l_2=0$, then $m_2=-2$ which is impossible because of $l_2+m_2\not=-2$.
%
%Then, $l_1=\frac{l_2m_2}{m_1}$ and from the first equation, we obtain $(m_1-l_2)(m_2+l_1+2)=0$.
%
%If $m_2+l_1+2=0$, then $\frac{\mu_2+\lambda_2-1}{\mu_2}=-\frac{\mu_1+\lambda_1-1}{\lambda_1}$, and $l_1m_1=l_2m_2$ is equivalent to $\frac{\mu_1+\lambda_1-1}{\lambda_1\mu_1}=\frac{\mu_2+\lambda_2-1}{\lambda_2\mu_2}$. Combining the last two equations, we get $\mu_1+\lambda_2=0$ which is impossible.
%
%Therefore, $m_1=l_2$ and $l_1=m_2$, and the curve $x_1=x_3$ is outside of the projection of $C$. But, if $x_1,x_2$ satisfies $P_1(x_1,x_2)=0$, then $x_2,x_1$ satisfies $P_2(x_2,x_3)=0$. This means that the point $(x_1,x_2,x_1)$ lies in $C$, and its projection is on the line $x_1=x_3$, contradiction.

Now we show that~\eqref{eq-deg4} is an irreducible curve of degree $4$. 
%%% OLD ARGUMENT 
%Computing the coefficients at $x_3^4$ and $x_1^2x_3^2$ and using $m_1\ne 0$, we see that the left side of~\eqref{eq-deg4} has degree $4$ unless $m_2=l_1=0$. The latter possibility is ruled out because otherwise $(x_1,x_2,x_3)=\left(\frac{m_1+2-m_1t^2}{2t},t,\frac{l_2+2-l_2t^2}{2t}\right)$ is a conic contained in $C_1$. 
Since $C_1$ does not contain conics or lines and has degree at most $4$, it follows that $C_1$ is irreducible of degree $4$ or $3$. Then curve~\eqref{eq-deg4} is also irreducible as a projection of an irreducible curve.  
%If equation~\eqref{eq-deg4} is non-reduced, then its left side is a complete square or cube or has a non-constant factor without real roots. In either case,  
However, %this does not yet imply that the curve has degree $4$ because a priori 
equation~\eqref{eq-deg4} can a priori be non-reduced: its left side can be a complete square or cube or have a non-constant factor without real roots. In either case, curve~\eqref{eq-deg4} would have degree at most $2$. This would contradict to Lemma~\ref{ll-projection}.
%WRONG !!! To rule out such a possibility, we prove that the projection preserves the degree of $C_1$. We need to check the following two properties of the projection: it is injective almost everywhere on $C_1$, and $C_1$ has no asymptotes parallel to the projection direction. \mscomm{To do later: maybe insert a reference to the well-known result used here!} 
The assumptions of the lemma hold 
%The projection is injective almost everywhere 
because for $l_2x_1-m_1x_3\ne 0$ the second coordinate of a point $(x_1,x_2,x_3)\in C_1$ is uniquely determined by $x_1$ and $x_3$ via~\eqref{express-x2}, the line $l_2x_1-m_1x_3=0$ intersects curve~\eqref{eq-deg4} only at finitely many points by the above, and 
the projectivization of the cylinder $P_1(x_1,x_2)=0$ does not contain the improper point of the $x_2$-axis by the assumption $m_1\ne 0$. 
%Thus the projection preserves the degree of the curve $C_1$. The latter can only be $4$ or $3$, hence the left side of~\eqref{eq-deg4} is not a complete square or cube. 
Thus~\eqref{eq-deg4} is reduced of degree $4$ or $3$. Computing the coefficients at $x_3^4$ and $x_1^2x_3^2$ and using $m_1\ne 0$, we see that~\eqref{eq-deg4} has degree $4$ unless $m_2=l_1=0$. The latter possibility is ruled out by Remark~\ref{l-m2=l1=0}.
%Lemmas~\ref{l-conic} and~\ref{remark-lemma}(ii).
%%% OLD ARGUMENT
%Finally, the left side of~\eqref{eq-deg4} is not proportional to a complete square (of an irreducible polynomial).
%\mscomm{Compare with your previous argument:}
%\mscomm{Here explain why it is not a square of a quadratic polynomial!} {\color{red}done!}
%If it has the form $(ax_1^2+bx_1x_3+cx_3^2+d)^2$, then
%$$\begin{cases}
%    d^2=m_1(-l_1l_2 + m_1m_2 - 2l_2 + 2m_1)^2\\ 
%    c^2=m_1^3m_2^2\\
%    a^2=l_1^2l_2^2m_1\\
%    ab=-2l_1l_2m_1\\
%    2ac+b^2=-2l_1l_2m_1^2m_2 + 4l_1m_1^2 + 4l_2m_1m_2\\
%    bc=-2m_1^2m_2\\
%    db=2m_1(l_1l_2 + 2l_2m_1 + m_1m_2 + 2l_2 + 2m_1)\\ 
%    dc=m_1^2(l_1l_2m_2 - m_1m_2^2 + 2l_2m_2 - 2m_1m_2 - 2l_1 - 2m_1 - 4)\\
%    da=-l_2m_1(l_1^2l_2 - l_1m_1m_2 + 2l_1l_2 - 2l_1m_1 + 2l_2 + 2m_2 + 4)
%\end{cases}$$
%
%Substituting in $d^2\cdot a^2-(da)^2=0$ and $d^2\cdot ab-da\cdot db=0$ the right-side expressions, we get
%$$4l_2^2m_1^2\big(l_2 + m_2 + 2\big)\big(l_1^2l_2 - l_1m_1m_2 + 2l_1l_2 - 2l_1m_1 + l_2 + m_2 + 2\big)=0,$$
%and 
%$$4l_2m_1^2\big(l_2 + m_2 + 2\big)\big(m_1(l_1^2l_2 - l_1m_1m_2 + 2l_1l_2 - 2l_1m_1+l_2+m_2+2) + l_2(l_1 + m_1 + 2)\big)=0.$$
%Using $m_1(l_1+m_1+2)(l_2+m_2+2)\not=0$ in these two equations we get $l_2=0$. Then $a=0$ and from $b^2\cdot c^2-(bc)^2=0$ we get $m_1m_2(l_1m_1-1)=0$ which implies $m_2=0$, thus $c=0$. Then from $dc=0$, we obtain $l_1+m_1+2=0$ which is impossible. 
%%%
%    (b): $m_1= 0$, $l_2\ne 0$. This case is similar to (a), only we substitute $x_2$ from~\eqref{express-x2} to $4x_1^2P_2(x_2,x_{3})=0$. 

    (b): $m_1=l_2= 0$.  In this case, we eliminate $x_2$ by taking 
\begin{equation}
    \label{eq-express-x2-m_1=l_2=0}
    x_3P_1(x_1,x_2)-x_1P_2(x_2,x_3)
    =l_1x_1^2x_3-m_2x_3^2x_1-(l_1+2)x_3+(m_2+2)x_1=0.
\end{equation}
The resulting equation describes a set \emph{containing} the projection of $C_1$ to the $x_1x_3$-plane.

Set~\eqref{eq-express-x2-m_1=l_2=0} actually \emph{coincides} with the projection after removing the origin. Indeed, take an arbitrary $(x_1,x_3)\ne (0,0)$ satisfying~\eqref{eq-express-x2-m_1=l_2=0}. We have $x_1,x_3\ne 0$, otherwise $l_1=-2$ or $m_2=-2$ contradicting to Lemma~\ref{remark-lemma}. % for $m_1=l_2=0$. 
Set $x_2=(l_1+2-l_1x_1^2)/(2x_1)=(m_2+2-m_2x_3^2)/(2x_3)$. Then $P_1(x_1,x_2)=P_2(x_2,x_3)=0$. Thus $(x_1,x_2,x_3)\in C_1$ and $(x_1,x_3)$ belongs to the projection of $C_1$.

Finally, we show that~\eqref{eq-express-x2-m_1=l_2=0} is an irreducible cubic curve. Eq.~\eqref{eq-express-x2-m_1=l_2=0} is reduced and has degree $3$, otherwise the curve is a line, hence $C_1$ is the intersection of a plane and cylinder, i.e., a conic. Similarly, the curve is irreducible, otherwise it contains a line, and hence $C_1$ contains a conic or a line.
\end{proof}

For the curve $C_3$, once it does not contain a conic, the projection to the $x_1x_3$-plane is given by similar equations~\eqref{eq-deg4-3}--\eqref{express-x2-m_1=l_2=0}, only $l_1,m_1,l_2,m_2$ are replaced by $m_4,l_4,m_3,l_3$ respectively. If the $2\times 2$ net is deformable, then by Lemma~\ref{l-system} %system~\eqref{neweq-main-system} has a continuous family of solutions $(x_1(t),x_2(t),x_3(t),x_4(t))$, hence 
the projections of $C_1$ and $C_3$ have a common curve $(x_1(t),x_3(t))$. But both projections are irreducible, thus %they must coincide and 
their reduced equations are \emph{proportional}, i.e., the left sides become equal polynomials after multiplication by a nonzero constant. Let us study two typical possibilities for that.

\begin{lemma} \label{l-proportional}
    Assume that $C_1$ does not contain a conic. 
    Denote by \textup{(\ref{eq-deg4-3}')--(\ref{express-x2-m_1=l_2=0}')} equations~\eqref{eq-deg4-3}--\eqref{express-x2-m_1=l_2=0} with $l_1,m_1,l_2,m_2$ replaced by $m_4,l_4,m_3,l_3$ respectively. If 
    at least one pair of equations
    \begin{itemize}
    \item[(a)] \textup{(\ref{eq-deg4-3}')} and~\eqref{eq-deg4-3}, where $m_1\ne 0$; %\quad
    %\item[(b)] (\ref{eq-m_1=0<>l_2}') and~\eqref{eq-deg4-3}, where $m_1,m_3\ne 0$, $l_4=0$; \quad
    \item[(b)] \textup{(\ref{express-x2-m_1=l_2=0}')} and \eqref{express-x2-m_1=l_2=0}, where $m_1=l_2=m_3=l_4=0$;
    \end{itemize}
    is proportional, then the $2\times 2$ net satisfies condition~(i) of Theorem~\ref{th-classification}. %\mscomm{Is it indeed the case or we can get condition~(ii) as well???}
\end{lemma}

%Beware that in (b), the two equations have different numbers.

\begin{proof}[Proof]
(a) %Since (\ref{eq-deg4-3}') and~\eqref{eq-deg4-3} are irreducible and equivalent, their coefficients are proportional, i.e., 
%%%
%Assume that (\ref{eq-deg4-3}') and~\eqref{eq-deg4-3} are proportional,
%We shall see that (\ref{eq-deg4-3}') has an overall factor of $l_4$ and~\eqref{eq-deg4-3} has an overall factor of $m_1$. Thus for convenience, let the proportionality coefficient be denoted by $\alpha l_4/m_1$ for some $\alpha\ne 0$, so that (\ref{eq-deg4-3}') equals~\eqref{eq-deg4-3} multiplied by $\alpha l_4/m_1$. 
Assume that (\ref{eq-deg4-3}') equals~\eqref{eq-deg4-3} multiplied by some $\alpha\ne 0$.
Comparing the coefficients at $x_1x_3^3$, $x_3^4$, $x_1^3x_3$, $x_1^4$, $x_1^2x_3^2$, $x_1x_3$, $1$, $x_3^2$, $x_1^2$ respectively, we get the %following
system (see~\cite[Section 5]{pirahmad-maple}): 
%\mscomm{!!!The system is wrong!!! For instance, the coefficient at $x_1^3x_3$ is $-4m_1l_1l_2$ rather than $l_1l_2$. Actually, the right side of each equation should have a factor of $m_1$. The whole proof needs to be changed, and as a result, it probably becomes much shorter.}
%of equations:  %\mscomm{- change the left side of each equation by the right one, with  $l_1,m_1,l_2,m_2$ replaced by $m_4,l_4,m_3,l_3$ respectively! So that the first equation becomes $l_4^2l_3=\alpha m_1^2m_2$ and so on. The proof below alreay refers to the changed system.}
\begin{align}
    \label{eq-system-proportional-1}
    l_4l_3&=\alpha m_1m_2\\
    \label{eq-system-proportional-2}
    l_4^2l_3^2&=\alpha m_1^2m_2^2\\ 
    \label{eq-system-proportional-3}
    m_4m_3&=\alpha l_1l_2\\
    \label{eq-system-proportional-4}
    m_4^2m_3^2&=\alpha l_1^2l_2^2\\
    \label{eq-system-proportional-5}
    l_3m_3l_4m_4  - 2l_4m_4- 2l_3m_3&=\alpha \Big(l_1m_1l_2m_2 - 2l_1m_1 - 2l_2m_2\Big)\\
    \label{eq-system-proportional-6}
    (2m_3 + l_3 + 2)l_4 + (m_4 + 2)m_3&=\alpha \Big((2l_2 + m_2 + 2)m_1 + (l_1 + 2)l_2\Big)\\
    \label{eq-system-proportional-7}
    \Big((m_4 + 2)m_3 - l_4(l_3 + 2)\Big)^2&=\alpha \Big((l_1 + 2)l_2 - m_1(m_2 + 2)\Big)^2\\
    \label{eq-system-proportional-8}
    l_4\Big((l_3^2+2l_3+2)l_4-(l_3m_3-2)(m_4+2)\Big)&=\alpha m_1\Big((m_2^2 + 2m_2 + 2)m_1 - (l_2m_2 - 2)(l_1 + 2)\Big)\\
    \label{eq-system-proportional-9}
    %\!\!\!\!\!\!\!\!
    \Scale[0.95]{m_3\Big((m_4^2 + 2m_4 + 2)m_3 - (l_4m_4 - 2)(l_3 + 2)\Big)}&=\Scale[0.95]{\alpha l_2\Big((l_1^2+2l_1+2)l_2-(l_1m_1-2)(m_2+2)\Big).}
\end{align}
%\begin{align}
    %\label{eq-system-proportional-1p}
    %x_1x_3^3:l_4^2l_3&=\alpha m_1^2m_2\\
    %\label{eq-system-proportional-2}
    %x_3^4:l_4^3l_3^2&=\alpha m_1^3m_2^2\\ 
    %\label{eq-system-proportional-3}
    %x_1^3x_3:l_4m_4m_3&=\alpha m_1 l_1l_2\\
    %\label{eq-system-proportional-4}
    %x_1^4:l_4m_4^2m_3^2&=\alpha m_1l_1^2l_2^2\\
    %\label{eq-system-proportional-5}
    %x_1^2x_3^2:l_4\Big(l_3m_3l_4m_4  - 2l_4m_4- 2l_3m_3\Big)&=\alpha m_1\Big(l_1m_1l_2m_2 - 2l_1m_1 - 2l_2m_2\Big)\\
    %\label{eq-system-proportional-6}
    %x_1x_3:l_4\Big((2m_3 + l_3 + 2)l_4 + (m_4 + 2)m_3\Big)&=\alpha m_1\Big((2l_2 + m_2 + 2)m_1 + (l_1 + 2)l_2\Big)\\
    %\label{eq-system-proportional-7}
    %1:l_4\Big((m_4 + 2)m_3 - l_4(l_3 + 2)\Big)^2&=\alpha m_1\Big((l_1 + 2)l_2 - m_1(m_2 + 2)\Big)^2\\
    %\label{eq-system-proportional-8}
    %x_3^2:l_4^2\Big((l_3^2+2l_3+2)l_4-(l_3m_3-2)(m_4+2)\Big)&=\alpha m_1^2\Big((m_2^2 + 2m_2 + 2)m_1 - (l_2m_2 - 2)(l_1 + 2)\Big)\\
    %\label{eq-system-proportional-9}
    %\!\!\!\!\!\!\!\!x_1^2:l_4m_3\Big((m_4^2 + 2m_4 + 2)m_3 - (l_4m_4 - 2)(l_3 + 2)\Big)&=\alpha m_1l_2\Big((l_1^2+2l_1+2)l_2-(l_1m_1-2)(m_2+2)\Big)
%\end{align}
Let us prove that all the solutions of the system are given by $\alpha=1$ 
and $(l_1,m_1,l_2,m_2)=(m_4,l_4,m_3,l_3)$.

First note that $l_4\ne 0$. Indeed, otherwise \eqref{eq-system-proportional-1} implies $m_2=0$ and \eqref{eq-system-proportional-8} implies $l_1+m_1+2=0$ which contradicts Lemma~\ref{remark-lemma}. Now consider two cases.

Case 1: $m_3=0$. Then we get $l_2=0$ from~(\ref{eq-system-proportional-3}) if $l_1\not=0$, and from~(\ref{eq-system-proportional-9}) if $l_1=0$ because of $l_2+m_2+2\not=0$. %GAP: Then from~(\ref{eq-system-proportional-3}), we get $l_2=0$  because of $l_1\not=0$ by Remark~\ref{l-m2=l1=0}. 
If $\alpha\not=1$, then using $m_1,l_4\ne 0$, from~\eqref{eq-system-proportional-1}--\eqref{eq-system-proportional-2} we get $m_2=l_3=0$. Using  $l_2,m_2,l_3,m_3=0$ in~(\ref{eq-system-proportional-6})--(\ref{eq-system-proportional-7}), we obtain $l_4=\alpha m_1$ and $l_4^2=\alpha m_1^2$ which together give $\alpha=1$, contradiction. Therefore, $\alpha=1$, and using $m_3=l_2=0$ in~(\ref{eq-system-proportional-1}),~(\ref{eq-system-proportional-5}),~(\ref{eq-system-proportional-6}), we get $l_3l_4=m_1m_2$, $l_4m_4=l_1m_1$, $l_4=m_1$. These all together give $(l_1,m_1,l_2,m_2)=(m_4,l_4,m_3,l_3)$ and we are done.

Case 2: $m_3\ne 0$. 
First, let us prove that $\alpha=1$. Assume the converse. Since $m_1,l_4\ne 0$, from~\eqref{eq-system-proportional-1}--\eqref{eq-system-proportional-2} we get $m_2=l_3=0$, and from~\eqref{eq-system-proportional-3}--\eqref{eq-system-proportional-4} we get $l_1l_2=m_4m_3=0$. Since $l_1\ne 0$ by Remark~\ref{l-m2=l1=0} and $m_3\ne 0$, we obtain $l_2=m_4=0$. Then~\eqref{eq-system-proportional-9} becomes $m_3(l_3+m_3+2)=0$
which contradicts to Lemma~\ref{remark-lemma}. Thus $\alpha=1$.

Then ~\eqref{eq-system-proportional-1} and~\eqref{eq-system-proportional-3} become $l_3l_4=m_1m_2$ and $l_1l_2=m_3m_4$, respectively. Using the latter equations in~(\ref{eq-system-proportional-6}), we express 
\begin{equation} \label{eq-expressions}
l_2=\frac{(m_3+1)(l_4+1)}{m_1+1}-1, \quad l_3=\frac{m_1m_2}{l_4}, \quad m_4=\frac{l_1l_2}{m_3}=\frac{l_1(m_3l_4+m_3+l_4-m_1)}{m_3(m_1+1)}.
\end{equation}
Substituting \eqref{eq-expressions} into~(\ref{eq-system-proportional-7}) and~(\ref{eq-system-proportional-8}), we get (see~\cite[Section 6]{pirahmad-maple}) 
%$$(m_1 + m_3 + 2)(l_4-m_1)(\underbrace{l_1l_4m_3 - l_3l_4m_1 + l_1l_4 - l_1m_1 + l_1m_3 - l_3l_4 - l_4m_1 + l_4m_3 - m_1^2 + m_1m_3 - 2m_1 + 2m_3}_A)=0,$$
\begin{align*}
(m_1 + m_3 + 2)(l_4-m_1)\Big(\underbrace{(m_3+1)(l_1+1)(l_4+1)-(m_1+1)(m_1m_2+l_1+l_4+m_1-m_3+1)}_A\Big)&=0,\\
(l_4-m_1)\Big(A+(1 + m_3)m_1(\underbrace{m_1m_2 + m_2m_3 + l_1 + m_2 + l_4 + m_1 + 2}_B)\Big)&=0.
\end{align*}
%$$\text{and}\qquad(l_4-m_1)\Big(A+(1 + m_3)(\underbrace{l_3l_4m_1 + l_3l_4m_3 + l_1m_1 + l_3l_4 + l_4m_1 + m_1^2 + 2m_1}_B)\Big)=0,$$
%respectively. 

Now if $l_4\not=m_1$, then $A=B=0$ because $m_1,m_1+m_3+2,1+m_3\not=0$ by Lemma~\ref{remark-lemma}. Then $A+(m_1-m_3)B=(1+m_3)(l_1l_4-m_2m_3)=0$ (see~\cite[Section 7]{pirahmad-maple}), thus $l_1l_4=m_2m_3$. Hence $B=(m_1+1)(m_2+1)+(l_1+1)(l_4+1)=0$ which contradicts to Lemma~\ref{remark-lemma}. 

Therefore, $l_4=m_1$, and from \eqref{eq-expressions} we get $(l_1,m_1,l_2,m_2)=(m_4,l_4,m_3,l_3)$ and we are done.

(b) If the coefficients of (\ref{express-x2-m_1=l_2=0}') and \eqref{express-x2-m_1=l_2=0} are proportional, then $m_4=l_1$ and $l_3=m_2$. Since $m_1=l_2=m_3=l_4=0$, by Lemma~\ref{remark-lemma}(i) we arrive at condition~(i).
\end{proof}

Finally, we summarize the whole argument.

\begin{proof}[Proof of Theorem~\ref{th-classification}]
%Denote the first, second, third and fourth equation of the following system of equations by $\mathcal{C}_1$, $\mathcal{C}_2$, $\mathcal{C}_3$ and $\mathcal{C}_4$, respectively
%$$\begin{cases}
%    \mu_1(\mu_1-1)x_1^2+2\lambda_1\mu_1x_1x_2+\lambda_1(\lambda_1-1)x_2^2=(\lambda_1+\mu_1)(\lambda_1+\mu_1-1)\\
%    \mu_2(\mu_2-1)x_2^2+2\lambda_2\mu_2x_2x_3+\lambda_2(\lambda_2-1)x_3^2=(\lambda_2+\mu_2)(\lambda_2+\mu_2-1)\\
%    \mu_3(\mu_3-1)x_3^2+2\lambda_3\mu_3x_3x_4+\lambda_3(\lambda_3-1)x_4^2=(\lambda_3+\mu_3)(\lambda_3+\mu_3-1)\\
%    \mu_4(\mu_4-1)x_4^2+2\lambda_4\mu_4x_4x_1+\lambda_4(\lambda_4-1)x_1^2=(\lambda_4+\mu_4)(\lambda_4+\mu_4-1)
%\end{cases}$$
%
Conditions~(i)--(ii) are sufficient by Example~\ref{ex-ii}.
Let us prove that they are necessary. Assume that the $2\times 2$ net is deformable. Then by Lemma~\ref{l-system} system~\eqref{neweq-main-system} has a continuous family of real solutions $\big(x_1(t),x_2(t),x_3(t),x_4(t)\big)$ with all $x_i(0)=1$.
%for $i=1,2,3,4$. 
%Assume further that not all $l_i$ and $m_i$ vanish, otherwise we have case~(i). We may assume that $m_1$ has the maximal absolute value among all $l_i$ and $m_i$ (otherwise perform either a suitable cyclic permutation of variables or reverse their order). Hence $m_1\ne 0$.
Recall that $C_i$ denotes the curve given by $P_i(x_i,x_{i+1})=P_{i+1}(x_{i+1},x_{i+2})=0$, where the indices are cyclic modulo~$4$.

Case (ii): \emph{each curve $C_1,\dots,C_4$ contains a conic}. Then the theorem follows from Corollary~\ref{l-all-conics}.

%\mscomm{And what's next? why the other three pairs of opposite ratios are equal?} 

%\mscomm{What if $C$ contains a conic, the projection of $C'$ contains a conic, but $C'$ does not contain a conic?}

Case (i): \emph{at least one of the curves $C_1,\dots,C_4$ does not contain a conic}. We use a symmetry argument to minimize the number of subcases below: First, assume without loss of generality that $C_1$ does not contain a conic. Second, assume that $m_1$ has the maximal absolute value among all $m_i$ and $l_{i+1}$ such that $C_i$ does not contain a conic (otherwise perform either a suitable cyclic permutation of variables $x_1,\dots,x_4$ or reverse their order).

The projections of $C_1$ and $C_3$ to the $x_1x_3$-plane have a common curve $\big(x_1(t),x_3(t)\big)$. By Lemma~\ref{l-projection}, the projection of $C_1$ is an irreducible curve of degree $3$ or $4$. Thus the curve $C_3$ of degree at most $4$ cannot contain a conic, otherwise both $C_3$ and its projection 
split into a union of conics and/or lines, not curves of higher degree. So, by Lemma~\ref{l-projection} the projections of $C_1$ and $C_3$ are irreducible curves given by%one of equations
~\eqref{eq-deg4-3} or \eqref{express-x2-m_1=l_2=0}, with $l_1,m_1,l_2,m_2$ replaced by $m_4,l_4,m_3,l_3$ respectively in case of $C_3$. Since both projections are irreducible, %and 
their reduced equations %are reduced, it follows that the equations 
must be proportional. Consider %the following 
$2$ subcases.

Subcase (a): $m_1\ne 0$. In this case, the two projections are given by~\eqref{eq-deg4-3} and~(\ref{eq-deg4-3}'): they cannot be given by~\eqref{eq-deg4-3} and~(\ref{express-x2-m_1=l_2=0}') because the degrees are equal. By Lemma~\ref{l-proportional}(a), the theorem follows.

%Subcase (b): $m_1\ne 0$, $l_4=0$. In this case we have $m_3\ne0$, otherwise by 
%Lemma~\ref{l-projection} the projections of $C_1$ and $C_3$ have distinct degrees and cannot be proportional. 
%No: \mscomm{It is better to consider the case $m_3=0$ separately in Lemma~\ref{l-proportional} because Lemma~\ref{l-projection} does not assert anything precise on the degrees!}. 
%Thus the theorem follows from Lemma~\ref{l-proportional}(b).

Subcase (b): $m_1=0$. Recall that $m_1$ was chosen to have the maximal absolute value among all $m_i$ and $l_{i+1}$ such that $C_i$ does not contain a conic. Thus $|l_2|,|m_3|,|l_4|\le |m_1|=0$ and hence $l_2=m_3=l_4=0$.  
By Lemma~\ref{l-proportional}(b), the theorem follows.
\end{proof}

%\mscomm{Add a remark about the Stachel conjecture: one of the resultants is always reducible.}

\begin{remark}  In our setup, %in general position, 
the so-called \emph{Stachel conjecture} holds: the resultant of at least one consecutive pair of homogenized polynomials $P_i(x_i,x_{i+1})$, where $i=1,2,3,4$, is reducible. Indeed, by Theorem~\ref{th-classification}, at least one pair of adjacent faces $f_i$, $f_{i+1}$ has equal opposite ratios with respect to their common edge. By Lemma~\ref{l-conic}, the curve $C_i$, hence its projection to the $x_ix_{i+2}$-plane, contains a conic or a line. Hence the resultant of $P_i$ and $P_{i+1}$ has a factor of degree $2$ or $1$. But the resultant itself, when homogenized, has always degree $3$ or $4$; % unless $l_i=m_{i+1}=0$; 
see~\eqref{eq-deg4-3}--\eqref{express-x2-m_1=l_2=0}. Thus %in general position, 
it is reducible. The example of a $2\times 2$ net with all the faces being parallelograms shows that the homogenization of the polynomials is necessary here.
%for the Stachel conjecture to be true. 
%However, there are exceptions to the Stachel conjecture: for instance, a net such that all the faces are parallelograms. 
%the resultants of all consecutive pairs of polynomials $P_i(x_i,x_{i+1})$ are linear and hence irreducible.   
\end{remark}

\section{Deformable $m\times n$ nets}
\label{sec-deformable-mxn}

In this section, we characterize all deformable nets of arbitrary size. 
We state the classification %and construction results (Theorem~\ref{th-mxn} and Theorem~\ref{th-classification-mxn}) 
in Section~\ref{ssec-statement-deformable-mxn}, discuss the geometry of deformable nets in Section~\ref{ssec-properties-deformable-mxn}, and give the proof in Section~\ref{ssec-proof-deformable-mxn}. 
%See Figures~\ref{fig:type i}--\ref{fig:type ii}. % Under the notations of Section~\ref{ssec-classification-deformable-2x2},

%\mscomm{————– MIKHAIL WILL EDIT THIS SECTION —————–}

%\mscomm{To do}

%Recall that a net is deformable if it belongs to a continuous family of parallel nets with equal areas of corresponding faces. In this section we prove some results related to properties of deformable net and its generalization. 

\subsection{Statement of the classification}
\label{ssec-statement-deformable-mxn}

%Our characterization of deformable $m\times n$ nets is in terms of their %$2\times 2$ sub-nets. 
We need the following notions. By an $a\times b$ \emph{sub-net} of a given $m\times n$ net with the points $P_{ij}$, where $0\leq i\leq m$ and $0\leq j\leq n$, we mean an $a\times b$ net with the points $P_{ij}$, where $p\leq i\leq p+a$ and $q\leq j\leq q+b$, for some integers $0\leq p\leq n-a$ and $0\leq q\leq m-b$. (The collection of $(a+1)(b+1)$ points indexed in this way is still viewed as an $a\times b$ net.) 
%By a \emph{deformation} of a net we mean any parallel net with the same areas of corresponding faces.

%Consider an $m\times n$ net with indexed vertices $P_{ij}$. We say that a pair of neighboring faces is \emph{horizontal} (respectively, \emph{vertical}), if their common edge has form $P_{i,j}P_{i,j+1}$ (respectively, $P_{i,j}P_{i+1,j}$) for some $0\leq i\leq m$ and $0\leq j\leq n$.

\begin{theorem}
\label{th-mxn}
An $m\times n$ net %, where $m,n\ge 2$, 
is deformable if and only if %at least 
one of the following conditions~holds:
\begin{itemize}
\item[\textup{(i)}] in each $1\times 2$ sub-net or in each $2\times 1$ sub-net, the two faces are affine symmetric with respect to their common edge;
%each horizontal or each vertical pair of neighboring faces is affine symmetric with respect to the common edge;
\item[\textup{(ii)}] each pair of faces with a common edge has equal opposite ratios with respect to that edge.
%each two neighboring faces have equal opposite ratios with respect to their common edge.
\end{itemize}
%either each $2\times 2$ sub-net satisfies condition~(i) of Theorem~\ref{th-classification} or each $2\times 2$ sub-net satisfies condition~(ii) of Theorem~\ref{th-classification}. 
%
%In the former case, for $m,n\ge 2$, either in each $1\times n$ sub-net or in each $m\times 1$ sub-net, all pairs of faces with a common edge are affine symmetric with respect to that edge.
\end{theorem}

%An important point here is that 
Thus all the $2\times 2$ sub-nets of a deformable $m\times n$ net belong to the same class, (i) or (ii). %We shall say that such an $m\times n$ net belongs to class (i) or (ii) respectively. 

Theorem~\ref{th-mxn} allows us to check if a given net is deformable. Now we give a parametrization of deformable nets, allowing us to construct all deformable nets close enough to a square net. 

By an \emph{L-shaped net of size $m\times n$} we mean an indexed collection of $2m+2n$ points $P_{ij}$, where the indices satisfy the inequalities $0\le i\le m$, $0\le j\le n$, $\min\{i,j\}\le 1$, such that $P_{ij},P_{i+1,j},P_{i,j+1},P_{i+1,j+1}$ are vertices of a convex quadrilateral whenever $i=0$, $0\le j< n$ or $j=0$, $0\le i< m$. %\mscomm{Add figure!}
See Figure~\ref{L-shape}.
\emph{Edges}, \emph{faces}, and \emph{sub-nets} are defined analogously to the ones of an $m\times n$ net. An \emph{L-shaped square net} is an L-shaped net such that the faces are coplanar non-coincident squares.
\begin{figure}[htbp]
    \centering
    \includegraphics[scale=0.025]%[scale=0.15]    
{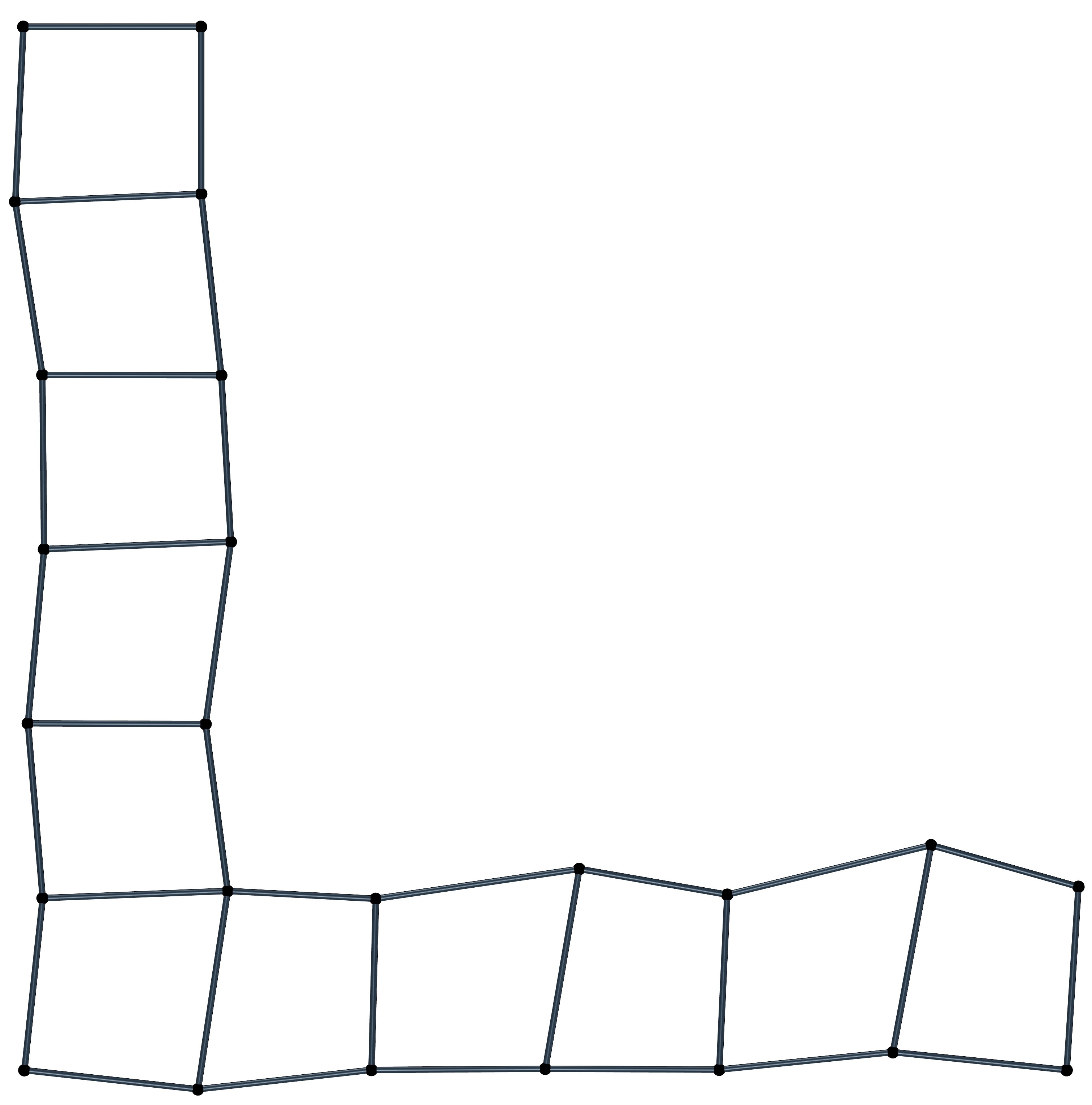}\qquad\qquad\qquad\qquad
\includegraphics[scale=0.025]%[scale=0.15]
{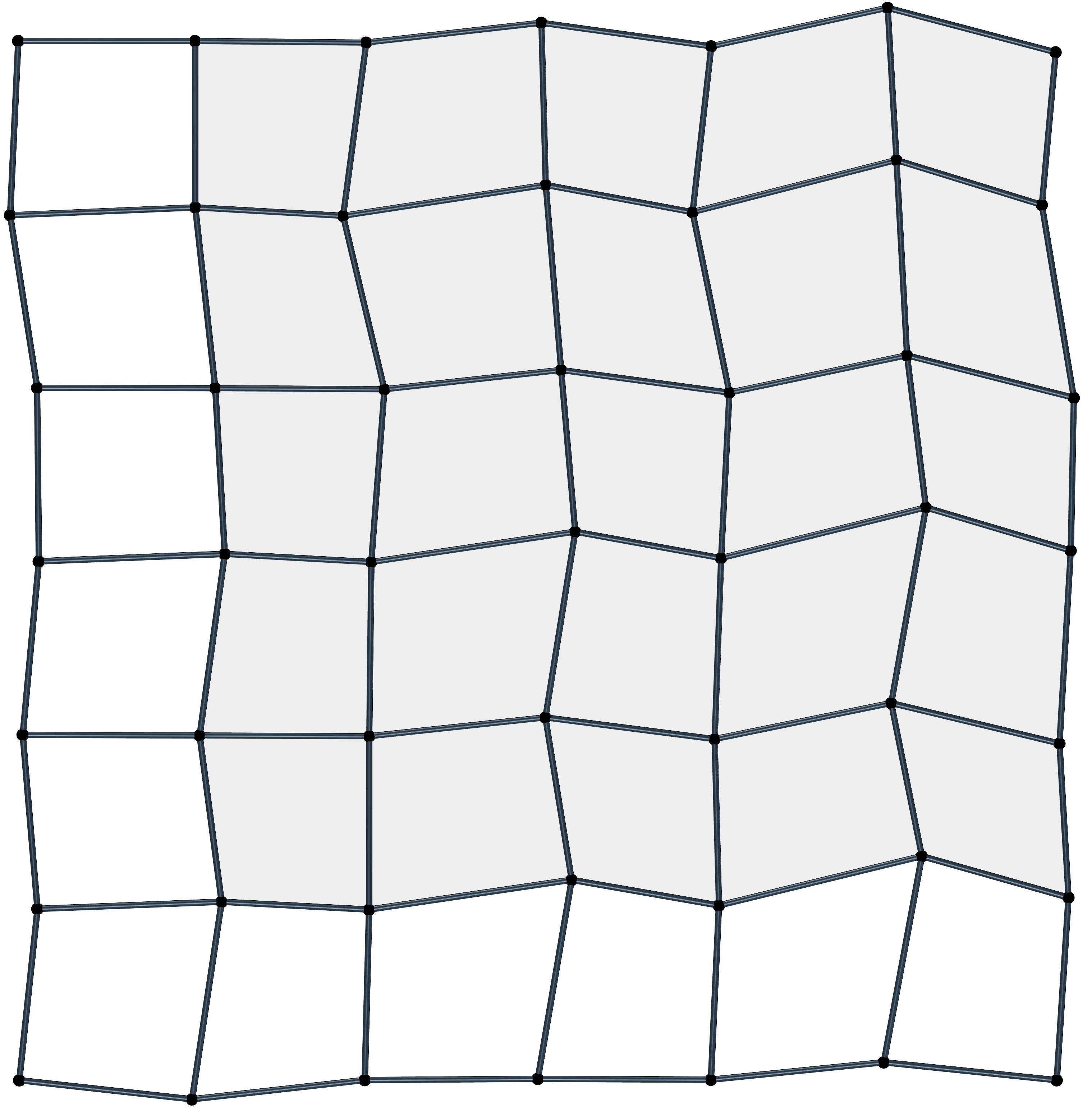}
    \caption{Left: An L-shaped net of size $6\times6$. Right: The unique deformable $6\times6$ net containing the L-shaped net. See Corollary~\ref{cor-L}.}
    \label{L-shape}
\end{figure}

%By an \emph{L-square shaped net of size $m\times n$} we mean an L-shaped net of size $m\times n$ which is sufficiently close to L-shaped net of size $m\times n$ with square faces. 

\begin{corollary}  \label{cor-L}
    If an L-shaped net of size $m\times n$ is sufficiently close to the L-shaped square net and satisfies one of %the following 
    conditions (i) or (ii) in Theorem~\ref{th-mxn}, then it is contained in exactly one deformable  $m\times n$ net. 
    %%%%%
    %\begin{enumerate}
    %\item[\textup{(i)}] each two neighboring faces in the $m\times 1$ sub-net or each two neighboring faces in the $1\times n$ sub-net are affine-symmetric with respect to their common edge, or
    %%    which indices satisfy $i=0,1$, $j=0,\ldots,m$ or any two neighbour faces which indices satisfy $i=0,\ldots,n$, $j=0,1$ are affine-symmetric with respect to their common side.
    %\item[\textup{(ii)}] each two neighboring faces have equal opposite ratios with respect to their common edge.
    %\end{enumerate}
    %%%%%%%
    Moreover, the resulting $m\times n$ net continuously depends on the given L-shaped net.
\end{corollary}

The reader interested in the proofs of the theorem and the corollary can proceed to Section~\ref{ssec-proof-deformable-mxn}, and now we discuss geometric properties of the resulting classes (i) and (ii) of deformable $m\times n$ nets. %\mscomm{Check if the assumption of being close to the square L-shaped net can be dropped.}

\subsection{Geometric properties}
\label{ssec-properties-deformable-mxn}
%
%\ \ \ \\
%\mscomm{Here we add geometric properties of deformable $m\times n$ nets, e.g., relation to cone nets}

\subsubsection{Class (i)}

Deformable nets from class (i) are related to cone nets, which are studied %discussed in% the work
in~\cite{kilian2023smooth}. 

An $1 \times n$ net with points $P_{ij}$, where $0 \leq i \leq 1$, $0 \leq j \leq n$, is called a \emph{cone strip} if all the lines $P_{0j}P_{1j}$ for $0 \leq j \leq n$ are either concurrent or parallel. See Figure~\ref{fig:cone-cylind-strip}(left). 
An $m \times n$ net is called a \emph{cone net} if each $1\times n$ or each $m\times 1$ sub-net is a cone strip. See Figure~\ref{fig:cone-cylinder net}. Applying Proposition~\ref{th-old} repeatedly (see the red lines in Figure~\ref{figure:properties}), we see that each net from class (i) is a cone net.

The other condition depicted in Figure~\ref{figure:properties} leads us to the following notions. By \emph{a cone-cylinder strip} we mean a cone strip such that $P_{0,j}P_{0,j+1}\parallel P_{1,j}P_{1,j+1}$ for each $0\leq j\leq n-1$. See Figure~\ref{fig:cone-cylind-strip}(middle).
A \emph{doubled cone-cylinder strip} is a cone strip such that 
$P_{0,j}P_{0,j+2}\parallel P_{1,j}P_{1,j+2}$ 
and $P_{0,j},P_{0,j+2},P_{1,j},P_{1,j+2}$ are non-collinear
%$P_{0,j}P_{0,j+2}P_{1,j+2}P_{1,j}$ is a trapezoid or parallelogram 
for each $0\leq j\leq n-2$. See Figure~\ref{fig:cone-cylind-strip}(right).
We see that a doubled cone-cylinder strip consists of two interleaved cone-cylinder strips forming a cone strip together. A \emph{cone-cylinder net} and a \emph{doubled cone-cylinder net} are now defined analogously to a cone net. See Figure~\ref{fig:cone-cylinder net}. The points $P_{00},P_{0n},P_{m0},P_{mn}$ are called the \emph{corners} of an $m\times n$ net. 

%By \emph{a cylinder-cone strip} we mean a cone strip such that $P_{0,j}P_{0,j+2}\parallel P_{1,j}P_{1,j+2}$ for each $0\leq j\leq n-2$. %An $m \times n$ net is called \emph{a cone-net} if each $1\times n$ or each $m\times 1$ sub-nets is cone strip.  

\begin{figure}[htbp]
    \centering
    \includegraphics[scale=0.169]{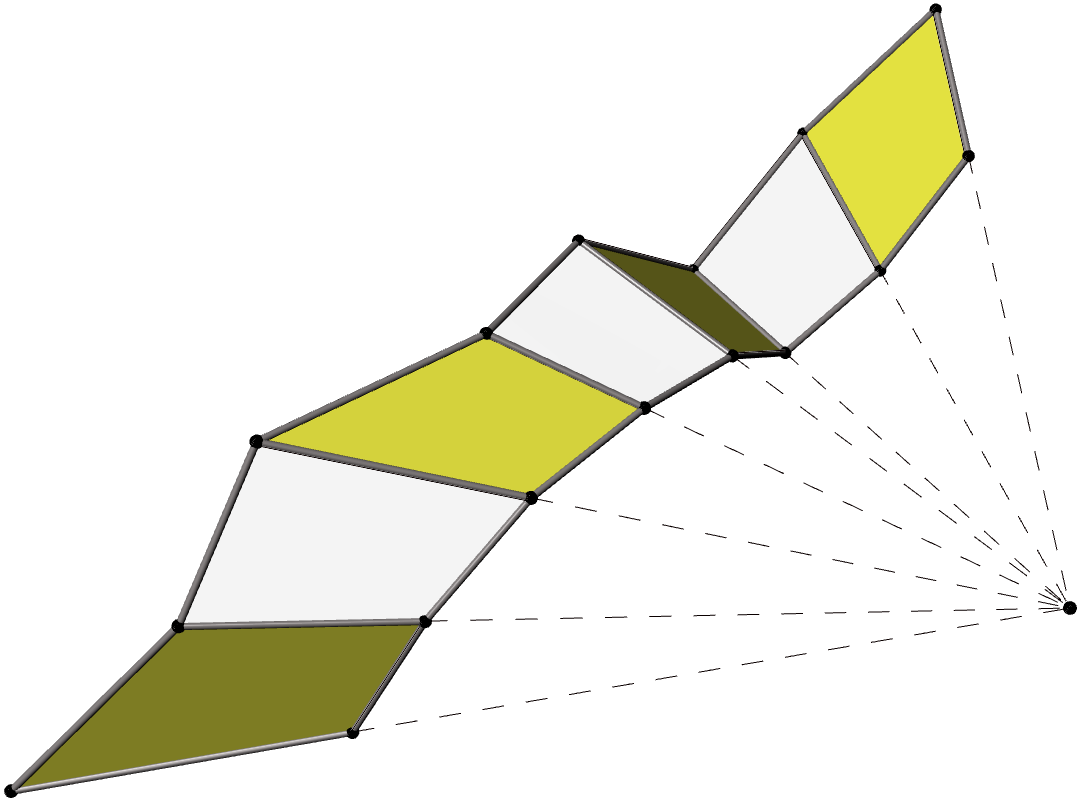}\quad\includegraphics[scale=0.135]{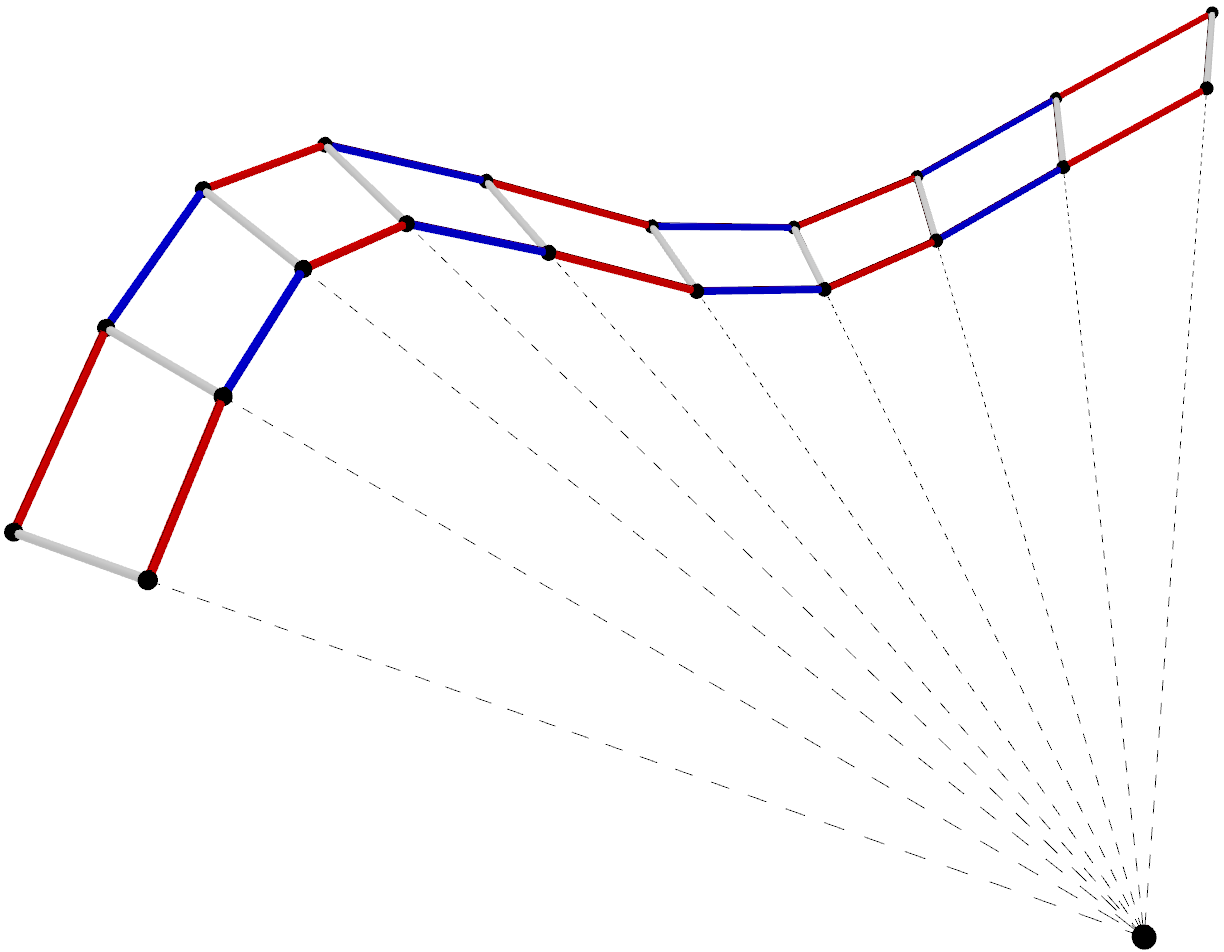}%\quad\includegraphics[scale=0.16]{fig/d9.png}
    \quad\includegraphics[scale=0.154]{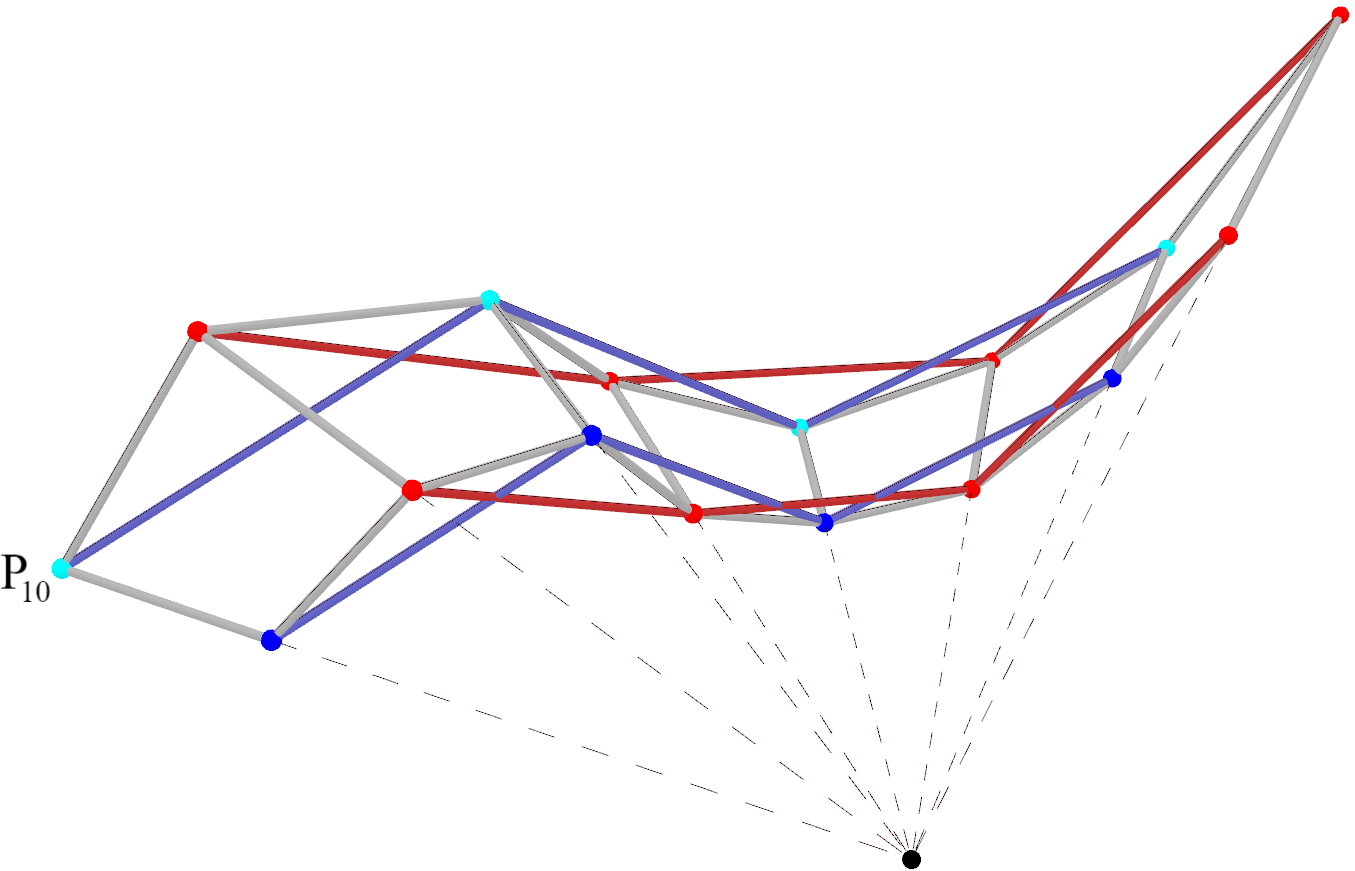}
    \caption{Left: a cone strip. Middle: a cone-cylinder strip. Right: a doubled cone-cylinder strip. Corresponding blue (red) segments are parallel, and the dashed lines are concurrent. The doubled cone-cylinder strip consists of two interleaved cone-cylinder strips shown in red and shades of blue respectively. Given the red strip, one can construct the blue one so that they form a doubled cone-cylinder strip together: the dark blue points can be chosen (almost) arbitrarily, the point $P_{10}$ can be freely chosen on the dashed line, and the remaining light blue points are determined by the parallelism of the resulting blue segments. The construction can then be propagated to the next strips in a net. 
    }
    \label{fig:cone-cylind-strip}
\end{figure}

\begin{figure}[htbp]
    \centering
    \includegraphics[scale=0.06]{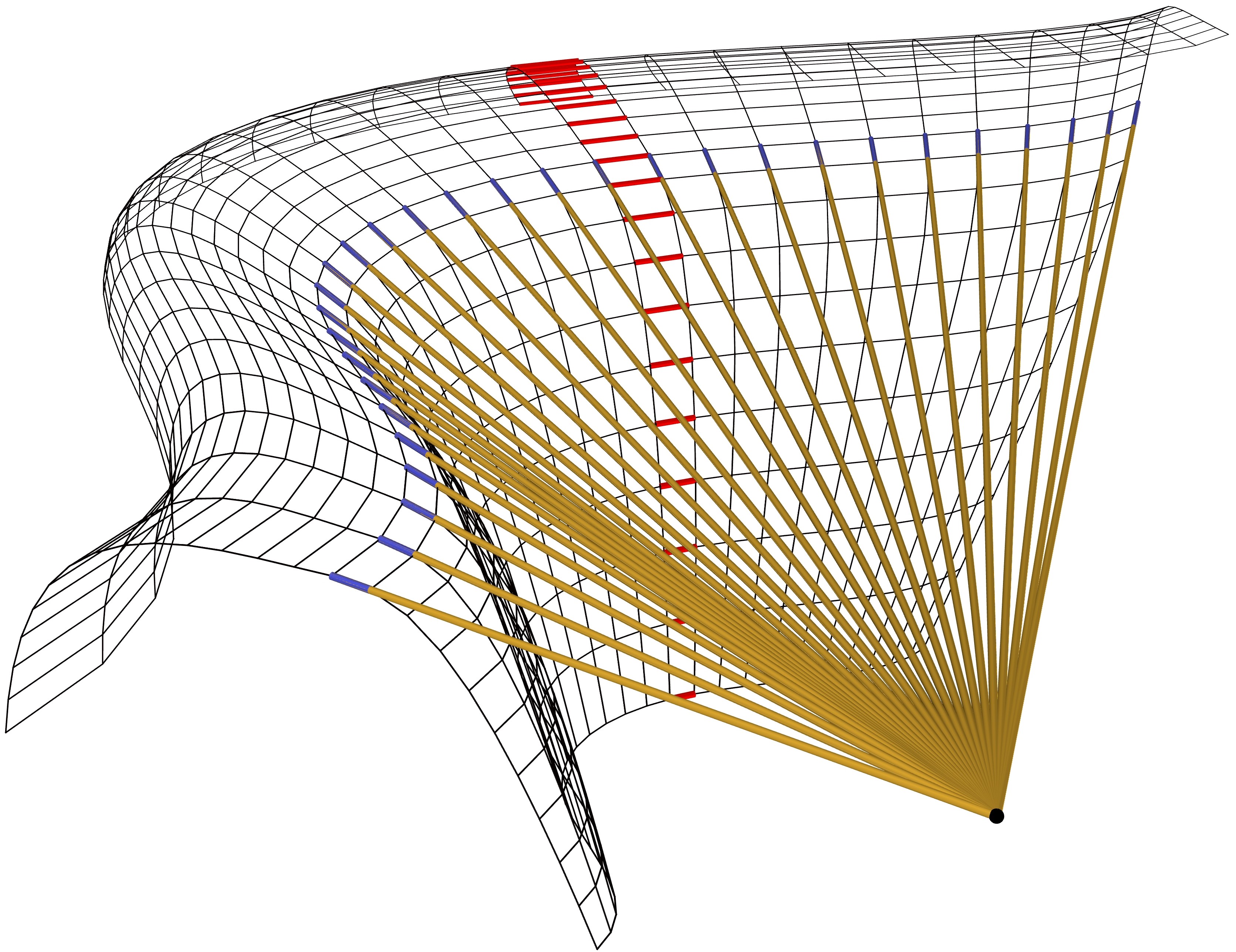}\quad\quad\quad\includegraphics[scale=0.063]{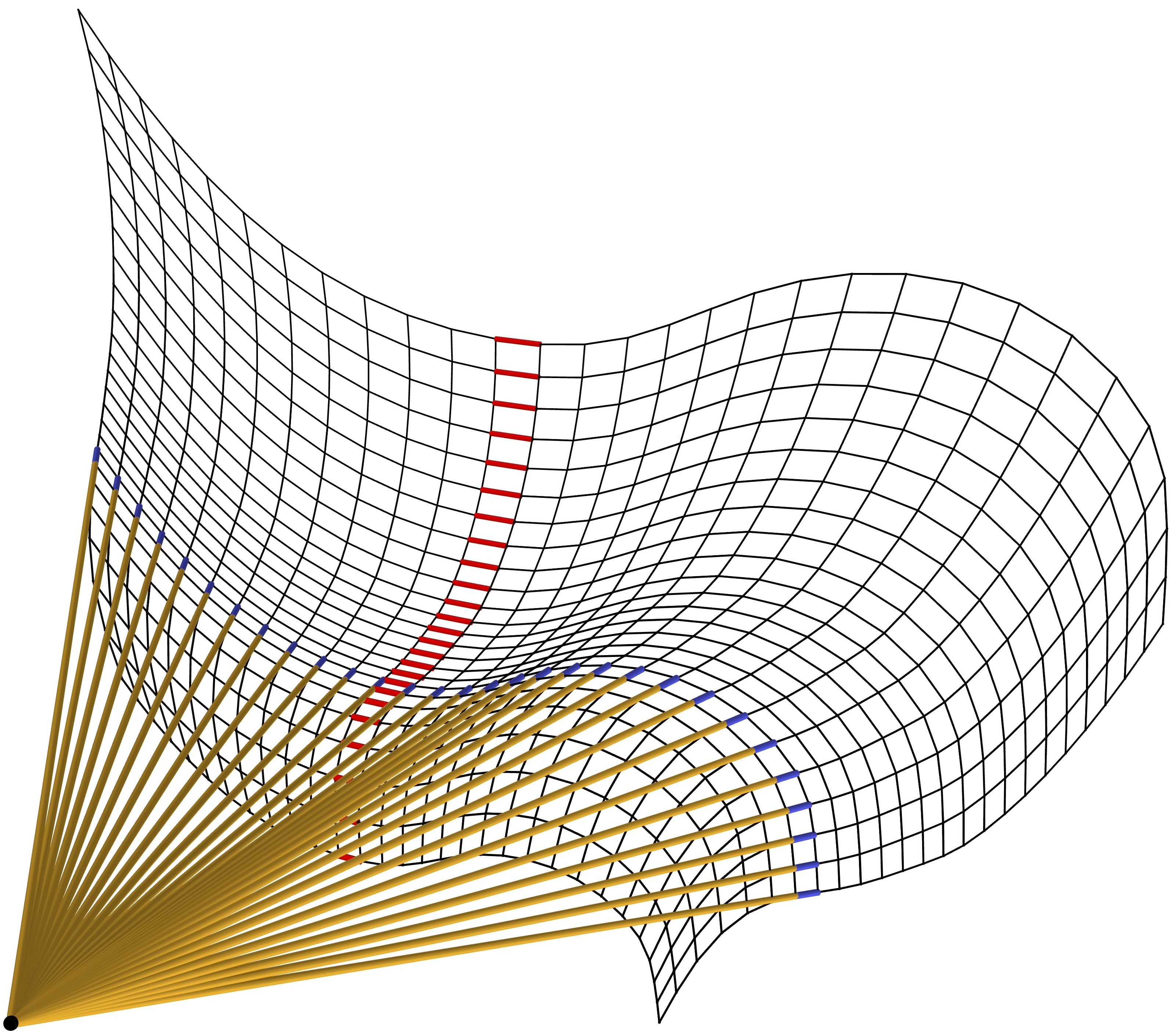}

    \caption{Cone-cylinder nets. Left: the deformable net from Figure~\ref{fig:type i}. Right: another example. The yellow lines are concurrent and the red lines are parallel.  }
    \label{fig:cone-cylinder net}
\end{figure}

As an immediate consequence of Proposition~\ref{th-old}, we get a geometric characterization of class (i).

\begin{proposition}
    \label{th-cyl-cone-str}
    An $m\times n$ net %, where $m,n\ge 2$, 
    such that the corners of each $2\times 1$ and $1\times 2$ sub-net are non-collinear
    satisfies condition (i) of Theorem~\ref{th-mxn}
    if and only if it is a doubled cone-cylinder net.
    %if and only if each its $2\times2$ sub-net satisfies condition (i) of Theorem~\ref{th-classification}.
    %Each $2\times2$ sub-net of an $m\times n$ net, where $m,n\ge 2$, satisfies condition (i) of Theorem~\ref{th-classification} if and only if all $1\times n$ or all $m\times 1$ sub-nets are cylinder-cone strips.
\end{proposition}

We see that any $m\times n$ net from class (i) consists of two cone-cylinder nets forming a cone net together. Given one of the two cone-cylinder nets,
one can construct the other one so that they comprise a net from class (i) together. %, at least when the faces of the former are %sufficiently 
%close to squares. 
See Figure~\ref{fig:cone-cylind-strip}(right).
%This can be proved analogously to Corollary~\ref{cor-L}, but we omit the details.

This suggests taking a closer look at the cone-cylinder nets. They have applications themselves~\cite{glymphetal}.

\begin{proposition} \label{p-cone-cylinder-nets}
An $m\times n$ net with the points $P_{ij}$, where $0 \leq i \leq m$, $0 \leq j \leq n$, is a cone-cylinder net if and only if up to interchanging the indices $i$ and $j$ we have
\begin{equation}
    \label{eq-p-cone-cylinder-nets}
    P_{ij}=a_i + \sigma_i b_j,
    %P_{ij}=a(i) + \sigma(i) b(j), 
    \qquad%\text{for each }
    0 \leq i \leq m\text{ and } 0 \leq j \leq n,
\end{equation}
for some $a_0,\dots,a_m,b_0,\dots,b_n\in\mathbb{R}^d$ and $\sigma_0,\dots,\sigma_m\in\mathbb{R}$.
%functions $a\colon \{0,\dots,m\}\to\mathbb{R}^d$, $b\colon \{0,\dots,n\}\to\mathbb{R}^d$, and $\sigma\colon \{0,\dots,m\}\to\mathbb{R}$.
\end{proposition}
%Let us move to the type (ii).

\begin{proof} \emph{The `only if' part.} Take a cone-cylinder net. Up to interchanging the indices $i$ and $j$, we may assume that all $1\times n$ sub-nets are cone-cylinder strips. Then for each $0 \leq i < m$, the broken line $P_{i+1,0}\dots P_{i+1,n}$ is obtained from $P_{i,0}\dots P_{i,n}$ by a central similarity or a translation. Then all the broken lines are obtained from $P_{0,0}\dots P_{0,n}$ by central similarities or translations. This means that~\eqref{eq-p-cone-cylinder-nets} holds for $b_j:=P_{0,j}$ and suitable $a_i$ and $\sigma_i$.
%$b(j):=P_{0,j}$ and suitable $a(i)$ and $\sigma(i)$.

\emph{The `if' part.} Take an $m\times n$ net given by~\eqref{eq-p-cone-cylinder-nets}.
Any broken line $P_{i,0}\dots P_{i,n}$ (with constant $i$) arises from the broken line $b_0\dots b_n$ %$b(0)\dots b(n)$ 
by uniform scaling
with factor $\sigma_i$ and subsequent translation by $a_i$. Hence, two consecutive broken lines  $P_{i,0}\dots P_{i,n}$ and $P_{i+1,0}\dots P_{i+1,n}$ are related by a central similarity or a translation in case of $\sigma_i=\sigma_{i+1}$. 
Thus, the broken lines have parallel edges and are connected by a cone, whose vertex $c_{i}$ is easily seen to be 
$$ c_{i}=\frac{\sigma_{i+1}a_{i}-\sigma_{i}a_{i+1}}{\sigma_{i+1}-\sigma_{i}}.$$ 
%$$ c_{i}=\frac{\sigma(i+1)a(i)-\sigma(i)a(i+1)}{\sigma(i+1)-\sigma(i)}.$$ 
%
The cone becomes a cylinder with ruling direction $a_{i}-a_{i+1}$ for $\sigma_{i}=\sigma_{i+1}$. We get a cone-cylinder~net. 
\end{proof}

As a %an immediate 
corollary, %of Propositions~\ref{p-cone-cylinder-nets}, \ref{th-cyl-cone-str}, and~\ref{th-mxn}, 
%we get the following.
any cone-cylinder net %(such that the corners of each $2\times 1$ and $1\times 2$ sub-net are non-collinear) 
(at least, with noncoplanar faces) 
is also a doubled cone-cylinder net,
%\begin{corollary} A cone-cylinder net with disjoint interiors of faces is a doubled cone-cylinder~net. %and 
hence is deformable.    
%\end{corollary}
%\begin{proof}
%It suffices to prove that any cone-cylinder strip is a doubled cone-cylinder strip; then the result follows from Propositions~\ref{th-cyl-cone-str} and~\ref{th-mxn}. By definition, for all $j$ the lines $P_{0j}P_{1j}$ are concurrent or parallel, $P_{0,j}P_{0,j+1}\parallel P_{1,j}P_{1,j+1}$,  and $P_{0,j+1}P_{0,j+2}\parallel P_{1,j+1}P_{1,j+2}$. Then triangles  $P_{0,j}P_{0,j+1}P_{0,j+2}$ and $P_{1,j}P_{1,j+1}P_{1,j+2}$ are homothetic or congruent. Hence $P_{0,j}P_{0,j+2}\parallel P_{1,j}P_{1,j+2}$, as required.
%\end{proof}
%
%Together with Propositions~\ref{th-cyl-cone-str} and~\ref{th-mxn}, this already implies that any cone-cylinder net is deformable. 
One can also construct the deformation explicitly. This will be a special case of a \emph{conical Combescure transform} from \cite[\S4]{kilian2023smooth}. We 
assume $\sigma_i>0$ for simplicity.
%to avoid case-by-case analysis.

%Notice that for any cone-cylinder net~\eqref{eq-p-cone-cylinder-nets} we have $\sigma\ne 0$ everywhere, and we may assume $\sigma>0$ NO!!!

\begin{proposition} \label{p-cone-cylinder-deformation}
    A cone-cylinder net~\eqref{eq-p-cone-cylinder-nets} with all $\sigma_i>0$ is embedded into a family 
\begin{equation} \label{eq:discrete-family}
P_{ij}(t):=
    \underbrace{a_0+\sum_{k=1}^{i} \frac{\left(a_k-a_{k-1}\right)(\sigma_{k}+\sigma_{k-1})}
    {\sqrt{t+\sigma_{k}^2\vphantom{t+\sigma_{k-1}^2}}+\sqrt{t+\sigma_{k-1}^2}} 
    }_{a_{i}(t)}
    +\underbrace{\sqrt{t+\sigma_{i}^2}
    \vphantom{a_0+\sum_{k=1}^{i} \frac{\left(a_i-a_{i-1}\right)(\sigma_{i}+\sigma_{i-1})}
    {\sqrt{t+\sigma_{i}^2}+\sqrt{t+\sigma_{i-1}^2}}}}_{\sigma_i(t)} b_j,
    \qquad %\text{for each }
    t\in [0,\varepsilon],\, 0 \leq i \leq m\text,\, 0 \leq j \leq n,
    %%%%%%%%%%%%%%%%%%%%%%
    %\underbrace{a(0)+\sum_{k=1}^{i} \frac{\left(a(i)-a(i-1)\right)(\sigma(i)+\sigma(i-1))}
    %{\sqrt{t+\sigma(i)^2}+\sqrt{t+\sigma(i-1)^2}} 
    %}_{a(i,t)}
    %+\underbrace{\sqrt{t+\sigma(i)^2}
    %\vphantom{a(0)+\sum_{k=1}^{i} \frac{\sigma(i)+\sigma(i-1)}{\sqrt{t+\sigma(i)^2}\sqrt{t+\sigma(i-1)^2}}\left(a(i)-a(i-1)\right)}}_{\sigma(i,t)} b(j)
    %%%%%%%%%%%%%%%%%%%%%%
\end{equation} 
    of cone-cylinder nets which are 
    %related to each other by 
    its area-preserving Combescure transformations, for some $\varepsilon>0$.
\end{proposition}

\begin{proof} By Proposition~\ref{p-cone-cylinder-nets}, net~\eqref{eq:discrete-family} is a cone-cylinder net for each $t$ sufficiently close to $0$, because it is of the form $P_{ij}(t)=a_i(t)+\sigma_i(t)b_j$. For $t=0$, it coincides with~\eqref{eq-p-cone-cylinder-nets} because all $\sigma_i>0$. Since 
    \begin{align*}
    \overrightarrow{P_{i,j-1}(t)P_{ij}(t)}&=
    \sqrt{t+\sigma_i^2}\cdot(b_j-b_{j-1})=
    \frac{\sqrt{t+\sigma_i^2}}{\sigma_i}\cdot\overrightarrow{P_{i,j-1}P_{ij}},\\
    \overrightarrow{P_{i-1,j}(t)P_{ij}(t)}&=
    \frac{\left(a_i-a_{i-1}\right)(\sigma_{i}+\sigma_{i-1})}
    {\sqrt{t+\sigma_{i}^2\vphantom{t+\sigma_{i-1}^2}}+\sqrt{t+\sigma_{i-1}^2}}+
    \left(\sqrt{t+\sigma_{i}^2\vphantom{t+\sigma_{i-1}^2}}-\sqrt{t+\sigma_{i-1}^2}\right)b_j
    =\frac{(\sigma_{i}+\sigma_{i-1})\overrightarrow{P_{i-1,j}P_{ij}}}
    {\sqrt{t+\sigma_{i}^2\vphantom{t+\sigma_{i-1}^2}}+\sqrt{t+\sigma_{i-1}^2}},
    \end{align*}
    for each $t\in[0,\varepsilon]$, nets~\eqref{eq:discrete-family} and~\eqref{eq-p-cone-cylinder-nets} are Combescure transforms. Area preservation follows from
    $$
    \left(P_{i,j-1}(t)P_{ij}(t)+P_{i-1,j-1}(t)P_{i-1,j}(t)\right)\cdot P_{i-1,j}(t)P_{ij}(t)=
    |b_j-b_{j-1}|(\sigma_{i}+\sigma_{i-1})P_{i-1,j}P_{ij} =\mathrm{const}.
    $$
    For distinct $t\in[0,\varepsilon]$, nets~\eqref{eq:discrete-family} are non-congruent because the edge lengths $P_{i,j-1}(t)P_{ij}(t)$ are distinct.
\end{proof}

\subsubsection{Class~(ii)} All we said about $2\times 2$ nets from class~(ii) in Section~\ref{ssec-properties-deformable-2x2} remains true for $m\times n$ nets. % of class~(ii). 
%In particular, 
Propositions~\ref{p-Koenigs}, \ref{ex-ii-geo}, \ref{p-class-ii} remain true with the same proofs, if ``$2\times 2$'' is replaced by ``$m\times n$'':
%(an $m\times n$ \emph{K\oe nigs net} is the one such that all its $2\times 2$ sub-nets are K\oe nigs nets): 

\begin{proposition}
    \label{th-christoffel}
    An $m\times n$ net satisfies condition (ii) of Theorem~\ref{th-mxn} if and only if it has a Christoffel dual with the same areas of corresponding faces. For such an $m\times n$ net, a family of area-preserving Combesure transformations is given by~\eqref{eq-ii-geo} for $0\le i\le m$ and $0\le j\le n$. 
    %Each $2\times2$ sub-net of an $m\times n$ net, where $m,n\ge 2$, satisfies condition (i) of Theorem~\ref{th-classification} if and only if all $1\times n$ or all $m\times 1$ sub-nets are cylinder-cone strips.
\end{proposition}

In class~(ii), the following \emph{zig-zag phenomenon} occurs. Recall that by Proposition~\ref{p-opposite-ratio-equivalent-def} the rays extending the sides 
$BC$ and $AD$ of a convex quadrilateral $ABCD$ intersect if and only if the opposite ratio of $ABCD$ with respect to $AB$ is greater than $1$. If two quadrilaterals $ABCD$ and $ABC'D'$ in the plane have a common side $AB$ (and no other common points), and opposite ratios with respect to that side are equal, then the pairs of lines $BC, AD$ and  $BC', AD'$ intersect on the opposite sides of the line $AB$, unless $BC\parallel AD$ and $BC'\parallel AD'$. See Figure~\ref{figure:inter-point-oppos-side}. For spatial nets from class~(ii), the same happens in the projection to any plane. This shows that the `zig-zag' shape of the discrete parameter lines like in Figure~\ref{fig:type ii} is unavoidable unless the parameter lines have parallel edges. The same is true for class~(i), but only for the isolines in one of the two directions. Cf.~\cite[Figure~5 to the right]{kilian2023smooth}.

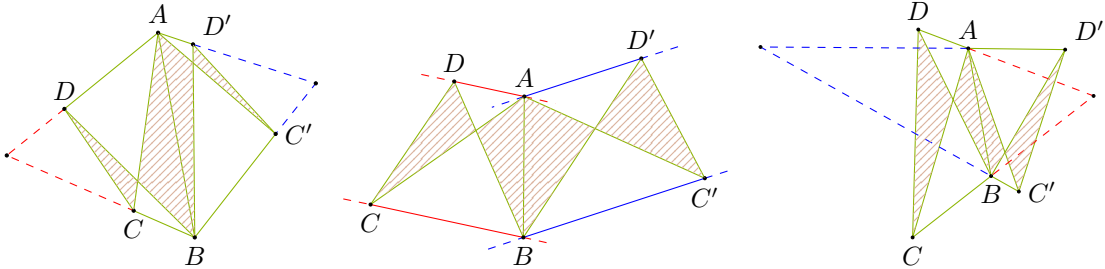
\begin{figure}[htbp]
  \centering
\begin{tikzpicture}[scale=0.85]

\coordinate (R) at (-2.72553, -2.25912);
\coordinate (A) at (-0.359, -0.33859);
\coordinate (B) at (0.21442, -3.5425);
\coordinate (C) at (-0.74129, -3.12381);
\coordinate (D) at (-1.82443, -1.53096);
\coordinate (C') at (1.4796, -1.92234);
\coordinate (D') at (0.18712, -0.52063);
\coordinate (R') at (2.10764, -1.13047);
\coordinate (P) at (-0.67757, -2.66871);
\coordinate (P') at (0.18712, -0.8119);

\draw[applegreen] (A) -- (B) -- (C) -- (D) -- (A) -- (D') -- (C') -- (B);
\draw[applegreen] (C) -- (A) -- (C');
\draw[applegreen] (D') -- (B) -- (D);
\draw[red, dashed, thin] (D) -- (R) -- (C);
\draw[blue, dashed, thin] (D') -- (R') -- (C');

\filldraw[black] (B) circle (0.7pt) node[anchor=north]{$B$};
\filldraw[black] (C) circle (0.7pt) node[anchor=north]{$C$};
\filldraw[black] (A) circle (0.7pt) node[anchor=south]{$A$};
\filldraw[black] (C') circle (0.7pt) node[anchor=west]{$C'$};
\filldraw[black] (D) circle (0.7pt) node[anchor=south]{$D$};
\filldraw[black] (D') circle (0.7pt) node[anchor=south west]{$D'$};
\filldraw[black] (R) circle (0.7pt) node[anchor=east]{};
\filldraw[black] (R') circle (0.7pt) node[anchor=west]{};
%\filldraw[black] (P) circle (0.7pt) node[anchor=west]{};
%\filldraw[black] (P') circle (0.7pt) node[anchor=east]{};

\path[name path=line 1] (B) -- (D);
\path[name path=line 2] (A) -- (C);
\path[name path=line 3] (B) -- (D');
\path[name path=line 4] (A) -- (C');

\fill[name intersections={of=line 1 and line 2, by=X1}];
\fill[name intersections={of=line 3 and line 4, by=X2}];

\fill[pattern=north east lines, pattern color=antiquebrass!70] (C) -- (X1) -- (D) -- cycle;
\fill[pattern=north east lines, pattern color=antiquebrass!70] (A) -- (X1) -- (B) -- cycle;
\fill[pattern=north east lines, pattern color=antiquebrass!70] (B) -- (X2) -- (A) -- cycle;
\fill[pattern=north east lines, pattern color=antiquebrass!70] (C') -- (X2) -- (D') -- cycle;

\end{tikzpicture}
\begin{tikzpicture}[scale=0.57]

\coordinate (R) at (-2.72553, -2.25912);
\coordinate (A) at (3.73856, 6.18675);
\coordinate (B) at (3.71663, 2.8976);
\coordinate (C) at (0.15339, 3.66507);
\coordinate (D) at (2.10495, 6.53759);
\coordinate (C') at (7.94867, 4.27904);
\coordinate (D') at (6.46855, 7.07482);
\coordinate (X1) at (1.33748, 6.70205);
\coordinate (X2) at (-0.48252, 3.79663);
\coordinate (X3) at (4.39639, 2.7441);
\coordinate (X4) at (4.46217, 6.02229);
\coordinate (Y1) at (7.32373, 7.34891);
\coordinate (Y2) at (2.98205, 5.93458);
\coordinate (Y3) at (2.89434, 2.6235);
\coordinate (Y4) at (8.6065, 4.48735);
\coordinate (P) at (2.60928, 5.39735);
\coordinate (P') at (5.39409, 5.43024);

\draw[red, thin] (D) -- (A);
\draw[red, thin] (B) -- (C);
\draw[red, dashed, thin] (X1) -- (D);
\draw[red, dashed, thin] (A) -- (X4);
\draw[red, dashed, thin] (B) -- (X3);
\draw[red, dashed, thin] (X2) -- (C);

\draw[blue, thin] (A) -- (D');
\draw[blue, thin] (B) -- (C');
\draw[blue, dashed, thin] (Y1) -- (D');
\draw[blue, dashed, thin] (A) -- (Y2);
\draw[blue, dashed, thin] (Y3) -- (B);
\draw[blue, dashed, thin] (C') -- (Y4);

\draw[applegreen] (C) -- (D);
\draw[applegreen] (D') -- (C');
\draw[applegreen] (A) -- (B);
\draw[applegreen] (C) -- (A) -- (C');
\draw[applegreen] (D') -- (B) -- (D);

\filldraw[black] (B) circle (1pt) node[anchor=north]{$B$};
\filldraw[black] (C) circle (1pt) node[anchor=north]{$C$};
\filldraw[black] (A) circle (1pt) node[anchor=south]{$A$};
\filldraw[black] (C') circle (1pt) node[anchor=north]{$C'$};
\filldraw[black] (D) circle (1pt) node[anchor=south]{$D$};
\filldraw[black] (D') circle (1pt) node[anchor=south]{$D'$};
%\filldraw[black] (P) circle (1pt) node[anchor=west]{};
%\filldraw[black] (P') circle (1pt) node[anchor=east]{};

\path[name path=line 1] (B) -- (D);
\path[name path=line 2] (A) -- (C);
\path[name path=line 3] (B) -- (D');
\path[name path=line 4] (A) -- (C');

\fill[name intersections={of=line 1 and line 2, by=X1}];
\fill[name intersections={of=line 3 and line 4, by=X2}];

\fill[pattern=north east lines, pattern color=antiquebrass!70] (C) -- (X1) -- (D) -- cycle;
\fill[pattern=north east lines, pattern color=antiquebrass!70] (A) -- (X1) -- (B) -- cycle;
\fill[pattern=north east lines, pattern color=antiquebrass!70] (B) -- (X2) -- (A) -- cycle;
\fill[pattern=north east lines, pattern color=antiquebrass!70] (C') -- (X2) -- (D') -- cycle;

\end{tikzpicture}
\begin{tikzpicture}[scale=0.53]

\coordinate (R) at (2.7847, -1.53179);
\coordinate (A) at (-0.359, -0.33859);
\coordinate (B) at (0.21442, -3.5425);
\coordinate (C) at (-1.75432, -5.08407);
\coordinate (D) at (-1.61179, 0.13471);
\coordinate (C') at (0.90989, -3.93287);
\coordinate (D') at (2.07206, -0.36962);
\coordinate (R') at (-5.56973, -0.30384);
\coordinate (P) at (-0.73468, -1.6195);
\coordinate (P') at (0.55905, -2.94612);

\draw[applegreen] (A) -- (B) -- (C) -- (D) -- (A) -- (D') -- (C') -- (B);
\draw[applegreen] (C) -- (A) -- (C');
\draw[applegreen] (D') -- (B) -- (D);
\draw[red, dashed, thin] (A) -- (R) -- (B);
\draw[blue, dashed, thin] (A) -- (R') -- (B);

\filldraw[black] (B) circle (1pt) node[anchor=north]{$B$};
\filldraw[black] (C) circle (1pt) node[anchor=north]{$C$};
\filldraw[black] (A) circle (1pt) node[anchor=south]{$A$};
\filldraw[black] (C') circle (1pt) node[anchor=west]{$C'$};
\filldraw[black] (D) circle (1pt) node[anchor=south]{$D$};
\filldraw[black] (D') circle (1pt) node[anchor=south west]{$D'$};
\filldraw[black] (R) circle (1pt) node[anchor=west]{};
\filldraw[black] (R') circle (1pt) node[anchor=east]{};
%\filldraw[black] (P) circle (1pt) node[anchor=west]{};
%\filldraw[black] (P') circle (1pt) node[anchor=east]{};

\path[name path=line 1] (B) -- (D);
\path[name path=line 2] (A) -- (C);
\path[name path=line 3] (B) -- (D');
\path[name path=line 4] (A) -- (C');

\fill[name intersections={of=line 1 and line 2, by=X1}];
\fill[name intersections={of=line 3 and line 4, by=X2}];

\fill[pattern=north east lines, pattern color=antiquebrass!70] (C) -- (X1) -- (D) -- cycle;
\fill[pattern=north east lines, pattern color=antiquebrass!70] (A) -- (X1) -- (B) -- cycle;
\fill[pattern=north east lines, pattern color=antiquebrass!70] (B) -- (X2) -- (A) -- cycle;
\fill[pattern=north east lines, pattern color=antiquebrass!70] (C') -- (X2) -- (D') -- cycle;

\end{tikzpicture}

\caption{The position of the intersection points of the opposite sides of quadrilaterals $ABCD$ and $ABC'D'$ having equal opposite ratios with respect to the common side $AB$.}
\label{figure:inter-point-oppos-side}
\end{figure}

%We conclude this subsection with a necessary condition for the deformability, arising from the infinitesimal deformability. An $m\times n$ net is a \emph{K\oe nigs net} if all its $2\times 2$ sub-nets are K\oe nigs nets (defined in Section~\ref{ssec-properties-deformable-2x2}). By Propositions~\ref{th-mxn} and~\ref{p-Koenigs}, a deformable $m\times n$ net is a K\oe nigs net.

%\begin{proposition} \label{th-Koenigs}
%    A deformable $m\times n$ net is a K\oe nigs net.
%\end{proposition}  

%We can write the above Theorem for $m\times n$ net, but we need to give an accurate definition for dual net.

\subsection{Proof of the classification}
\label{ssec-proof-deformable-mxn}

For the proof of the %proposition and the theorem
classification of deformable $m\times n$ nets, we need the following lemmas describing deformations of their small sub-nets.

\begin{lemma}
    \label{lemma-uniq-def-1x1}
    For each segment $e'$ parallel and sufficiently close in length to an edge $e$ of a $1\times1$ net, there exists a unique deformation of this net such that $e'$ is its edge corresponding to $e$.  
\end{lemma}
\begin{proof}[Proof]
    Denote by $ABCD$ the $1\times1$ net with $AB:=e$. Let $A', B'$ be the endpoints of the segment $e'$ so that $\overrightarrow{AB}$ and $\overrightarrow{A'B'}$ have the same direction. Consider arbitrary points $C', D'$ on the two rays parallel to $BC, AD$ starting at $B', A'$, respectively, such that $\overrightarrow{C'D'}$ and $\overrightarrow{CD}$ have the same direction. Since $e'$ has a sufficiently close length to $e$, we may continuously move the segment $C'D'$  so that it remains parallel to $CD$, the quadrilateral $A'B'C'D'$ remains convex, and the area of $A'B'C'D'$ changes monotonically and continuously in a small interval containing the value of the area of $ABCD$. Therefore, we can choose positions of $C', D'$ so that the quadrilaterals $ABCD$ and $A'B'C'D'$ have the same areas. The resulting quadrilateral $A'B'C'D'$ is the desired unique deformation.       
\end{proof}

\begin{lemma}
    \label{lemma-uniq-def-2x2}
    For each segment $e'$ parallel and sufficiently close in length to an edge $e$ of a deformable $2\times2$ net, there is a unique deformation of the net such that $e'$ is its edge corresponding~to~$e$.  
\end{lemma}
\begin{proof}[Proof]
    \emph{Uniqueness.} Consider a deformation of the $2\times2$ net
    such that $e'$ is its edge corresponding to $e$. By Lemma~\ref{lemma-uniq-def-1x1}, each  
    $1\times 1$ sub-net containing the edge $e'$ is determined uniquely. This means that the faces containing $e'$ and all the edges of those faces are determined uniquely. Applying Lemma~\ref{lemma-uniq-def-1x1} to the adjacent faces repeatedly, we get that they all are determined uniquely.

    \emph{Existence}. By Theorem~\ref{th-classification}, a deformable $2\times2$ net 
    satisfies one of the conditions~(i)--(ii) in that theorem. Take a particular family of deformations given by Example~\ref{ex-ii}.
    %%%%%%%%%%%%
    %The deformable $2\times2$ net is contained in a continuous family of non-congruent parallel nets with equal areas of corresponding faces. 
    Consider the ratio of the lengths of $e$ and its corresponding edge in a net from this family. By the `moreover' part of Example~\ref{ex-ii}, such ratios form an interval containing $1$ in its interior. %, which has a non-zero length because of the uniqueness part. 
    Since $e'$ is close enough in length to $e$, there is a net in this family that the edge corresponding to $e$ has the same length as $e'$. %It remains to apply a translation.
    %%%%%%%
    %%Let $f_1,f_2,f_3,f_4$ be faces of the deformable $2\times2$ net enumerated clockwise so that $e\subset f_1$ and $e\not\subset f_4$. Applying Lemma~\ref{lemma-uniq-def-1x1}, we construct a quadrilateral $f_1'$ on the segment $e'$ which have parallel sides and equal area to $f_1$. Again using Lemma~\ref{lemma-uniq-def-1x1}, on the side $f_1'$ which is parallel to common sides of $f_1,f_2$, we construct the quadrilateral $f'_2$ that has parallel sides and equal area to $f_2$. Similarly, define quadrilaterals $f_3’,f_4’$, i.e. they are constructed on the sides of $f'_2,f'_3$ and have parallel sides and equal areas to $f_3,f_4$, respectively. Now, we show that the quadrilaterals $f_4',f_1'$ have a common side. 
    %%By the above, the other sides of this deformation are determined uniquely, meaning that quadrilaterals $f_4',f_1'$ have a common side.
    %Finally, we show that obtained deformation is not congruent to the net. Indeed, if they have corresponding equal edge, then by uniqueness all the corresponding edges are equal which contradicts $e\not=e'$.
\end{proof}

The `if' part in Theorem~\ref{th-mxn} follows from Theorem~\ref{th-classification} and the following lemma.

\begin{lemma}
    \label{lemma-def-2x2}
    If each $2\times2$ sub-net of an $m\times n$ net is deformable, then the net is deformable.
\end{lemma}
\begin{proof}[Proof]

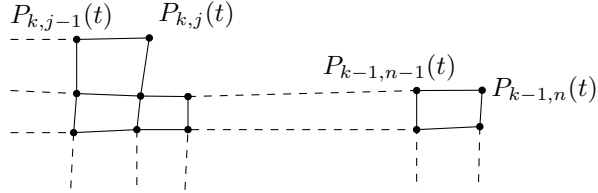
\begin{figure}[htbp]
\centering
\begin{tikzpicture}[scale=0.4]
%\draw[black, thin] (-1.6, 7.1) -- (-0.3, 7.1) -- (-0.3, 5.7) -- (-1.7, 5.6) -- (-1.6, 7.1) -- (-1.90231, 8.70474) -- (-0.3, 8.7) -- (-0.3, 7.1);
\draw[black, thin] (7.2,7) -- (7.2, 8.8) -- (9.61303, 8.84968) -- (9.32315, 6.91716) -- (10.9, 6.9) -- (10.9, 5.8) -- (9.2, 5.8) -- (7.1, 5.7) -- (7.2,7) -- (9.32315, 6.91716) -- (9.2, 5.8);
\draw[black, thin] (18.5, 7.1) -- (20.67669, 7.11041) -- (20.6, 5.9) -- (18.5, 5.8) -- (18.5, 7.1);
%\draw[black, thin] (20.3, -3.8) -- (20.49356, -5.69926) -- (18.6, -5.5) -- (18.5, -3.8) -- (20.3, -3.8);
%\draw[black, thin] (8.4, -5.4) -- (8.5, -3.7) -- (10.3, -3.7) -- (10.1, -5.4) -- (8.4, -5.4) -- (6.9, -5.4) -- (6.9, -3.7) -- (8.5, -3.7);
%\draw[black, thin] (-0.6, -5.4) -- (-2.54833, -5.25769) -- (-2.8, -3.6) -- (-0.5, -3.6) -- (-0.6, -5.4);

\draw[black, thin, dashed] (5, 8.8) -- (7.2, 8.8);
\draw[black, thin, dashed] (10.9, 6.9) -- (18.5, 7.1);
\draw[black, thin, dashed] (20.6, 5.9) -- (20.55, 3.8);
\draw[black, thin, dashed] (5, 7.1) -- (7.2, 7);
\draw[black, thin, dashed] (9.2, 5.8) -- (9.2, 3.8);
\draw[black, thin, dashed] (5, 5.7) -- (7.1, 5.7) -- (7, 3.8);
\draw[black, thin, dashed] (10.8, 3.8) -- (10.9, 5.8) -- (18.5, 5.8) -- (18.5, 3.8);

%\draw[black, thin, dashed] (-0.3, 8.7) -- (7.2, 8.8);
%\draw[black, thin, dashed] (10.9, 6.9) -- (18.5, 7.1);
%\draw[black, thin, dashed] (20.6, 5.9) -- (20.3, -3.8);
%\draw[black, thin, dashed] (18.6, -5.5) -- (10.1, -5.4);
%\draw[black, thin, dashed] (6.9, -5.4) -- (-0.6, -5.4);
%\draw[black, thin, dashed] (-2.8, -3.6) -- (-1.7, 5.6);
%\draw[black, thin, dashed] (-0.3, 7.1) -- (7.2, 7);
%\draw[black, thin, dashed] (9.2, 5.8) -- (8.5, -3.7);

%\draw[black, thin, dashed] (-0.5, -3.6) -- (-0.3, 5.7) -- (7.1, 5.7) -- (6.9, -3.7) -- (-0.5, -3.6);
%\draw[black, thin, dashed] (10.9, 5.8) -- (18.5, 5.8) -- (18.5, -3.8) -- (10.3, -3.7) -- (10.9, 5.8);

%\filldraw[black] (-1.6, 7.1) circle (4pt) node[anchor=east]{$P_{k-1,0}(t)$};
%\filldraw[black] (-0.3, 7.1) circle (4pt) node[anchor=north west]{};
%\filldraw[black] (-0.3, 5.7) circle (4pt) node[anchor=north east]{};
%\filldraw[black] (-1.90231, 8.70474) circle (4pt) node[anchor=south east]{$P_{k,0}(t)$};
%\filldraw[black] (-1.7, 5.6) circle (4pt) node[anchor=north east]{$P_{k-2,0}(t)$};
%\filldraw[black] (-0.3, 8.7) circle (4pt) node[anchor=south west]{$P_{k,1}(t)$};

\filldraw[black] (7.2,7) circle (3pt) node[anchor=north west]{};
\filldraw[black] (7.2, 8.8) circle (3pt);
\filldraw[black] (6.7, 8.8) circle (0pt) node[anchor=south]{$P_{k,j-1}(t)$};
\filldraw[black] (9.61303, 8.84968) circle (3pt) node[anchor=south west]{$P_{k,j}(t)$};
\filldraw[black] (9.32315, 6.91716) circle (3pt) node[anchor=north east]{};
\filldraw[black] (10.9, 6.9) circle (3pt) node[anchor=north west]{};
\filldraw[black] (10.9, 5.8) circle (3pt) node[anchor=south east]{};
\filldraw[black] (9.2, 5.8) circle (3pt) node[anchor=west]{};
\filldraw[black] (7.1, 5.7) circle (3pt) node[anchor=north east]{};

\filldraw[black] (18.5, 7.1) circle (3pt);
\filldraw[black] (17.6, 7.1) circle (0pt) node[anchor=south]{$P_{k-1,n-1}(t)$};
\filldraw[black] (20.67669, 7.11041) circle (3pt) node[anchor=west]{$P_{k-1,n}(t)$};
\filldraw[black] (20.6, 5.9) circle (3pt) node[anchor=west]{};
\filldraw[black] (18.5, 5.8) circle (3pt) node[anchor=north east]{};

\end{tikzpicture}\\
\caption{Construction of a deformation layer by layer; see the proof of Lemma~\ref{lemma-def-2x2}.}
\label{figure:net-for_S_i}
\end{figure}

    Let us construct the desired deformation by slightly moving vertices $P_{ij}$ to new positions $P_{ij}(t)$, where $0\leq i\leq m$ and $0\leq j\leq n$, one by one. %one by one in the dictionary order, that is, first for $i=0$ and $j=0,\dots,n$, then for $i=1$ and $j=0,\dots,n$, and so on. See Figure~\ref{figure:net-for_S_i}.  

    First, let us define inductively the points $P_{ij}(t)$ for $i=0,1$ and $j=0,\ldots,n$. Let $P_{00}(t)$ and $P_{10}(t)$ be arbitrary points such that the edges $P_{00}(t)P_{10}(t)$ and $P_{00}(0)P_{10}(0)=P_{00}P_{10}$ are parallel and have sufficiently close lengths (and continuously depend on $t$ but are not congruent for $t\ne 0$). Now, assume that the points $P_{0j}(t)$ and $P_{1j}(t)$, where $0\leq j\leq n-1$ have already been defined. Then by Lemma~\ref{lemma-uniq-def-1x1} there exist unique points $P_{0,j+1}(t)$ and $P_{1,j+1}(t)$ such that the quadrilaterals $P_{0,j}P_{0,j+1}P_{1,j+1}P_{1,j}$ and $P_{0,j}(t)P_{0,j+1}(t)P_{1,j+1}(t)P_{1,j}(t)$ have parallel sides, equal areas,  and depend on $t$ continuously. By induction, we define $P_{ij}(t)$ for $i=0,1$ and all $j=0,\ldots,n$.
    
    Now we define the points $P_{ij}(t)$ inductively for $i=2,\dots, m$ and $j=0,\dots,n$. See Figure~\ref{figure:net-for_S_i}. 

    Assume that for some $1\leq k\leq n-1$ the points $P_{ij}(t)$, where $0\leq i\leq k-1$ and $j=0,\ldots,n$ have already been defined. Let us define the points $P_{ij}(t)$ for $i=k$ and $j=0,\ldots,n$. 
    
    Apply Lemma~\ref{lemma-uniq-def-2x2} to the $2\times2$ sub-net formed by the vertices $P_{ij}$, where $k-2\le i\le k$ and $0\le j\le 2$, the edges $e'=P_{k-1,0}(t)P_{k-1,1}(t)$ and $e=P_{k-1,0}P_{k-1,1}$. The latter edges are indeed parallel and have close lengths by construction. Consider the deformation $P_{ij}(t)$, where $k-2\le i\le k$ and $0\le j\le 2$, of the $2\times2$ sub-net given by the lemma.
    Then the vertices $P_{k,0}(t),P_{k,1}(t),P_{k,2}(t)$ are the desired ones. Note that the positions of the vertices $P_{k-1,2}(t),P_{k-2,0}(t),P_{k-2,1}(t),P_{k-2,2}(t)$ in the deformation of the $2\times2$ sub-net coincide with the ones constructed before because of Lemma~\ref{lemma-uniq-def-1x1}.
    
    Finally, assume that the points $P_{k,0}(t),\ldots,P_{k,l}(t)$ have already been defined for some $2\leq l\leq n-1$. Let us define the point $P_{k,l+1}(t)$. Apply Lemma~\ref{lemma-uniq-def-2x2} to the $2\times2$ sub-net formed by the vertices $P_{ij}$, where $k-2\le i\le k$ and $l-1\le j\le l+1$, the edges $e'=P_{k-1,l-1}(t)P_{k-1,l}(t)$ and $e=P_{k-1,l-1}P_{k-1,l}$. The resulting
    deformation of the $2\times2$ sub-net gives the desired point $P_{k,l+1}(t)$. Note that the positions of the other vertices in the deformation of the sub-net coincide with the ones constructed before because of Lemma~\ref{lemma-uniq-def-1x1}. Now the lemma follows by induction.
\end{proof}

\begin{remark}
    Note that this lemma was not simple at all: to ``propagate'' a deformation to %further and further 
    all $2\times 2$ sub-nets, we needed to ensure that \emph{for any edge, a $2\times 2$ sub-net has both a deformation increasing its length and a deformation decreasing it} (the `moreover' part of Example~\ref{ex-ii}). We do not know how to prove this without proving the whole Classification Theorem~\ref{th-classification}.

    The same concerns %issue appears in the proof of an analogous assertion for 
    flexible nets in Euclidean geometry \cite[Theorem~3.2]{Schief2008}: to ``propagate'' a deformation to all $3\times 3$ sub-nets, we need to ensure that \emph{for any edge, a $3\times 3$ sub-net has both a flexion increasing the dihedral angle at the edge and a flexion decreasing it}. Otherwise, a priori, one flexible $3\times 3$ sub-net of a $3\times 4$ net could have no flexions increasing a particular dihedral angle, and the other flexible $3\times 3$ sub-net could have no flexions decreasing it; thus the whole $3\times 4$ net would not be flexible. This has been overlooked before although a similar phenomenon %indeed 
    occurs in the smooth setup; see, e.g., \cite[Figure~18]{Izmestiev-etal-23}. %We are not aware of a proof of 
    We conjecture that 
    the analog of Lemma~\ref{lemma-def-2x2} for Euclidean flexible nets is \emph{not} true. % does not hold without additional assumptions. 
    This shows again that the isotropic analog of a Euclidean problem gives insight for the latter.
\end{remark}

For the `only if' part of Theorem~\ref{th-mxn}, we need the following construction and lemmas.

To a $m \times n$ net $P_{ij}$ assign two $m \times n$ tables $H$ and $V$ filled by real numbers as follows. The cells of the tables are in the obvious bijection with the faces of the net. We fill the cells (faces) as follows.
    
Into a face $P_{ij}P_{i-1,j}P_{i-1,j-1}P_{i,j-1}$ of the table $H$, where $1\le i\le m$ and $1\le j\le n$, put the opposite ratio of the face with respect to the side $P_{ij}P_{i,j-1}$ (respectively, $P_{i-1,j}P_{i-1,j-1}$) if $i$ is even (respectively, odd).
Into the same face of the table $V$, put the opposite ratio with respect to the side $P_{ij}P_{i-1,j}$ (respectively, $P_{i,j-1}P_{i-1,j-1}$) if $j$ is even (respectively, odd).

Clearly, conditions~(i)--(ii) of Theorem~\ref{th-mxn} can then be restated as follows.

\begin{lemma} \label{l-tables-opposite-ratios}
    Conditions~(i)--(ii) of Theorem~\ref{th-mxn} are equivalent to the following ones:

    (i) both tables $H$ and $V$ have equal rows or both have equal columns;
    
    (ii) table $H$ has equal rows and table $V$ has equal columns.
\end{lemma}

For a deformable $2\times 2$ net, one of the resulting conditions holds by Theorem~\ref{th-classification}. The same concerns $2\times 2$ sub-nets of a deformable $m\times n$ net. We come to the following simple lemma.

\begin{lemma}
    \label{lemma-table-mxn}
    An $m\times n$ table is filled by real numbers. Assume that each square $2\times2$ consists of two equal rows or two equal columns. Then all rows or all columns of the table are equal.
\end{lemma}
\begin{proof}[Proof] Assume that $m,n\geq 2$; otherwise there is nothing to prove.
    We prove the lemma by induction on $m+n$. The induction base, $m+n=4$, i.e., $m=n=2$, is automatic. To perform the induction step, assume that $m+n\ge 5$ and the lemma holds for any $(m-1)\times n$ or $m\times (n-1)$ table, and let us prove it for an $m\times n$ table. Consider the following two cases:

    Case 1: some two neighboring non-corner cells of the  $m\times n$ table contain unequal numbers. Without loss of generality, the two cells belong to one row. They are both contained in either two different $(m-1)\times n$ tables or two different $m\times (n-1)$ tables. Then by the inductive hypothesis, each of the two tables has equal rows or equal columns. Since the two cells belong to one row and contain different numbers, it follows that the rows are equal. Since the two tables have either a common row or a common column, the rows in the whole $m\times n$ table are equal as well.
    
    Case 2: any two neighboring non-corner cells contain equal numbers. Then consider a $2\times2$ corner table. If $m,n\geq3$ then the numbers in its cells must be equal, leading to the table with equal numbers. If, say, $m=2$ then the $2\times2$ table has equal rows, leading to equal rows of the $m\times n$ table. % m is the number of rows by the convention for an m x n matrix
\end{proof}

Let us summarize the argument.

\begin{proof}[Proof of Theorem~\ref{th-mxn}]
The 'if' part follows from Theorem~\ref{th-classification} and Lemma~\ref{lemma-def-2x2}. The 'only if' part follows from Theorem~\ref{th-classification} and Lemmas~\ref{l-tables-opposite-ratios}--\ref{lemma-table-mxn}. %because $2\times 2$ sub-nets of a deformable net are deformable.  
\end{proof}

\begin{proof}[Proof of Corollary~\ref{cor-L}]
Given %the entries in 
the first row and the first column of tables $H$ and $V$, any of conditions~(i)--(ii) in Lemma~\ref{l-tables-opposite-ratios} uniquely determines all the other entries. By Theorem~\ref{th-mxn}, this means that the $L$-shaped $m\times n$ net uniquely determines %all 
the opposite ratios of all the other faces of the deformable $m\times n$ net. It remains to show that those faces are uniquely determined by their opposite~ratios.

We define inductively the points $P_{ij}$ where $2\le i\le m$, $2\le j\le n$, repeatedly applying Corollary~\ref{cor-existence-uniqueness-opposite-ratios} to all faces; see Figure~\ref{figure:net-for_S_i}.    
Assume that the points $P_{i-1,j-1}$, $P_{i-1,j}$, $P_{i,j-1}$ have already been determined. The points are not collinear if the $L$-shaped $m\times n$ net is %sufficiently 
close to the $L$-shaped square net. By Corollary~\ref{cor-existence-uniqueness-opposite-ratios}, there is a unique point $P_{ij}$ such that $P_{i-1,j-1}P_{i-1,j}P_{ij}P_{i,j-1}$ has the desired opposite ratios with respect to $P_{i-1,j-1}P_{i-1,j}$ and $P_{i-1,j-1}P_{i,j-1}$. By induction, the corollary follows.
\end{proof}

%%%%%%%%%%%%%%%%%%%%%%%%%%%%%%%%%%%%%%%%%%%%%%%%%%%%%%5
\section{Smooth deformable nets} %Continuum limit} %Smooth deformable nets
\label{sec-smooth-def}
%%%%%%%%%%%%%%%%%%%%%%%%%%%%%%%%%%%%%%%%%%%%%%%%%%%%%%%%%%

In the previous section, we described all deformable discrete nets.
We now proceed to smooth deformable nets and do this only by 
finding the smooth analogs of the above classes. %discrete versions. %We first discuss the smooth limit informally and then make a few precise assertions.

%\subsection{Examples} 
\subsection{Cone-cylinder nets}
%Informally, a smooth limit means a smooth surface ``approximated'' by a sequence of $m\times n$ nets with $m$ and $n$ increasing but the edge lengths becoming smaller and smaller.  

Discrete deformable nets of class (i) are
combinations of two interleaved discrete cone-cylinder nets.
The smooth analog of both must be the same smooth surface. Passing from the discrete variables $i$ and $j$ to continuous variables $u$ and $v$ in~\eqref{eq-p-cone-cylinder-nets}, we conclude that the smooth %deformable nets 
analog of class (i) is exactly the \emph{smooth cone-cylinder
nets}, that is, the ones that possess a parameterization of the form %\HP{would be good to have its discrete version in Section 3, derived in the obvious purely geometric way}
\begin{equation}
 f(u,v)=a(u) + \sigma(u) b(v), \quad  (u,v)\in U=[\alpha,\beta]\times [\gamma,\delta]\subset\mathbb{R}^2, \label{eq:cone-cylinder}
\end{equation} 
for some smooth functions $a\colon [\alpha,\beta]\to \mathbb{R}^d$, $b\colon [\gamma,\delta]\to \mathbb{R}^d$, $\sigma \colon [\alpha,\beta]\to \mathbb{R}$.
These surfaces have been used in architectural design \cite{glymphetal}, where they are called \emph{scale-translational surfaces} with 
base curves $a(u)$ and $b(v)$ and scaling function $\sigma(u)$. 
Indeed, curve $b(v)$ gets scaled with $\sigma(u)$. Without scaling ($\sigma(u) = 1$), one obtains ordinary translational surfaces.

%Let us provide a few more details. 
%
%Any u-parameter line $f(u,v_1)$ (with constant $v_1$) arises from the curve $b(u)$ by uniform scaling with factor $\sigma(v_1)$ and subsequent translation by $a(v_1)$. Hence, two $u$-parameter lines $f(u,v_1)$ and $f(u,v_2)$ are related by a central similarity or a translation in case of $\sigma(v_2)=\sigma(v_1)$. Thus,  they are connected by a cone, whose vertex $c_{12}$ is easily seen to be 
%
%$$ c_{12}=\frac{1}{\sigma(v_2)-\sigma(v_1)}(\sigma(v_2)a(v_1)-\sigma(v_1)a(v_2)). $$ 
%
%It becomes a cylinder with ruling direction $a(v_1)-a(v_2)$ for $\sigma(v_2)=\sigma(v_1)$. As a limit case ($v_2 \to v_1$), the tangents to the $v$-parameter lines at points of a $u$-parameter line $f(u,v_1)$ form a cone with vertex $\sigma'(v_1)a(v_1)-\sigma(v_1)a'(v_1)$, which becomes a cylinder parallel to the vector $a'(v_1)$ for $\sigma'(v_1)=0$.
%
%Any two $v$-parameter lines, say $f(u_1,v)$ and $f(u_2,v)$ are connected by a cylinder surface with rulings parallel to $b(u_2)-b(u_1)$, since $f(u_2,v)-f(u_1,v)=\sigma(v)(b(u_2)-b(u_1))$. In particular, transversal tangent vectors $f_u(u_1,v)=\sigma(v)b'(u_1)$ along any $v$-curve $f(u_1,v)$ are parallel and form a cylinder. 
%%%%%%%%%%%%%%%%%%%%%%%%%%%%%%%
In the following, we assume a parameter rectangle $U=[\alpha,\beta]\times [\gamma,\delta]\subset\mathbb{R}^2$ on which parameterizations $f(u,v)$ are regular.
This means that %requires 
$f_u(u,v)\nparallel f_v(u,v)$ for each $(u,v)\in U$ or, in our case,
%equivalently, 
$ a'(u)+\sigma'(u)b(v)\nparallel \sigma(u)b'(v)$. In particular, $\sigma(u) \ne 0$ everywhere. Assume without loss of generality that $\sigma(u)>0$, otherwise change the sign of both $\sigma$ and $b$.
We see that $f(u,v)$ is a \emph{conjugate net}, which means that %since 
the mixed partial derivative $f_{uv}$ %=\sigma'(v)b'(u)$ 
at each point is parallel to the tangent plane; here $f_{uv}=\sigma'(u)b'(v)$ is even parallel to $f_v=\sigma(u)b'(v)$.

%%%%% INSERTED PROPOSITION %%%%%%%%%%
\begin{proposition}
    \label{th-deform-surf-i} %\mscomm{change! only one case should remain!}
    A %regular 
    conjugate net $f\colon U\rightarrow \mathbb R^d$ has form~\eqref{eq:cone-cylinder} with $\sigma(u)>0$ if and only if 
    the tangents to the $u$-parameter lines at points of each $v$-parameter line are concurrent or parallel %form a cone or a cylinder 
    and the tangents to the $v$-parameter lines at points of each $u$-parameter line are parallel. %form a cylinder. 
\end{proposition}

\begin{proof}%[Proof of Proposition~\ref{th-deform-surf-i}]
%\mscomm{Mikhail will change the notation and shorten the proof.}
Since $f$ is a conjugate net, it follows that $f_{uv} = pf_u +qf_v$ for some $p,q\colon U\to\mathbb{R}$. Then \cite[Lemma~2]{kilian2023smooth} asserts that the conditions on the $u$- and $v$-parameter lines in the proposition are equivalent to $pq = q_v$ and $p=0$ respectively. Those are satisfied for net~\eqref{eq:cone-cylinder} because it has $f_{uv} = f_v\sigma'(u)/\sigma(u)$. 

Conversely, if $pq = q_v$ and $p=0$, then $q=q(u)$ does not depend on $v$. We get $f_{uv}=q(u)f_v$. Integrating with respect to $v$, we get $f_u=q(u)f+c(u)$ for some function $c\colon[\alpha,\beta]\to \mathbb R^d$. Let $\sigma\colon[\alpha,\beta]\to \mathbb R$ be any positive solution of $\sigma'/\sigma=q$. Then we can write the equation $f_u=q(u)f+c(u)$ in the form $\left(f/\sigma\right)'_u=c/\sigma$. Again, integrating with respect to $u$ and multiplying both sides by $\sigma$, we arrive at $f(u,v)=a(u)+\sigma(u)\cdot b(v)$ for some $a\colon [\alpha,\beta]\rightarrow \mathbb R^d$ and $b\colon [\gamma,\delta]\rightarrow \mathbb R^d$.
%,\sigma\colon [\alpha,\beta]\rightarrow \mathbb R$. 
\end{proof}

%%%%%%%

We found more than just the fact that $f(u,v)$ in equation (\ref{eq:cone-cylinder}) represents
a smooth cone-cylinder net. We see that
\emph{sampling the parameter
intervals for $a(u)$ and $b(v)$ and evaluating $f(u,v)$ on
the resulting grid yields a discrete $m\times n$ cone-cylinder net}. 
We have here an instant of a multi-Q-net in the sense of
Bobenko et al. \cite{multinets-2018}.

We now proceed to deformations (area-preserving C-trafos) of
surfaces (\ref{eq:cone-cylinder}) and provide an analytical
proof of deformability. 

%%%%%%%%%%%%%%%%%%%%%%%%
% THE FOLLOWING PARAGRAPH MOVED UP
%%%%%%%%%%%%%%%%%%%%%%%%
%In the following, we assume a parameter rectangle $U=[\alpha,\beta]\times [\gamma,\delta]\subset\mathbb{R}^2$ on which parameterizations $f(u,v)$ are regular.
%This means that %requires 
%$f_u(u,v)\nparallel f_v(u,v)$ for each $(u,v)\in U$ or
%equivalently, $\sigma(v)b'(u) \nparallel a'(v)+\sigma'(v)b(u)$. In particular, $\sigma(v) \ne 0$ everywhere. Assume without loss of generality that $\sigma(v)>0$, otherwise change the sign of both $\sigma$ and $b$.
%We see that $f(u,v)$ is a \emph{conjugate net}, which means that %since 
%the mixed partial derivative $f_{uv}$ %=\sigma'(v)b'(u)$ 
%at each point is parallel to the tangent plane; here $f_{uv}=\sigma'(v)b'(u)$ is even parallel to $f_u$. 
%%%%%%%%%%%%%%%%%%%%%%%%

Two conjugate nets $f,f^+\colon U\to \mathbb R^d$ are \emph{parallel} or 
 \emph{Combescure transforms} of each other if $f_u(u,v)\parallel f_u^+(u,v)$ and $f_v(u,v)\parallel f_v^+(u,v)$ for each $(u,v)\in U$. 
A Combescure transform $f^+$ of a conjugate net $f$ is \emph{area-preserving}, if determinants of the first fundamental forms agree,
$$\langle f_u,f_u\rangle\cdot\langle f_v,f_v\rangle-\langle f_u,f_v\rangle^2=\langle f^+_u,f^+_u\rangle\cdot\langle f^+_v,f^+_v\rangle-\langle f^+_u,f^+_v\rangle^2,$$ 
at each point $(u,v)\in U$, where $\langle x, y \rangle$ denotes the scalar product of vectors $x$ and $y$.
Two conjugate nets $f,f^+\colon U\to \mathbb R^d$ are \emph{congruent}, if $f^+=g\circ f$ for some isometry $g\colon \mathbb R^d\to \mathbb R^d$.

A conjugate net $f(u,v)$ is called \emph{deformable} if it belongs to a continuous family of non-congruent area-preserving Combescure transforms $f^+(u,v,t)$, where $t\in [0,1]$.

It would be interesting to find all smooth deformable nets. We have the following partial result.

\begin{proposition}
    \label{th-deform-surf}
    Any regular cone-cylinder net~%f(u,v) in 
    \eqref{eq:cone-cylinder}
    is deformable. For $\sigma(u)>0$, it is embedded into a one-parameter family
\begin{equation} \label{eq:cont-family}
f^+(u,v,t):=
    \underbrace{a(\alpha)+\int_\alpha^u \frac{a'(w)\sigma(w)\,dw}{\sqrt{t+\sigma(w)^2}}}_{a(u,t)}+\underbrace{\sqrt{t+\sigma(u)^2}
    \vphantom{\frac{a'(w)\sigma(w)\,dw}{\sqrt{t+\sigma(w)^2}}}}_{\sigma(u,t)} b(v),\quad\text{for}\ t\in [0,1],\ %\xi=const.\in (c,d),  
    (u,v)\in U,
\end{equation} 
    of cone-cylinder nets which are related to each other by
    area-preserving Combescure transformations.
\end{proposition}

This is a particular case of 
a \emph{conical Combescure transformation} introduced in~\cite[Definition~4]{kilian2023smooth}.

\begin{proof}
    %We have $\sigma \ne 0$ everywhere. Assume without loss of generality that $\sigma >0$, otherwise change the sign of both $\sigma$ and $b$. 
% Since $f_{uv}=B'C'=\frac{B'}{B}f_u+0\cdot f_v$, it follows that $f$ is a conjugate net. 
To prove that $f$ is deformable, the continuous family (\ref{eq:cont-family}) has to consist of 
    non-congruent area-preserving Combescure transforms $f^+\colon U\times [0,1]\rightarrow \mathbb R^d$. 
    %Let $X(v,t)$ be any solution of $X_v=A'B/\sqrt{t+B^2}$ and denote $A(v,t)=X(v,t)-X(v,0)+A(v)$. 
  %  Fix any $\xi\in(c,d)$ and set 
  %  $$f^+(u,v,t):=
 %   \underbrace{\int_\xi^v \frac{A'(w)B(w)\,dw}{\sqrt{t+B(w)^2}}+A(\xi)}_{A(v,t)}+\underbrace{\sqrt{t+B(v)^2}
  %  \vphantom{\frac{A'(w)B(w)\,dw}{\sqrt{t+B(w)^2}}}}_{B(v,t)}\cdot C(u),\qquad\text{for each } t\in [0,1],\ (u,v)\in U.$$

    Since $f^+(u,v,0)=f(u,v)$, the family $f^+$ contains $f$. %Since the integral of a continuous function or the sum of two continuous functions is a continuous function, $f^+(u,v,t)$ is continuous for each $t\in[0,1]$. 
    Clearly, $f^+(u,v,t)$ is continuous.
    %It is a smooth parametrized surface for each $t\in[0,1]$ because $\sqrt{t+B^2}\cdot C'=f^+_u\nparallel f^+_v=\frac{B}{\sqrt{t+B^2}}(A'+B'C)$. It is a conjugate net for each $t\in[0,1]$ because it has form 
    %$f^+(u,v,t)=A(v,t)+B(v,t)C(u)$. 
    Since 
    $$\sigma\cdot b'=f_v(u,v)\parallel f^+_v(u,v,t)=\sqrt{t+\sigma^2}\cdot b' \quad\text{and} \quad a'+\sigma'b=f_u(u,v)\parallel f^+_u(u,v,t)=\frac{\sigma(a'+\sigma'b)}{\sqrt{t+\sigma^2}},$$
    for each $t\in[0,1]$ and $(u,v)\in U$, the nets $f$ and $f^+$ are Combescure transforms. In particular, 
    $f^+(u,v,t)$ is regular for each $t\in[0,1]$.  
    It is a cone-cylinder net for each $t\in[0,1]$ because it is of the form 
    $f^+(u,v,t)=a(u,t)+\sigma(u,t)b(v)$. 
    Area preservation follows from
    $$\langle f^+_u,f^+_u\rangle \langle f^+_v,f^+_v\rangle-\langle f^+_u,f^+_v\rangle^2=\sigma^2 (\|a'+ \sigma'b\|^2\ \|b'\|^2-\langle b',a'+\sigma'b\rangle^2)=\langle f_u,f_u\rangle \langle f_v,f_v\rangle-\langle f_u,f_v\rangle^2.$$
    For distinct $t\in[0,1]$, the nets $f^+(u,v,t)$ are non-congruent because $f^+_v=\sqrt{t+\sigma^2}\cdot b'$ are distinct by the regularity condition $b'\ne 0$.
    %Now we prove that for any two distinct $t_1,t_2\in[0,1]$, the surfaces $f^+(u,v,t_1)$ and $f^+(u,v,t_2)$ are distinct, where $(u,v)\in U$. By contrary assume that $f^+(u,v,t_1)=f^+(u,v,t_2)$ for all $(u,v)\in U$. Let us fix $v$ and write this equation in the form,
    %$$A(v,t_1)-A(v,t_2)=(\sqrt{t_1+B(v)}-\sqrt{t_2+B(v)})C(u).$$
    %Since $\sqrt{t_1+B(v)}\not=\sqrt{t_2+B(v)}$, we get that $C(u)$ is constant function for all $u\in(a,b)$ which contradicts the condition $A'(v)+B'(v)C(u)\nparallel B(v)C'(u)$ because of $C'(u)=0$.
    %Since $F_u=\sqrt{t+B^2}C'$, $F_v=\frac{B}{\sqrt{t+B^2}}(A'+B'C)$ and $A'+B'C\nparallel BC'$, we get that $F_u(u,v,t)\nparallel F_v(u,v,t)$ for each $t\in[0,1]$ and all $(u,v)\in U$. This means that for each fixed $t\in[0,1]$, $F(u,v,t)$ where $(u,v)\in U$ is defomable net of type (i) by Proposition~\ref{th-deform-surf-i}. 
    %
    %By Proposition~\ref{th-deform-surf-i}, $f$ is deformable net of type (i) which means that it is a conjugate net. By Proposition~\ref{prop-form-def}, we get that $f$ is deformable. 
\end{proof}

As for discrete deformable nets of class (ii), they do not lead to new smooth examples. Their smooth analog is smooth nets having an area-preserving Chrostoffel dual. Those are translational nets, a particular case of~\eqref{eq:cone-cylinder} with $\sigma(u)=1$. Indeed, recall that a \emph{Christoffel dual} $f^*\colon U\to\mathbb{R}^d$ of a smooth net $f\colon U\to\mathbb{R}^d$ is defined by the conditions $f^*_u=f_u/\nu^2$ and $f^*_v=-f_v/\nu^2$ for some smooth function $\nu\colon U\to (0,+\infty)$.
The transform $f^*$ is area-preserving if and only if $\nu=1$. We get $(f+f^*)'_v=0$, $(f-f^*)'_u=0$, and arrive at~\eqref{eq:cone-cylinder} with $\sigma(u)=1$.
%%%%%
%This can already be concluded from the zig-zag phenomenon discussed in Section~\ref{ssec-properties-deformable-mxn}, which occurs unless the discrete parameter lines have parallel edges. Since a discrete parameter line ``approaching'' a smooth curve cannot have zig-zags, it follows that a smooth limit can only be a translational surface, a particular case of~\eqref{eq:cone-cylinder} with $\sigma(u)=1$. More accurate arguments lead to the same conclusion. 
%%%%%

\begin{remark} \label{rem-gaussian} Area-preserving Combescure transforms of smooth nets preserve the Gaussian curvature, the isotopic Gaussian curvature, and more generally, the relative Gaussian curvature with respect to any relative sphere. Indeed, since the directions of the tangent planes are preserved, it follows that the Gaussian image is preserved, hence its area. The area of the deforming surface does not change as well. The Gaussian curvature is a limit of the ratio of these areas and is thus preserved.  
%
%\mscomm{MS: Move the next paragraph to the next section in some form or just remove it?} Applying the isotropic metric duality, we get \emph{isotropic isometric deformations} in the sense of~\cite{isometric-isotropic}, i.e., deformations of smooth surfaces preserving both isotropic distances and the isotropic Gaussian curvature. Indeed, the isotropic Gaussian curvatures of two isotropic metric dual surfaces are mutually inverse \cite{Strubecker78}, hence both are preserved. Since the directions of the tangent planes are preserved, the points of the dual surface move in the isotropic directions, hence the isotropic distances are preserved. 
\end{remark}

%\HP{We could add another subsection with geometric remarks.
%
%(a) We show that area preserving C-trafos of smooth nets preserve Gaussian curvature with respect to any relative sphere. This includes the Euclidean and isotropic case (the latter implies that isotropic metric duality yields flexible nets). For a proof note that the relative Gauss image remains unchanged under a C-trafo and areas of deforming surfaces do not change as well. Gauss curvature is a limit of the ratio of these areas for contraction to a point, and thus preserved. 
%
%(b) We show that orthogonal deformable nets are exactly the principal
%curvature nets on cylinders,
%cones and rotational surfaces. We can use results of Darboux and 
%the cone-net paper, but have to show that application of M\"obius
%trafos different from Euclidean similarities to these three basic
%types will destroy the cylinder family in the net. I guess this is also true for the
%discrete setting (using the cone-net paper). 
%}
%
%\mscomm{Do you think we should better state this as an open problem? A draft proof is available, but it is not very elegant. This result does not seem obligatory: orthogonal nets have never been discussed in the paper before, and we only describe orthogonal \emph{cone-cylinder} nets, not orthogonal \emph{deformable} nets anyway.} \HP{Mikhail, if you have a proof, it is not an open problem
%anymore. You may state it in a remark, but don't write down the proof. 
%Maybe just write something like "It can be shown that ...", but reference the work of
%Darboux and the cone net paper.}

\subsection{Principal cone-cylinder nets}

Let us give a classification of smooth \emph{principal} (i.e., orthogonal) cone-cylinder nets. Hereafter we use some terminology from differential geometry; see e.g.~\cite{kilian2023smooth}.

\begin{proposition} %\mscomm{State as open problem?} 
Smooth orthogonal cone-cylinder nets without umbilics are exactly the principal curvature nets on cylinders,
cones, and rotational surfaces without umbilics.    
\end{proposition}

\begin{proof}
    Clearly, such principal curvature nets %on cylinders, cones, and rotational surfaces without umbilics 
    are smooth orthogonal cone-cylinder nets.

    Conversely, a smooth orthogonal cone-cylinder net without umbilics %We may assume that along each $u$-parameter curve, the tangents to all $v$-parameter curves form a cylinder (\emph{cylinder property}).
    %, and along each $v$-parameter curve we get a cone or a cylinder (\emph{cone property}). 
    is a particular case of a principal double cone-net, %\cite[Definition~5]{kilian2023smooth}, 
    and by a theorem by Darboux \cite[Theorem~16]{kilian2023smooth}, is a M\"obius transformation of the principal curvature net on a cylinder, a cone, or a rotational surface. It remains to show that the M\"obius transformation leads to a net of the same type, once the image has the following \emph{cylinder property}: along each $u$-parameter line, the tangents to all $v$-parameter line form a cylinder. 

    We may assume that the M\"obius transformation is an inversion~$I$. Indeed, any 
    M\"obius transformation is either a similarity or a composition of an inversion and a similarity, and all our nets are invariant under a similarity.
    
    %Assume that $F$ is not a similarity, i.e. does not preserve the infinity, otherwise there is nothing to prove. 
    We shall see that for the listed types of nets, each $u$-parameter line lies on a plane or a sphere $S$ orthogonal to the $v$-parameter lines. The cylinder property implies that $I(S)$ is a plane rather than a sphere, otherwise the tangents to the $v$-parameter lines would intersect at the center of the sphere. Thus the inversion center $O$ 
    %  the preimage $F^{-1}(\infty)$ of the infinity 
    should belong to the intersection of all such planes or spheres $S$. We get a strong restriction in each of the possible cases:

    Case 1: $u$-parameter lines are profiles of a rotational surface. In this case, $S$ is a plane passing through the rotation axis. Hence $O$ belongs to the axis, and
    %, and $F$ is a composition of an inversion centered at $F^{-1}(\infty)$ and a similarity. 
    the image is still %a principle curvature net on 
    a rotational surface. 

    Case 2: $u$-parameter lines are parallel circles on a rotational surface. In this case, $S$ is a sphere or a plane symmetric with respect to the axis of revolution. Such spheres can have a common point only if they have a common circle or all are tangent at one point. In either case, the images of the profiles are orthogonal to a pencil of planes, and we get still a rotational surface. 
    
    Case 3: $u$-parameter lines are rulings of a non-rotational cylinder. Then $S$ is a plane orthogonal to the profiles (the curves orthogonal to the rulings). The intersection of all such planes $S$ is empty, 
    otherwise, the planes form a pencil, the profiles are circles and we have a rotational cylinder. %There are no suitable inversions $I$.
    %Hence $F$ is a similarity.

    Case 4: $u$-parameter lines are profiles of a non-rotational cylinder. Then $S$ is a plane orthogonal to rulings. Again, the intersection of all such planes $S$ is empty and there are no suitable inversions $I$. 

    Case 5: $u$-parameter lines are rulings of a non-rotational cone. Then $S$ is a plane orthogonal to the profiles. The intersection of all such planes $S$ is the cone vertex, otherwise, the profiles are circles and we have a rotational cone. Thus $O$ is the cone vertex, and the image is still a cone.

    Case 6: $u$-parameter lines are profiles of a non-rotational cone. Then $S$ is a sphere centered at the cone vertex. The intersection of all such spheres $S$ is empty, which completes the proof.
\end{proof}

It would be interesting to obtain an analogous result in the discrete setup. Cf.~\cite[Theorem~38]{kilian2023smooth}.

%%%%%%%%%%%%%%%%%%%%%%%%%%%%%%%%%%%%%%%%%%%%%%%%%%%%%%%%%%%%%%%%%%%%%%
\section{Conclusion, %remarks on 
flexible %Q-
nets, and
future work} \label{sec:final}
%%%%%%%%%%%%%%%%%%%%%%%%%%%%%%%%%%%%%%%%%%%%%%%%%%%%%%%%%%%%%%%%%%%%%%

Motivated by the design of flexible Q-nets in Euclidean geometry, we first turned to
the isotropic counterpart, and more precisely, to the metric dual in $I^3$. These dual nets are  Q-nets which are deformable in the sense that they
admit a one-parameter family of area preserving Combescure transformations. Using elementary
algebraic and geometric methods, we could
completely classify these nets. They fall into two classes. Class (i) is composed of
two interleaved cone-cylinder nets. Class (ii) is characterized by the existence of a Christoffel dual with equal areas of corresponding faces. %equal opposite area ratios of each pair of quads with a common edge. 
This implies in general a visually non-smooth behavior. As a result of that, the smooth analogs of (ii) are just translational nets and a special
case of the smooth analogs of type (i), which are smooth cone-cylinder nets, also 
 known as scale-translational surfaces.

While the Euclidean classification of flexible Q-nets has so far only been only done for
nets with $3 \times 3$ faces and led to a large number of classes \cite{izmestiev-2007},
the isotropic classification is for arbitrary $m \times n$ nets and has only two classes. 

%Finally, we noted preservation of Gaussian curvature (in any relative
%differential geometry) during a continuous area preserving Combescure transformation.

\subsection{Flexible %Q-
nets in $I^3$}
Let us briefly mention some basic facts on flexible %Q-
conjugate nets in isotropic 3-space $I^3$ (both smooth and discrete).
They are related to the deformable nets studied in this
paper through metric duality. This duality
can be realized by polarity with respect to the \emph{isotropic unit sphere}
$S_i^2: 2z=x^2+y^2$, which is a rotational paraboloid in Euclidean space with Cartesian
coordinates $(x,y,z)$. It maps a point with coordinates $(u, v, w)$ to a plane and
vice versa:
\[
 \delta \colon (u, v, w) \longleftrightarrow u x + v y - z - w = 0.
\]
%Note that a point $p = (u, v, w)$ lies on the plane $\delta(p)$ if and only if
%$p$ lies on the isotropic unit sphere.
Parallel planes $ux+vy-z-w_i=0, \ i=1,2$, correspond to points
$(u,v,w_i)$ with the same projection $(u,v,0)$ onto the plane 
$z=0$ (\emph{top view}). Note that isotropic distances %and angles 
are
measured in the top view. A smooth surface $f$ seen as set of contact elements
$(p,T)$ (points plus tangent planes) gets mapped to a surface $\delta(f)$ as a set
of contact elements $(\delta(p),\delta(T))$ (tangent planes plus contact points). If $K_i$ is the isotropic Gaussian curvature at a 
contact element of $f$, the isotropic Gaussian curvature at the
image element is $1/K_i$ \cite{Strubecker78}. 
For a deformation of $f$, both are preserved by Remark~\ref{rem-gaussian}. 
%Finally, a Q-net corresponds to a Q-net.
Putting all these facts together we see that \emph{a smooth deformable %Q-
net $f$
corresponds under metric duality $\delta$ to a flexible %Q-
net $\delta(f)$}: The deformation keeps the
top view (thus the intrinsic metric) and the isotropic Gaussian
curvature. This characterizes an isometric deformation in $I^3$
(see \cite{isometric-isotropic}). By definition, \emph{a deformable Q-net $f$
corresponds under metric duality $\delta$ to a flexible Q-net $\delta(f)$.} If a Q-net $f$ is deformable, so are
all Q-nets which have the same top view as $f$. Via duality, we obtain the
expected isotropic counterpart of a property of flexible Q-nets in Euclidean $\RR^3$: \emph{If a Q-net
is flexible in $I^3$, so are all nets related to it by Combescure transformations}. It is interesting if the bellows conjecture (Sabitov's theorem, cf.~\cite{Gaifullin-15}) holds in $I^3$.

Let us briefly discuss the dual $\delta(f)$ of a discrete cone-cylinder net $f$.
Any polarity relates lines through a fixed point to lines in a fixed plane. In particular, $\delta$
maps parallel lines to lines with the same top view. 
Thus, a cone-cylinder net corresponds to \emph{a Q-net with planar parameter lines, where one
family of parameter lines lie in isotropic planes (parallel to
the $z$-axis)}. This family corresponds to the cylindrical strips.
The parameter lines of the other family do in general not lie in isotropic
planes.  Only a translational net $f$
corresponds to a Q-net where both families of parameter lines
lie in isotropic planes. These are the isotropic counterparts
to Voss nets (see also \cite{isometric-isotropic}). Note that the
class of Q-nets whose parameter lines lie in planes is invariant under
C-trafos and so is the property of  parameter lines lying in isotropic planes.

Even without the isotropic flexibility, the nets $\delta(f)$ are of interest for applications,
namely in architectural structures. The planarity
of faces (panels) and of parameter lines (long range supporting beams) provides an advantage for construction (see \cite{PP-2022,Cheng-CAD-2023}). If we consider the $z$-axis vertical,
one family of these long range beams lie in vertical planes, 
which can be an advantage as well. An example is shown in 
Figure~\ref{fig:dualstructure}. 

\begin{figure}[htbp]
    \centering
    \includegraphics[scale=0.041]{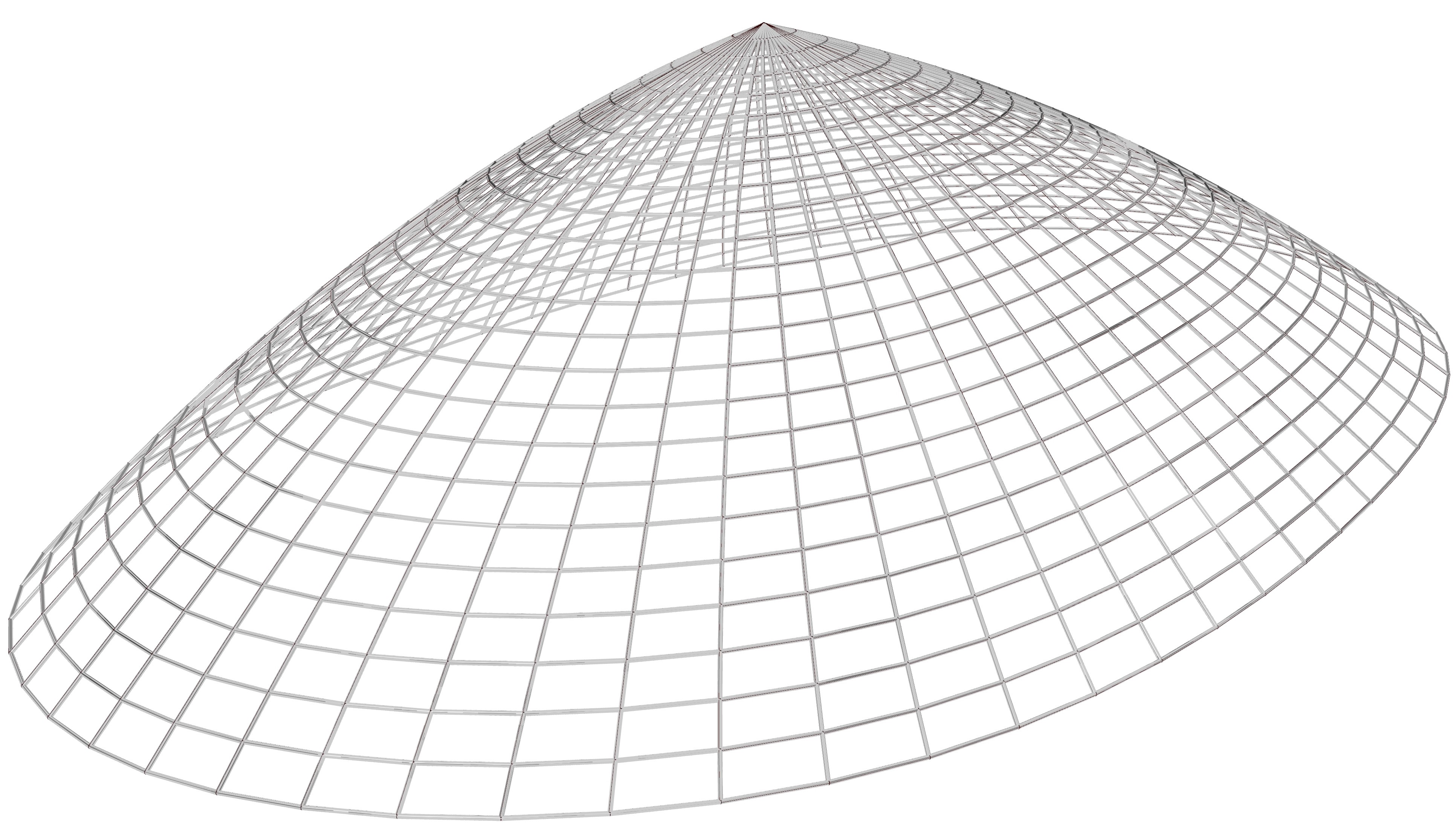}\quad\quad\quad\includegraphics[scale=0.156]{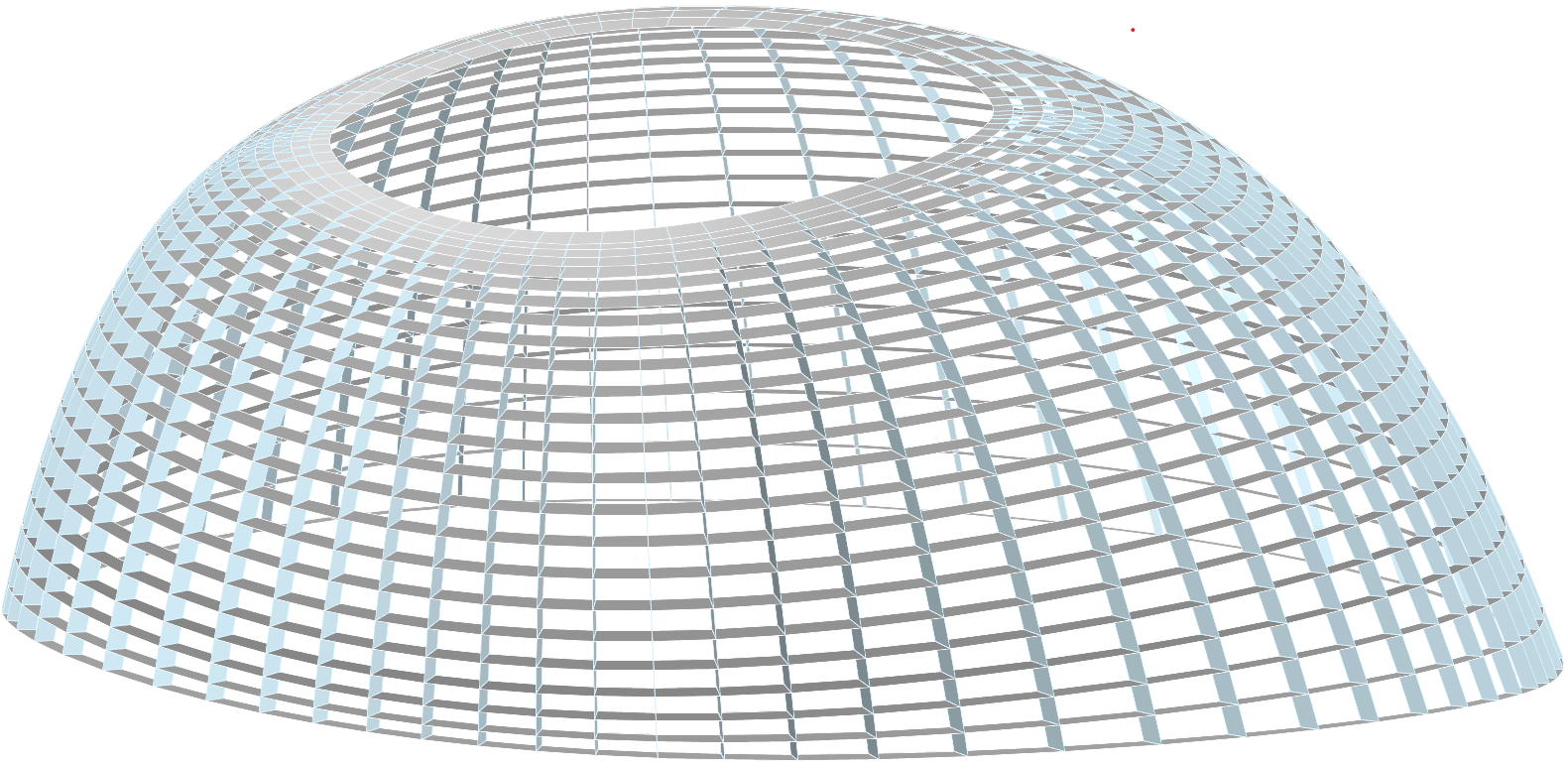}%\qquad\qquad\qquad\includegraphics[scale=0.20]{fig/g12.png}

    \ \ 
    \ \
    \ \
    \caption{A cone-cylinder net (left) and its metric dual (right). The latter is a flexible surface in isotropic geometry having planar discrete parameter lines, one family %of the two families 
    %of parameter lines lies 
    lying in isotropic planes.
    %\HP{To Olimjon: please, show here a discrete cone-cylinder net (left) and the metric dual (larger, on the right), rendered as an architectural structure. The dual should be a nice design.}
    }
    \label{fig:dualstructure}
\end{figure}

\subsection{Future research} In a subsequent paper, we will discuss flexible nets in $I^3$
in detail. As a conclusion from the present study, we know
already that the only nets with a visually smooth appearance
can be those of class (i). However, flexible mechanisms often
have a folding behavior and therefore we also need to address class
(ii). From a geometric perspective, it is interesting to
study the reciprocal parallel nets of flexible nets in $I^3$.
Their Euclidean counterparts (Bianchi surfaces in the smooth setting) have been addressed by Schief et al.~\cite{Schief2008}. On the
applications side, the main focus is on good ways to
control the shapes of flexible Q-nets in $I^3$  and on numerical optimization algorithms for the transition
to flexible Q-nets in Euclidean 3-space. Our initial numerical
experiments show that this transition works for Voss nets, and thus we are optimistic
that the more interesting general types can be handled as well. 
The isotropic flexible Q-nets are more general than just
isotropic counterparts to the
explicitly known Euclidean 
flexible $m \times n$ nets (Voss nets, T-nets and P-nets \cite{Nawratil2024}).  Possibly the simpler isotropic
versions lead to so far unknown explicit constructions
of certain types of Euclidean flexible Q-nets that do not require numerical optimization. 

The idea of initializing a numerical optimization algorithm
for the computation of a Euclidean structure by an isotropic
counterpart goes beyond the context of flexibility
and deserves to be investigated for other difficult problems as well. For instance, see \cite{MOROZOV2021, Yorov-etal}.

\subsubsection*{Acknowledgements}

The authors are grateful to Mekhron Bobokhonov, Ivan Izmestiev, Oleg Karpenkov,  Naimdzhon Khondzhonov, Niels Lubbes, Dmitry Lyakhov, Christian M\"uller, and Florian Rist for useful discussions. %\mscomm{Add more names?}

%%%%%%%%%%%%%%%%%%%%%%%%%%%%%%%%%%%%%%%%%%%%%%%%%%%%%%%%%%%%%%%%
\bibliographystyle{siamplain}
\bibliography{references}

\end{document}